\documentclass[a4paper,10pt,reqno]{amsproc}
\usepackage{amssymb}
\usepackage{graphicx}

\usepackage{amscd}

\usepackage{amssymb}

\usepackage{euscript}
\usepackage{epsfig}

\begin{document}

\theoremstyle{plain} \numberwithin{equation}{section}

\newtheorem{thm}{Theorem}[section]
\newtheorem{prop}[thm]{Proposition}
\newtheorem{cor}[thm]{Corollary}
\newtheorem{lemma}[thm]{Lemma}
\newtheorem{lemdef}[thm]{Lemma--Definition}
\newtheorem{ques}[thm]{Question}
\newtheorem{claim}[thm]{Claim}

\theoremstyle{definition}

\newtheorem{defn}[thm]{Definition}
\newtheorem{ex}[thm]{Example}
\newtheorem{notn}[thm]{Notation} 
\newtheorem{note}[thm]{Note}
\newtheorem{obs}[thm]{Observation}
\newtheorem{rmk}[thm]{Remark}
\newtheorem{Question}[thm]{Question}
\newtheorem{review}[thm]{}
\newtheorem{Empty}[thm]{}
\newcommand{\bi}{\begin{itemize}}
\newcommand{\ei}{\end{itemize}}
\newcommand{\bp}{\begin{proof}}
\newcommand{\ep}{\end{proof}}

\def\Trick#1{\begin{Empty}\bf#1\end{Empty}}

\def\AA{\mathbb{A}}
\def\CC{\mathbb{C}}
\def\FF{\mathbb{F}}
\def\GGG{\mathbb{G}}
\def\PP{\mathbb{P}}
\def\cPP{\check{\mathbb{P}}}

\def\QQ{\mathbb{Q}}
\def\RR{\mathbb{R}}
\def\VVV{\mathbb{V}}
\def\ZZ{\mathbb{Z}}

\def\m{\mathfrak{m}}

\def\cL{\Cal L}

\def\ov{\overline}

\def\al{\alpha}
\def\be{\beta}
\def\de{\delta}
\def\eps{\epsilon}
\def\ga{\gamma}
\def\io{\iota}
\def\ka{\kappa}
\def\la{\lambda}
\def\na{\nabla}
\def\om{\omega}
\def\si{\sigma}
\def\th{\theta}
\def\ups{\upsilon}
\def\ve{\varepsilon}
\def\vp{\varpi}
\def\vt{\vartheta}
\def\ze{\zeta}

\def\De{\Delta}
\def\Ga{\Gamma}

\def\cC{\Cal C}
\def\cU{\Cal U}

\def\PGL{\text{\rm PGL}}
\def\NE{\text{\rm NE}}

\def\dim{\text{\rm dim}}
\def\codim{\text{codim}}
\def\Eff{\text{\rm Eff}}
\def\Exc{\text{\rm Exc}}
\def\h{\text{h}}
\def\p{\mathfrak{p}}
\def\HH{\text{H}}
\def\MM{\overline{\text{M}}}
\def\M{\MM}
\def\NE{\overline{\text{NE}}}
\def\Pic{\text{\rm Pic}}
\def\Proj{\text{\rm Proj}}
\def\Spec{\text{Spec}}
\def\Bl{\text{\rm Bl}}

\def\dra{\dashrightarrow}
\def\hra{\hookrightarrow}
\def\lra{\leftrightarrow}
\def\ra{\rightarrow}


\def\oM{\overline{M}}

\def\bP{\Bbb P}
\def\Hes{\operatorname{Hes}}
\def\Ceva{\operatorname{Ceva}}

\def\bF{\Bbb F}
\def\Aut{\operatorname{Aut}}
\def\cO{\Cal O}
\def\lie{\operatorname{Lie}}
\def\Det{\operatorname{Det}}
\def\codim{\operatorname{codim}}
\def\rank{\operatorname{rank}}
\def\Perm{\operatorname{Perm}}
\def\mult{\operatorname{mult}}
\def\bZ{\Bbb Z}
\def\bC{\Bbb C}
\def\bQ{\Bbb Q}
\def\bG{\Bbb G}
\def\bR{\Bbb R}
\def\bA{\Bbb A}
\def\bB{\Bbb B}
\def\bX{\Bbb X}
\def\oX{\overline{X}}
\def\oW{\overline{W}}

\def\oY{\overline{Y}}
\def\oYn{{\overline{Y^n}}}
\def\obP{\bP^{\circ}}
\def\oF{\overline{F}}
\def\cB{\Cal B}
\def\cA{\Cal A}
\def\cD{\Cal D}
\def\cS{\Cal S}
\def\cE{\Cal E}
\def\cF{\Cal F}

\def\tS{\tilde S}
\def\cI{\Cal I}
\def\cA{\Cal A}
\def\tcA{\tilde{\cA}}
\def\tpi{\tilde{\pi}}
\def\tp{\tilde{p}}
\def\tA{\tilde{A}}
\def\tX{\tilde{X}}
\def\tbP{\tilde{\bP}}
\def\tP{\tilde{P}}

\def\uI{\underline{I}}
\def\uA{\underline{A}}
\def\ucI{\underline{\cI}}
\def\uP{{\underline{P}}}
\def\uS{\underline{S}}
\def\uB{\underline{B}}

\def\uW{\underline{W}}
\def\oS{\overline{S}}
\def\cT{\Cal T}
\def\cP{\Cal P}
\def\cB{\Cal B}
\def\tS{\tilde{S}}
\def\qp{\Cal P}
\def\cW{\Cal W}
\def\tqp{\tilde{\qp}}
\def\tB{\tilde{B}}
\def\tcB{\tilde{\cB}}
\def\oN{\overline{N}}
\def\INC{\operatorname{INC}}
\def\tcA{\tilde{\cA}}
\def\cB{\Cal B}
\def\cY{\Cal Y}
\def\ocY{\overline{\cY}}
\def\cD{\Cal D}
\def\cX{\Cal X}
\def\cH{\Cal H}
\def\cM{\Cal M}
\def\cZ{\Cal Z}
\def\ocZ{\overline{\Cal Z}}
\def\ocS{\overline{\Cal S}}
\def\tcS{\tilde{\cS}}
\def\mg{\overline{M}_g}
\def\Spec{\operatorname{Spec}}
\def\Proj{\operatorname{Proj}}
\def\PGL{\operatorname{PGL}}
\def\HH{\operatorname{H}}
\def\tM{\tilde{M}}
\def\lc{\operatorname{lc}}
\def\ss{\operatorname{ss}}
\def\oo{\overline{\omega}}
\def\ls{\oo^S}
\def\tx{\tilde{x}}
\def\utA{\underline{\tilde{A}}}
\def\Stab{\operatorname{Stab}}
\def\aff{\Bbb A}
\def\cE{\Cal E}
\def\czE{\mathop{\cE}\limits^\circ}
\def\bS{\Bbb S}
\def\bB{\Bbb B}
\def\ga{\gamma}
\def\dlog{\operatorname{dlog}}
\def\dlogs{\operatorname{dlogs}}
\def\qP{\Cal P}
\def\obS{\overline{\bS}}
\def\too{\tilde{\omega}}
\def\Res{\operatorname{Res}}
\def\Ord{\operatorname{Ord}}
\def\ord{\operatorname{ord}}
\def\oRes{\overline{\Res}}
\def\ores{\oRes}
\def\st{S}
\def\ocB{\overline{\cB}}
\def\tR{\tilde R}
\def\Star{\operatorname{Star}}
\def\oLambda{\overline{\Lambda}}
\def\olamda{\oLambda}
\def\osigma{\overline{\sigma}}
\def\bl{\operatorname{BL}}
\def\tqP{\tqp}
\def\ogamma{\overline{\gamma}}
\def\tL{\tilde{L}}
\def\bk{\Bbb k}
\def\tE{\tilde E}
\def\cK{\Cal K}
\def\obS{\overline{\bS}}\def\obB{\overline{\bB}}
\def\cU{\Cal U}
\def\cV{\Cal V}
\def\cR{\Cal R}
\def\cO{\Cal O}
\def\bU{\Bbb U}
\def\res{\Res}
\def\oZ{\overline{Z}}
\def\tPsi{\tilde{\Psi}}
\def\oR{\overline{R}}
\def\oK{\overline{K}}
\def\Sing{\operatorname{Sing}}
\def\Sp{\operatorname{Sp}}
\def\rk{\operatorname{rk}}
\def\O{\cO}
\def\Char{\operatorname{char}}
\def\Pic{\operatorname{Pic}}

\def\oPic{\overline{\Pic}}
\def\oPicone{\oPic^{\uone}}
\def\ocPic{\overline{\cPic}}

\def\Spec{\operatorname{Spec}}
\def\Im{\operatorname{Im}}
\def\pr{\operatorname{pr}}
\def\Hom{\operatorname{Hom}}
\def\Cal{\mathcal}
\def\oC{\overline C}
\def\bX{\Bbb X}
\def\oXTX{\overline{X/H}}
\def\oTTX{\overline{T/H}}
\def\ov{\overline}
\def\tX{\tilde X}
\def\tB{\tilde B}
\def\eps{\varepsilon}
\def\hY{\hat Y}
\def\hD{\hat D}
\def\hp{\hat p}

\def\W{\mathop{\circ}}
\def\B{\mathop{\bullet}}

\def\al{\alpha}
\def\be{\beta}
\def\ga{\gamma}

\def\arrow{\mathop{\longrightarrow}\limits}
\def\mysim{\mathop{\sim}\limits}

\def\bcup{\bigcup\limits}

\newtheorem{Review}[thm]{}

\title{Rigid curves on $\oM_{0,n}$ and arithmetic breaks}

\author{Ana-Maria Castravet}
\author{Jenia Tevelev}

\address{Ana-Maria Castravet: \sf Department of Mathematics, The Ohio State University, 100 Math Tower, 231 West 18th Avenue, Columbus, OH 43210-1174} 
\email{noni@alum.mit.edu}

\address{\vskip -.5cm Jenia Tevelev: \sf Department of Mathematics, 
University of Massachusetts at Amherst, Lederle Graduate Research Tower, Amherst, MA 01003-9305} 
\email{tevelev@math.umass.edu}

\thanks{The first author was partially supported  by the NSF grant DMS-1001157.
The second author was partially supported by the NSF grants DMS-0701191, DMS-1001344, 
and the Sloan fellowship.} 


\date{today}

\begin{abstract}
A result of Keel and M\textsuperscript{c}Kernan states that a hypothetical counterexample  to the
$F$-conjecture must come from rigid curves on $\overline{M}_{0,n}$ that intersect the interior. We exhibit several ways of constructing rigid curves. In all our examples, a reduction mod $p$ argument shows that the  classes of the rigid curves that we construct can be decomposed as sums of $F$-curves.

\end{abstract}

\maketitle


\section{Introduction}

Let $\MM_{0,n}$ be the moduli space of stable $n$-pointed rational curves. The one-dimensional boundary strata of the moduli space, i.e., the irreducible components of the locus parameterizing rational curves with at least $n-3$ components are often called \emph{$F$-curves}. A long standing open question \cite{KM} (known as the $F$-conjecture) is whether  the Mori cone of curves $\NE(\MM_{0,n})$ is generated by $F$-curves. Gibney, Keel and Morrison \cite{GKM} proved that the $F$-conjecture for all $n$ implies that the same is true for the moduli spaces  $\MM_{g,n}$ of stable, genus $g$, $n$-pointed curves, namely, that the Mori cone $\NE(\MM_{g,n})$ is generated by one-dimensional boundary strata (thus, giving an explicit description of the ample cone of $\MM_{g,n}$).

Keel and M\textsuperscript{c}Kernan \cite{KM} proved the $F$-conjecture for  $n\leq7$ and proved that a hypothetical counterexample to the $F$-conjecture must come from \emph{rigid curves} intersecting the interior $M_{0,n}$ (see Thm. \ref{KM theorem} for a precise statement).  
The notion of rigidity in the Keel-M\textsuperscript{c}Kernan result is a very strong one:
\begin{defn}\label{rigid def}
Let $C$ be a curve on a variety $X$. We say that $C$ \emph{moves} on $X$ if there is a flat family of curves $\pi:S\ra B$ over a curve germ $(b_0\in B)$, 
with a map $h:S\ra X$ such that $\dim\,h(S)=2$ and $h(S_{b_0})=C$ (set-theoretically).
We~say that $C$ is \emph{rigid} on $X$ if $C$ does not move. 
\end{defn}


\Trick{Constructing rigid curves.}
We observe that if the curve $C$ is an irreducible component of the exceptional locus of a regular map $X\ra Z$ (for some $Z$), then $C$ is rigid on $X$ in the sense of Def. \ref{rigid def}. 
Indeed, this is an immediate application of Mumford's rigidity lemma \cite[p.43]{Mu}.
On $\MM_{0,n}$, the natural maps to consider are products of forgetful maps. 
In Sections \ref{blow-ups section}, \ref{hypergraph section}, and \ref{HesseSection} 
we discuss a construction, which we call the \emph{hypergraph construction}.
The basic idea is that $\MM_{0,n}$ is covered by arbitrary blow-ups of $\bP^2$ in $n$ points
(as long as these points do not belong to a (possibly reducible) conic). The curves that we consider are $(-1)$-curves
in these blow-ups. The hypergraph construction uses a rigid configuration of $n$ points to construct interesting curves and surfaces in $\MM_{0,n}$, that intersect the interior $M_{0,n}$, and are contracted by some natural products of forgetful maps. 
It is in general difficult to decide when the exceptional locus of such a map has a $1$-dimensional irreducible component. We have been able to prove this by an ad-hoc argument in one example. Namely, one starts with the $(9_312_4)$ Hesse configuration
of $9$ inflection points of a non-singular plane cubic  and $12$ lines connecting them pairwise. Applying our hypergraph construction
to the configuration projectively dual to the Hesse configuration, 
one gets in this way a rigid curve on $\MM_{0,12}$. (This construction appeared first in the authors preprint \cite{CT1}.)

\Trick{Constructing rigid maps.}
The notion of rigidity in Def. \ref{rigid def} is much stronger than the one usually used for maps \cite{McM}: 
a map $f:C\ra X$ is called rigid if any family of maps containing $f$ is isotrivial. 
Here a family of maps is a proper flat family of curves $\pi:S\ra B$ with reduced fibers over a curve germ $(b_0\in B)$, with a map $h:\,S\ra X$ such that $h|_{S_{b_0}}=f$.
The family of maps is isotrivial if (after shrinking $B$) it is isomorphic over $B$ to the constant family $C\times B$, $h(c,b)=f(c)$.
If $C$ is a rigid curve on $X$, the embedding map $f:C\ra X$ is a rigid map, but the converse does not hold in general. Indeed, consider a family of quartic plane curves specializing to a double conic. If $C$ denotes the reduced conic on the total space $S$ of the family, then clearly $C$ is not a rigid curve on $S$, but the embedding map $C\hra S$ is rigid. 
Rigid maps $C\ra\MM_{0,n}$ were recently constructed by Chen~\cite{Ch} using results of McMullen \cite{McM} and M\"oller \cite{Mo} on Teichm\"uller curves. An amazing feature of Chen's curves is that their union is dense in $\MM_{0,n}$ for every $n\geq8$. It seems to be a difficult problem to decide whether these curves are rigid in the sense of Def.~\ref{rigid def}. 

In Section~\ref{deJ-K}  we present a different construction of rigid maps
inspired by discussions with J. Koll\'ar and J. de Jong from a few years ago. It uses rigid configurations of lines and conics in the plane. We call this the ``Two Conics'' construction. We give an explicit example of such a curve in Section \ref{break of 2-conics curve} using the configuration of Gr\"unbaum~\cite[5.5]{G}
of $9$ lines in the plane representing the golden ratio.

\Trick{Arithmetic Breaks.} 
We then proceed to show that all the rigid curves (and images of rigid maps) that we found can be decomposed into sums of $F$-curves. This is easy for curves found in \cite{Ch}:
these curves lie in the symmetric Mori cone $\NE(\MM_{0,n}/S_n)$ and their classes are easily seen to be sums of $F$-curves. 

Curves obtained using hypergraph and ``Two conics'' constructions are highly asymmetric,
and it is hard to see how their classes can break into sums of $F$-curves.
However, we have found a way to break not just the class of the curve but the curve itself
using a simple idea that we call an ``arithmetic break''.
The above-mentioned result of Keel and M\textsuperscript{c}Kernan says, roughly, that if a curve on $\oM_{0,n}$ moves in a one-parameter family then it breaks (one of the fibers is reducible). 

We~remark that even a rigid curve $C$ moves in an arithmetic sense. Namely, its field of definition $K$ is a field of algebraic numbers. Let $R\subset K$ be the integral closure of $\bZ$.
Then $C$ has an integral model $C_R$ over $\Spec R$, which is a subscheme in the $R$-moduli scheme $\oM_{0,n; R}$. 
We observe that in all our examples, one of the fibers of $C_R\to\Spec R$ is reducible. We further analyze irreducible components of this fiber (defined over the corresponding finite field), and show that these components move and break, 
and in fact break 
down to effective linear combinations of $F$-curves, thus showing 
that the class of $C$ is also an effective sum of $F$-curves.
Here we use a well-known fact that $\Pic \oM_{0,n}$ is characteristic-independent (this follows
from  the description of $\M_{0,n}$ as a blow-up - Kapranov \cite{Ka}, Keel \cite{Ke} or Knudsen \cite{Kn}). 
This raises an interesting question:
\begin{Question}
Is it possible to construct a rigid curve on $\oM_{0,n}$ 
that intersects the interior such that all its reductions modulo $p$ are irreducible?
\end{Question}
Note that rigidity is important here: it is possible to construct an embedding  
$\PP^1_R\ra\MM_{0,n; R}$ that intersects the interior, even though we know only one example,
which arises from the Gr\"unbaum
configuration (see Section \ref{rigid matroids}) 
using the hypergraph construction. However, the generic fiber of such a map is not a  rigid curve.

\Trick{\bf Structure of the paper.}
In Section \ref{KM argument}, for the reader's convenience we reproduce, with the authors' permission, the Keel-M\textsuperscript{c}Kernan argument (Thm. \ref{KM theorem} does not appear in its current form in \cite{KM}). In  Section \ref{blow-ups section} we give a general construction of surfaces in $\MM_{0,n}$, that intersect the interior $M_{0,n}$, starting with a configuration of points in $\PP^2$. Section \ref{hypergraph section} explains the 
hypergraph construction. In Section \ref{HesseSection} we consider a specific example coming from the Hesse configuration. We find a curve on $\MM_{0,12}$ which is an irreducible component of the exceptional locus of a generically finite map $\MM_{0,12}\ra Z$. In Section \ref{deJ-K}
we present the ``two conics construction". In Section \ref{break of hypergraph curve} we explain how the Hesse curve breaks into several components in positive characteristic. This allows us to write the class of the curve as a sum of $F$-curves. In Section \ref{break of 2-conics curve} we do the same for a curve obtained via the two-conic construction.

\

We work over an algebraically closed field $k$ (in  Sections \ref{KM argument}, \ref{blow-ups section}, \ref{hypergraph section}, \ref{HesseSection} and \ref{deJ-K}), 
unless we specify otherwise (such as in Sections \ref{break of hypergraph curve},  \ref{break of 2-conics curve} and \ref{tangent}).

\Trick{Acknowledgements}
The first author was partially supported  by the NSF grant DMS-1001157.
The second author was partially supported by the NSF grants DMS-0701191, DMS-1001344, 
and the Sloan fellowship. Parts of this paper were written while the first author was visiting the Max-Planck Institute in Bonn, Germany. The authors are grateful to the referee for useful suggestions on how to improve the exposition of the article. 

\tableofcontents


\section{The Keel-M\textsuperscript{c}Kernan theorem}\label{KM argument}

\begin{defn}\label{ConvexEdge}
We say that an extremal ray $R$ of a closed convex cone $\cC \subset \RR^n$
is an {\em edge} if $\cC$ is ``not rounded'' at $R$.
Concretely, the vector space $R^{\perp} \subset (\RR^n)^*$ (of linear forms that vanish
on $R$) should be generated by supporting hyperplanes for $\cC$. 
\end{defn}

\begin{thm}\cite{KM}\label{KM theorem}
Suppose that the Mori cone $\NE(\MM_{0,n})$ has an extremal ray which is an edge and is 
generated by a curve $C \subset \MM_{0,n}$ such that $C\cap M_{0,n}\ne\emptyset$. Then $C$ is rigid.
\end{thm}

\begin{rmk}
Assuming that the Mori cone $\NE(\MM_{0,n})$ is finitely generated, and moreover each extremal ray is generated by a curve, then  if $C$ is a curve that generates an extremal ray of $\NE(\MM_{0,n})$, then either $C$ intersects the interior $M_{0,n}$ and by Theorem  \ref{KM theorem} it must be rigid, or $C$ is contained in a boundary component. In the latter case, as 
\begin{equation}\label{product}
\NE(\MM_{0,k}\times \MM_{0,l})\cong\NE(\MM_{0,k})\times \NE(\MM_{0,l})
\end{equation}
it follows that $C$ is obtained from a curve in $\MM_{0,k}$ for some $k<n$,
by attaching a fixed curve with fixed markings. Moreover, $C$ itself generates an extremal ray of  $\NE(\MM_{0,k})$. It follows that either $C$ is an $F$-curve, or eventually one obtains a counterexample to the $F$-conjecture from a rigid curve $C\subset\MM_{0,n}$ that intersects the interior $M_{0,n}$.
\end{rmk}

For completeness, note that $F$-curves do generate extremal rays of  $\NE(\MM_{0,n})$. This is easily seen by induction, using (\ref{product}). Moreover, for any $F$-curve $F$, we have $F\cdot(K_{\MM_{0,n}}+\De)=1$, where $K_{\MM_{0,n}}$ is the canonical class and $\De$ is the sum of all boundary \cite[Rmk. 3.7 (1)]{KM}. 

\begin{rmk}\label{lower bound}
A \emph{rational} rigid curve $C\subset\MM_{0,n}$ has the property that 
$K_{\MM_{0,n}}\cdot C\geq n-6$. This follows from the usual lower bound for the dimension of the Hom-scheme locally at a point $[f]$, where $f:\PP^1\ra\MM_{0,n}$:
$$\dim_{[f]}\Hom(\PP^1,\MM_{0,n})\geq -K_{\MM_{0,n}}\cdot f_*[\PP^1]+\dim(\MM_{0,n}).$$
If the curve $f(\PP^1)$ is rigid, then it must be that 
$$\dim_{[f]}\Hom(\PP^1,\MM_{0,n})\leq\dim\PGL_2=3.$$ 
Note that by \cite[Lemma 3.5]{KM}, we have:
\begin{equation}\label{canonical}
-K_{\MM_{0,n}}=\sum_{k=2}^{\lfloor\frac{n}{2}\rfloor}\big(2-\frac{k(n-k)}{n-1}\big)\de_k.
\end{equation}
where we denote $\de_k=\sum_{|I|=k}\de_I$.

\end{rmk}

\begin{defn}\label{2.1}
We say that an effective Weil divisor on a projective variety
has ample support if it has the same support as some effective ample divisor.
\end{defn}

\begin{defn}\label{antinef}
We say that an effective divisor $D$ is anti-nef if $D\cdot C\leq0$ for any curve $C$ contained in the support of $D$.
\end{defn}

\begin{prop}\cite{KM}\label{Prop2.2}
Let $M$ be a ${\QQ}$-factorial projective variety and $D$ an effective divisor with ample support, each of whose irreducible components are anti-nef. 
Let $C \subset M$ be a moving irreducible proper curve which generates an edge
$R$ of the Mori cone. Then $R$ is generated by a curve contained in the support of $D$.
\end{prop}

\begin{proof}
Let $p:\,S\ra B$ be a proper surjection from a surface $S$
to a non-singular curve $B$ and let $h:\,S \rightarrow M$ be a morphism such that
$T=h(S) \subset M$ is a surface and there exists a fibre $F$ of $p$ with $h(F)$ set theoretically equal to $C$. Clearly, we may assume $S$ is smooth.  Suppose on the contrary that no
curve in $D \cap T$ generates the same extremal ray as $C$.

Let $D = \sum D_i$ be the decomposition into irreducible components. Let $D'_i=h^{-1}(D_i)$. 
Clearly, each $D'_i$ is an effective $\QQ$-Cartier divisor, and in particular, is purely one dimensional. Let $D'=\sum D'_i$. 
As $D$ has ample support, $D'$ is non-empty. 
Since $C=h(F)$ generates an extremal ray of the Mori cone of $M$, it follows that
any component of any fiber of  $p:S\ra B$ is either contracted by $h$ or 
belongs to the same extremal ray $R$.  
In particular, all components of $D'$ which are not contracted by $h$ are multisections of $p:\,S\ra B$. 

We show next that we can find two irreducible curves $B_1, B_2 \subset D'$ and (after renaming) two divisors $D'_1, D'_2$ with $B_i \subset D_i$ such that $B_1.D'_2>0$ and $B_2.D'_1\geq0$.

Choose an irreducible component $B_1$ of $D'$ not contracted to a point by $h$
and contained in a maximal number of $D'_i$'s. (We use here that $h(S)$ is a surface.)
Suppose that $G_1,\dots,G_k$ are the components of $D'$ containing $B_1$. 
Since the $D_i$'s are anti-nef, $G_i\cdot B_1\le 0$.
Since $D$ has ample support, there exists a $D'_i$ such that $D'_i \cdot B_1 > 0$.
After renaming, we may assume that  $D'_2 \cdot B_1 > 0$. Pick any component $B_2$ of $D'_2$ not contracted to a point by $h$.  By the choice of $B_1$ there exists $G_i\not\supset B_2$. 
We set $D'_1= G_i$.

Let $\la>0$ be such that $E:=D_1' - \lambda D_2'$ has zero intersection with the general fiber. In particular, $E.F=0$. As $R$ is an edge, $E$ is numerically equivalent to $\sum a_i H_i$, where 
$a_i\in\RR$ and $H_i$ is a nef divisor on $S$ with $H_i\cdot F=0$ (pull-back of a supporting hyperplane on $M$). As $F^2=0$, the Hodge Index Theorem implies that for all $i$ we have $H_i=\mu_iF$, for some $\mu_i\in\RR$ and hence, $E=\mu F$, for some $\mu\in\RR$. Since $B_1$, $B_2$ are multisections, it follows that $E.B_1$, $E.B_2$ are both nonnegative or both negative. This gives a contradiction, as by choice of $B_1$ and $B_2$, we have $E.B_1<0$ and $E.B_2\geq0$. 
\end{proof}

\begin{proof}[Proof of Theorem~\ref{KM theorem}] 
Suppose, on the contrary, that $C$ is a moving curve that intersects $M_{0,n}$ and generates an edge of the Mori cone.  Note that the boundary $\De$ of $\MM_{0,n}$ has ample support \cite[Lem. 3.6]{KM} and  every boundary component is anti-nef \cite[Lem. 4.5]{KM}. 
By Prop.~\ref{Prop2.2}, $C$ is numerically equivalent to a positive multiple of a curve $C'$
on the boundary. By Lemma \ref{negative}, there is some boundary divisor which intersects it negatively. Then this divisor intersects $C$ negatively and therefore $C$ is contained in the boundary.  Contradiction.
\end{proof}

\begin{lemma}\label{negative}
For any curve $C$ in the boundary of $\MM_{0,n}$, there is a boundary component $\De_{\alpha}$ such that $C\cdot\De_{\alpha}<0$.
\end{lemma}

\begin{proof}
If $C$ is contained in $\de_{ij}$, consider the Kapranov morphism $\Psi:\MM_{0,n}\ra\PP^{n-3}$
with the $i$-th marking as a moving point. Then $\Psi(\de_{ij})= pt$; if we let $\psi_i=\Psi^*\O(1)$, then $C.\psi_i=0$. It is not hard to see that the class $\Psi_i$ can be written as $\sum\al_I\de_I$ with $\al_I>0$ for all $I$. If $C.\de_I\geq0$ for all $I$, then
it follows that $C.\de_I=0$ for all $I$, which is a contradiction, since the boundary has ample support. If $C$ is contained in some $\de_I$ with $|I|\geq3$, we prove the statement by induction on $|I|$: consider the forgetful map $\pi:\MM_{0,n}\ra\MM_{0,n-1}$ that forgets a marking $i\in I$. Then $\pi(\de_I)=\de_{I\setminus\{i\}}$. If $C$ is not contracted by $\pi$, then by induction, $\pi(C).\de_J<0$ for some $J\subset N\setminus i$.
By the projection formula, $C.\pi^{-1}\de_J=\pi(C).\de_J<0$ and the statement follows, as $\pi^{-1}\de_J=\de_J+\de_{J\cup\{n\}}$. If $\pi(C)=pt$ then  $C$ is a fiber of $\pi$ and it is an easy calculation to show that $C.\de_I<0$.
\end{proof}

For the reader's convenience, we sketch the proof of the following 

\begin{cor}\cite{KM}\label{F-conj}
For $n\leq 7$ the Mori cone $\NE(\MM_{0,n})$ is generated by $F$-curves. 
\end{cor}

\begin{proof}
This is clear for $n\leq5$. Assume $n=6$ or $7$.  We will use here the fact that the Mori cone 
$\NE(\MM_{0,n})$ is polyhedral for $n=6, 7$ (for details see the original argument in \cite{KM}).
By (\ref{canonical}), we have:
$$-K_{\MM_{0,6}}=\frac{2}{5}\de_2+\frac{1}{5}\de_3,\quad -K_{\MM_{0,7}}=\frac{1}{3}\de_2.$$
By Remark \ref{lower bound}, there are no \emph{rational} rigid curves intersecting the interior. This is immediate if $n=7$. For $n=6$ this follows from the fact that the boundary has ample support. We are left to prove that any extremal ray $R$ of  $\NE(\MM_{0,n})$ is generated by a rational curve (then Theorem \ref{KM theorem} will give a contradiction). This follows from the Cone Theorem  for $n=6$ (use that $-K_{\MM_{0,6}}$ has ample support). For $n=7$ this follows from \cite[Prop. 2.4]{KM} (with $D=\de$, $G=\epsilon\de_2$, for $\epsilon<<1$).
\end{proof}

To our knowledge, it is not known if $\NE(\MM_{0,n})$ is polyhedral when $n\geq8$. 


\section{Surfaces in $\MM_{0,n}$ from configurations of points in $\PP^2$}\label{blow-ups section}

We give a simple construction of surfaces in $\MM_{0,n}$ that intersect the interior $M_{0,n}$. 

\begin{thm}\label{blowupdescr}
Suppose $p_1,\ldots,p_n\in\bP^2$ are distinct points, and 
let $U\subset\bP^2$ be the complement to the union of lines connecting them.
The morphism 
$$F:\,U\to M_{0,n}$$ 
obtained by projecting $p_1,\ldots,p_n$ from points of $U$
extends to the morphism 
\begin{equation}\label{cutemorphism}
F:\,\Bl_{p_1,\ldots,p_n}\bP^2\to\oM_{0,n}.
\end{equation}

If there is no (probably reducible) conic through  $p_1,\ldots,p_n$
then $F$ is a closed embedding.
In this case the boundary divisors $\delta_I$ of $\oM_{0,n}$ pull-back as follows:
for each line $L_I:=\langle p_i\rangle_{i\in I}\subset\bP^2$, we have $F^*(\delta_I)=\tilde L_I$ 
(the proper transform of $L_i$) and (assuming $|I|\ge3$),
$F^*(\delta_{I\setminus\{k\}})=E_k$, where $k\in I$ and $E_k$ is the exceptional divisor over $p_k$.
Other boundary divisors pull-back trivially.
\end{thm}

\begin{proof}
For any $I\subset\{1,\ldots,n\}$, we denote by $F_I:\,\bP^2\dra \oM_{0,I}$ a
rational map defined as above but using only points $p_i$, $i\in I$. Then $F_I=\pi_I\circ F$, for the forgetful map
$\pi_I:\MM_{0,n}\ra\MM_{0,I}$.

First suppose that $n=4$. Consider three cases.
If no three out of the four points $p_1,\ldots,p_4$ lie on a line then 
$$F:\,\Bl_{p_1,p_2,p_3,p_4}\bP^2\simeq\oM_{0,5}\to\oM_{0,4}\simeq\bP^1$$ 
is  given by the pencil of conics through $p_1,\ldots,p_4$.
If $p_1,p_2,p_3$ lie on a line that does not contain $p_4$ then 
$F:\,\Bl_{p_4}\bP^2\to\oM_{0,4}\simeq\bP^1$ is a projection from $p_4$.
Finally, if all points lie on a line then $F_{1234}$ is a map to a point (given by the cross-ratio of $p_1,\ldots,p_4$
on the line they span).  Note that in all cases $F$ is regular on $\Bl_{p_1,\ldots,p_n}\bP^2$.
The product of all forgetful maps $\oM_{0,n}\to\prod_I\oM_{0,I}$ over all $4$-element subsets
is a closed embedding (see e.g.~\cite[Th.~9.18]{HKT}).
It follows that \eqref{cutemorphism} is regular.

Now suppose that there is no conic passing through all points.

The argument above shows that $F$ restricted to each exceptional divisor $E_i$ is a closed immersion.
Indeed,  there always exist three points $p_a,p_b,p_c$ such that $p_i$ does not belong
to a line spanned by any two of the three points (otherwise all points belong to a union of two lines
passing through $p_i$, which is a reducible conic). 
By the above, the morphism $F_{abci}|_{E_i}$ is a closed immersion (in fact an isomorphism).

Let $k$ be the maximal number such that there exist $k$ points out of $p_1,\ldots,p_n$
lying on a smooth conic. We can assume without loss of generality that $p_1,\ldots,p_k$ lie on smooth conic.
We consider several cases. Suppose first that $k\ge5$.
Since, for any $4$-element subset $I\subset\{1,\ldots,k\}$, $F_I$ is given by a linear system of conics 
through $p_i$, $i\in I$, the geometric fibers of $F_{1\ldots k}$ are:
(1) the proper transform $\tilde C$ of a conic $C$ through $p_1,\ldots,p_k$ (which does not pass through the remaining points);
(2) exceptional divisors $E_i$, $i>k$; (3) closed points in $\bP^2\setminus\{p_1,\ldots,p_n\}$.
Since we already know that $F|_{E_i}$ is a closed embedding,
it suffices to prove that $F|_{\tilde C}$ is a closed embedding. 
For this, consider $F_{123,k+1}$. There are two subcases.
If $p_1,p_2,p_3,p_{k+1}$ lie on a smooth conic then, since $p_{k+1}\not\in C$, the linear system of
conics through $p_1,p_2,p_3,p_{k+1}$ separate points of $\tilde C$.
If they lie on a reducible conic then $p_{k+1}$ must belong to a line connecting
a pair of points from $p_1,p_2,p_3$, for example $p_2$ and $p_3$.
Then the linear system of lines through $p_1$ separate points of~$\tilde C$.
In both cases, $F|_{\tilde C}$ is a closed embedding.

Note that $k\ne 2$ (otherwise all points lie on a line through $p_1$ and $p_2$).
We claim that $k\ne 3$ either. Arguing by contradiction, suppose that $k=3$.
Then, for any $i>3$, $p_i$ lies on one of the three lines connecting $p_1$, $p_2$, $p_3$.
Moreover, each of these lines must contain at least one of the points $p_i$, $i>3$,
because otherwise all points lie on the union of two lines.
So suppose that 
$$
p_4\in\langle p_1,p_2\rangle,\quad
p_5\in\langle p_2,p_3\rangle,\quad
p_6\in\langle p_1,p_3\rangle.$$ 
But then $p_2,p_3,p_4,p_6$ lie on a smooth conic.

So the only case left is $k=4$. 
Points $p_5,\ldots,p_n$ lie on a union of $6$ lines
connecting $p_1,\ldots,p_4$ pairwise. 
The geometric fibers of $F_{1234}:\,\Bl_{p_1,\ldots,p_n}\bP^2\to\oM_{0,\{1,2,3,4\}}$
are the preimages w.r.t.~the morphism $\Bl_{p_1,\ldots,p_n}\bP^2\to\Bl_{p_1,\ldots,p_4}\bP^2$
of proper transforms of conics $C$ through $p_1,\ldots,p_4$.
If $C$ is a smooth conic then the argument from the $k\ge5$ case
shows that $F|_{\tilde C}$ is a closed embedding.
So suppose that $C$ is a reducible conic, for example the union of lines $\langle p_1,p_2\rangle$
and $\langle p_3,p_4\rangle$. 
Note that not all points belong to these two lines, for example suppose $p_5$ belongs to $\langle p_1,p_3\rangle$.
Then $F_{1352}$ collapses $\langle p_1,p_2\rangle$ and separates points of $\langle p_3,p_4\rangle$.
$F_{1354}$ has an opposite effect. So $F_{13524}$ separates points of $\tilde C$ and we are done.

To compute pull-backs of boundary divisors, note that $F^{-1}(\partial\oM_{0,n})=\partial U$
(set-theoretically), and so, for any subset $I$, $F^*\delta_I$ (as a Cartier divisor) is a linear combination
of proper transforms of lines $L_J=\langle p_j\rangle_{j\in J}$ and exceptional divisors~$E_i$.
In order to compute multiplicity of $F^*\delta_I$ at one of these divisors $D$,
we can argue as follows: suppose $C\subset\Bl_{p_1,\ldots,p_n}\bP^2$
is a proper curve intersecting $D$ transversally at a point $p\in C$ that does not belong to any other boundary component.
By the projection formula, the multiplicity is equal to the local intersection number of $F(C)$ with 
$\delta_I$ at $F(p)$. But this intersection number can be immediately computed 
from the pullback of the universal family of $\oM_{0,n}$ to $C$.
To implement this program, we consider two cases.
First, suppose that $D=L_J$.
Working locally on $\bA^2_{x,y}\subset\bP^2$, we can assume that $p=(x,y)$,
$D=(x)$, $C=(y)$, $J=\{1,\ldots,k\}$, $p_i=(x,y-b_i)$, $b_i\ne0$, for $i\le k$, and $p_i=(x-a_i,y-b_i)$, for $i>k$, where 
$a_i\ne0$, $b_i\ne 0$, and $a_i/b_i\ne a_j/b_j$. Then (locally near $p$) the pull-back of the universal family
of $\oM_{0,n}$ to the punctured neighborhood $U\subset C$ of $p$ has a chart $\Spec k[x,1/x,s]_{(x)}\to\Spec k[x,1/x]_{(x)}$
with sections $(x+sb_i)$ for $i\le k$ and $(x+sb_i-a_i)$ for $i>k$.
Closing up the family in $\Spec k[x,s]_{(x)}\to\Spec k[x]_{(x)}$ and blowing-up the origin
$(x,s)\in \Spec k[x,s]_{(x)}$
separates the first $k$ sections. The special fiber has two components,
with points marked by $J$ one component and points marked by $J^c$ on the other.
This proves the claim in the first case.

Secondly, suppose that $D=E_1$. We assume that $p_1=(x,y)\in\bA^2\subset\bP^2$.
We work on the chart $\Spec k[x,t]\subset\Bl_{p_1}\bA^2$ where $y=tx$.
Then $E_1=(x)$. We can assume that $p=(x,t)$, $C=(t)$, and that $p_i=(x-a_i, t-t_i)$ for $i>1$, where $a_i\ne0$, $t_i\ne0$.
Then (locally near $p$) the pull-back of the universal family
of $\oM_{0,n}$ to the punctured neighborhood $U\subset C$ of $p$ has a chart $\Spec k[x,1/x,s]_{(x)}\to\Spec k[x,1/x]_{(x)}$
with sections $s_1=(s)$, $s_i=(s-t_i-sxa_i^{-1})$ for $i>1$. 
We close-up in $\Spec k[x,s]_{(x)}\to\Spec k[x]_{(x)}$
and resolve the special fiber by blowing up points $(x,s-t_i)$ each time there is more than one point 
with the same slope $t_i$. This yields a family of stable curves with a special fiber that contains
(a) a ``main'' component with points marked by $1$ and by $i$ each time 
there is just one point with the slope~$t_i$; (b) one component (attached to the main component)
for each $t_i$ that repeats more than once marked by $j$ such that $t_i=t_j$.
This proves the claim in the second case.
\end{proof}

\begin{ex}
Applying this to $n=6$ gives a covering of $\MM_{0,6}$ by cubic surfaces. This is related to the fact that $\oM_{0,6}$
is a resolution of singularities of the Segre cubic threefold
$$\cS=\{(x_0:\ldots:x_5)\,|\,\sum x_i=\sum x_i^3=0\}\subset\bP^5.$$
Using the formula \cite[Rk.3.1]{HT} for the pull-back of the hyperplane section of~$\cS$,
it is easy to check that our blow-ups are pull-backs of hyperplane sections of~$\cS$.
This proves a well-known classical fact that moduli of cubic surfaces are generated by hyperplane sections of $\cS$
(the Cremona hexahedral equations, see \cite{Do}).
It deserves mentioning that one of the (non-general) blow-ups of $\PP^2$ in $6$ points 
embedded in $\MM_{0,6}$ this way is the ``Keel-Vermeire divisor'', 
see \cite[Section 9]{CT2}.
\end{ex}

We end this section with the following observation:
\begin{prop}\label{lines+exc}
In the set-up of Theorem \ref{blowupdescr}, the numerical classes 
of proper transforms of lines and exceptional divisors on the blow-up $\Bl_{p_1,\ldots,p_n}\PP^2$ are sums of $F$-curves. 
\end{prop}

We will give an explicit example of how Prop. \ref{lines+exc} applies in \ref{example: sums of F-curves}.

\begin{proof}
We argue by induction on $n$. The proper transform of any line in $\Bl_{p_1,\ldots,p_n}\PP^2$
is linearly equivalent to the sum of exceptional divisors and the proper transform 
of a line passing through at least two points of $p_1,\ldots,p_n$, so it is enough to consider these
two cases.

Case I. The exceptional divisor over $p_i$ maps to a point by the $i$-th forgetful map $\oM_{0,n}\to\oM_{0,n-1}$.
Any irreducible component $C$ of any fiber of the forgetful map is easily seen to be a sum of $F$-curves:
if the corresponding irreducible component of the $(n-1)$-pointed stable rational curve has three distinguished points
then $C$ is an $F$-curve. However, any fiber can be degenerated to a fiber over a $0$-dimensional stratum of $\oM_{0,n-1}$.

Case II. Consider the proper transform of a line $L_I$ through at least $2$ points of $p_1,\ldots,p_n$. 
The corresponding curve of $\oM_{0,n}$ belongs to the boundary divisor~$\delta_I$, so it suffices
to show that its projections onto $\oM_{0,|I|+1}$ and $\oM_{0,n-|I|+1}$ are sums of $F$-curves.

The projection onto $\oM_{0,n-|I|+1}$ can be interpreted as follows: remove points indexed by $I$
from $\bP^2$ and place an extra point $p$ at a general point of $L_I$. Now repeat the 
construction of Theorem~\ref{blowupdescr} for this new configuration. 
By inductive assumption, the proper transform of the line
is the sum of $F$-curves on $\oM_{0,n-|I|+1}$.

The projection onto $\oM_{0,|I|+1}$ is immediate: forgetting the extra marking maps the curve
to a point of $\oM_{0,|I|}$ (given by cross-ratios of $p_i$, $i\in I$ along $L_I$).
So we are done as in Case~I.
\end{proof}

The next simplest curves in the surfaces $\Bl_{p_1,\ldots,p_n}\PP^2$ are proper transforms of conics through $5$ points. The following is an immediate
corollary of  Thm.~\ref{blowupdescr}:

\begin{cor}\label{class1}
In the set-up of Theorem \ref{blowupdescr}, assume the points $p_1,\ldots,p_5$ are in general position and the smooth conic $C$ passing through them contains no other points $p_i$, $i>5$. Then the proper transform $\tilde{C}\subset\oM_{0,n}$ of $C$ has the following intersections with boundary divisors:
for each line $L_I$, 
$$\delta_I\cdot\tilde C=2-|I\cap\{1,\ldots,5\}|,$$ and for each $k\in I$, 
$$\delta_{I\setminus\{k\}} \cdot\tilde C=\begin{cases}
1& \hbox{if}\ k\le5\cr
0 & \hbox{otherwise}.
\end{cases}$$
Other intersection numbers are trivial.
\end{cor}

We analyze in detail an example of such a curve in Section \ref{HesseSection}. 


\section{The hypergraph construction}\label{hypergraph section}

\begin{defn}
Let $f:X\ra Y$ be a quasiprojective morphism of Noetherian schemes. The \emph{exceptional locus} $\Exc(f)$ is the complement to the union of points in $X$ isolated in their fibers. By  \cite[4.4.3]{EGA3}., $\Exc(f)$ is closed. 
\end{defn}

We use the following observation to construct rigid curves on $\MM_{0,n}$:

\begin{prop}\label{rigid}
If a curve $C\subset X$ is an irreducible component of the exceptional locus of a morphism $f: X\ra Z$, with $X$ and $Z$ projective varieties, then $C$ is rigid on $X$.
\end{prop}

\begin{proof}
Assume $C$ is not rigid, i.e., there is a family $\pi:S\ra B$ over a smooth curve $B$, a map 
$h: S\ra X$ such that $h(S)$ is a surface and for some fiber $F$ of $\pi$ we have (set-theretically) that $h(F)=C$. Since the fibers of $\pi$ are numerically equivalent on $S$, and as $F$ is contracted by $f\circ h$, it follows that if $A$ is some ample divisor on $Z$, then every fiber of $\pi$ intersects $f\circ h^{-1}(A)$ trivially. Hence,  every fiber of $\pi$ is contracted by $f$, i.e., contained in $\Exc(f)$. It folows that $h(S)\subseteq\Exc(f)$. As $C$ is an irreducible component of $\Exc(f)$, this is a contradiction. 
\end{proof}

One is left to find a morphism $f:\MM_{0,n}\ra Z$ as in Prop. \ref{rigid}. The most natural morphisms to consider are products of forgetful morphisms. We first make the following:

\begin{defn}\label{hypergraph}
A \emph{hypergraph} $\Gamma=\{\Gamma_1,\ldots,\Gamma_d\}$ on the set $N=\{1,\ldots, n\}$
is a collection of subsets of $N$, called \emph{hyperedges}, such that the following conditions are satisfied:
\begin{itemize}
\item Any subset $\Gamma_j$ has at least three elements;
\item Any $i\in N$ 
is contained in at least
two subsets~$\Gamma_j$.
\end{itemize}
\end{defn}

\begin{defn}
We call a \em{hypergraph morphism} the product of fogetful maps 
$$\pi_{\Ga}: \MM_{0,n}\ra\prod_{\al=1}^d \MM_{0,\Gamma_{\al}}.$$
\end{defn}

Definition \ref{hypergraph} generalizes the notion of \emph{hypertree} introduced in \cite{CT2} 
(this construction has first appeared in \cite{CT1}). Essentially, a hypergraph is the simplest structure that allows one to study exceptional loci of products of fogetful maps, by using Brill-Noether theory of certain reducible curves. The following are some of the constructions in \cite{CT2} in a slightly more general context.

\begin{defn}\label{valence}
Let $\Gamma=\{\Gamma_1,\ldots,\Gamma_d\}$ be a hypergraph. 
A curve $\Sigma$ is called a {\em hypergraph curve} associated to $\Ga$ if it has $d$ irreducible components $\Sigma_1,\ldots, \Sigma_d$, with $\Sigma_{\al}\cong\PP^1$, marked by $\Gamma_{\al}$ and glued at identical markings as a scheme-theoretic push-out:
at each singular point $i\in N$, $\Sigma$
is locally isomorphic to the union of coordinate axes in $\AA^{v_i}$, where $v_i$
is the valence of $i$, i.e., the number of subsets $\Gamma_{\al}$ that contain~$i$.
We consider $\Sigma$ as a marked curve (by indexing its singularities).
\end{defn}

\begin{review} {\bf Identifying $M_{0,n}$ with a space of maps $\Sigma\ra\PP^1$.}
If not all the $\Ga_{\al}$ are triples, hypergraph curves will have moduli, namely
$$M_\Gamma:=\prod_{j=1,\ldots,d}M_{0,\Gamma_j}.$$

We observe that  $M_{0,n}$ can be identified with the variety of morphisms 
$$f:\,\Sigma\to\PP^1$$
(modulo the free action of $PGL_2$),  that send singular points
$p_1,\ldots,p_n$ of $\Sigma$  to different points $q_1,\ldots,q_n\in \PP^1$. (Note that the point in $M_\Gamma$ corresponding to $\Sigma$ is determined by 
the hypergraph morphism $M_{0,n}\ra M_{\Ga}$.) 
This gives a morphism 
\begin{equation}\label{vmap}
v:\,M_{0,n}\to \Pic^{\underline{1}},\quad f\mapsto f^*\cO_{\PP}(1)
\end{equation} 
from $M_{0,n}$ to the (relative over $M_\Gamma$) Picard scheme 
$\Pic^{\underline{1}}$ of line bundles on $\Sigma$ of degree $1$ on each irreducible component.
\end{review}

\begin{review} {\bf The exceptional locus of a product of forgetful maps.}
As remarked in \cite[Rmk. 2.6]{CT2} most of the constructions in \cite[2.1]{CT2} hold in this more general context. For the reader's convenience, we recall the main construction. 

A linear system on the hypergraph curve $\Sigma$ is said to be \emph{admissible} if it is globally generated and the corresponding morphism $\Sigma\ra\PP^1$ sends the singular points of $\Sigma$ to distinct points. We define the Brill-Noether loci $W^r$ and $G^r$ as follows. The locus $W^r\subset\Pic^{\underline{1}}$ parametrizes line bundles $L\in\Pic^{\underline{1}}$ such that for each hypergraph curve $\Sigma$, the complete linear system $|L_{|\Sigma}|$ is admissible, and we have:
$$\h^0(\Sigma, L)\geq r+1.$$

The locus $G^r$ parametrizes admissible pencils on $\Sigma$ such that the corresponding line bundle is in $W^r$. We have natural forgetful maps 
$$v: G^r\ra W^r.$$
We refer the reader to \cite[Section 2]{CT2}  for the details. Note that $G^r$ and $W^r$ could possibly be empty for $r\geq2$. The key point in the construction is the following:
\end{review}

\begin{thm}\cite[Thm 2.4]{CT2}\label{Exc}
There is an isomorphism $G^1\cong M_{0,n}$ over $M_{\Ga}$ and the map $$v:M_{0,n}\cong G^1\ra \Pic^{\underline{1}}$$ has exceptional locus $G^2$. In particular, $G^2$ is contained in the exceptional locus of the morphism: 
$${\pi_{\Ga}}_{|M_{0,n}}: M_{0,n}\ra\prod_{\al=1}^d M_{0,\Gamma_{\al}}$$
\end{thm}

In contrast with the map $v$, it seems quite difficult to understand in general the full exceptional locus of the map $\pi_{\Ga}$. (An easy case is when all $\Gamma_{\al}$ contain the same index \cite[Thm 2.4]{CT2}.) In our quest for small exceptional loci, the least we can require is that  $W^2$ is small (for example a point).  First note that Theorem \ref{Exc} has the following:

\begin{cor}\label{Cor1}
Let $\Sigma$ be a hypergraph curve and let  $L\in W^2\setminus W^3$ be an admissible line bundle whose restriction to $\Sigma$ gives a morphism $f:\,\Sigma\to\PP^2$.
Let $U:=\PP^2\setminus f(\Sigma)$. 
\begin{itemize}
\item[(a) ] The geometric fiber of $v:\,G^1 \to W^1$ over $(\Sigma,L)\in W^2$
is isomorphic to~$U$.
Its geometric points correspond to morphisms 
$$\Sigma\to f(\Sigma)\arrow^{pr_x}\PP^1,$$ 
where $pr_x:\,\PP^2\dra\PP^1$ is a linear projection from $x\in U$.

\item[(b) ] If $W^2$ is a point (and $W^3$ is empty) then $U$ is the exceptional locus of $v$.
\end{itemize}
\end{cor}

For the remaining part of this section we assume that we have the setup of Cor.~\ref{Cor1}~(b), i.e., ~$W^2$ is a point and $W^3$ is empty. Let $m_0=p(W^2)$ and let $\Sigma$ be the fiber of the universal family of hypergraph curves over $m_0\in M_\Gamma$.
Let $p_1,\ldots,p_n\in\bP^2_k$ be the images of singular points of $\Sigma$ under 
the linear system $|W^2|$.

\begin{prop}\label{propermap}
In the setup of Cor.~\ref{Cor1} (b),  $U$ belongs to the exceptional locus of $\pi_\Gamma:\,M_{0,n}\to M_\Gamma$.
If, moreover,  points  $p_1,\ldots,p_5$ lie on a smooth conic~$C$, then
$C\cap U$ belongs to the exceptional locus of the morphism 
\begin{equation}\label{record}
\pi:=(\pi_\Gamma\times\pi_I)_{|M_{0,n}}:\,M_{0,n}\arrow M_\Gamma\times M_{0, I},
\end{equation}
where $I=\{1,2,3,4,5\}$.
If $C\cap U$ is an irreducible component of the exceptional locus of the morphism $\pi$, then $\overline{C\cap U}\subset\oM_{0,n}$ 
is a rigid curve on $\oM_{0,n}$. 
\end{prop}

\begin{proof}
Clearly, $C\cap U$ is the exceptional locus for the map $U\to M_{0,5}$
given by projecting $p_1,\ldots,p_5$ from points of $U$. Hence, $\overline{C\cap U}$ is 
contained in the exceptional locus of the hypergraph morphism 
$$\tilde{\pi}:=\pi_\Gamma\times\pi_I:\MM_{0,n}\arrow \MM_\Gamma\times \MM_{0, I}.$$

Since $C\cap U$ is a component of $\Exc(\pi)=\Exc(\tilde{\pi})\cap M_{0,n}$, it follows that $\overline{C\cap U}$ must be an irreducible component of $\Exc(\tilde{\pi})$ and we are done by Prop.  
\ref{rigid}.
\end{proof}


\section{The dual Hesse configuration and a rigid curve on $\MM_{0,12}$.}\label{HesseSection}

It remains to find a hypergraph that satisfies the last condition of Prop.~\ref{propermap}.
At the very least we need $\Gamma$ such that $W^1$ has relative dimension~$0$.
By the Brill--Noether theory, the relative dimension of $W^1$ is at least 
$$g-2(g-d+1)=\dim M_{0,n}-\dim M_\Gamma,$$
where $g$ is the arithmetic genus of the associated hypergraph curve $\Sigma$. 

\begin{Review}
Consider the hypergraph of the {\em dual}\/ Hesse configuration (see Fig. 1).
We use the following enumeration of its hyperedges:
 $$\Ga_1=\{p, 1,b,\ga\},\quad \Ga_2=\{p,2,c,\be\},\quad \Ga_3=\{p,3,a,\al\}$$
$$\Ga_4=\{n, 2,a,\ga\},\quad \Ga_5=\{n,3,b,\be\},\quad\Ga_6=\{n,1,c,\al\}$$
$$\Ga_7=\{m, 1,2,3\},\quad \Ga_8=\{m,\al,\be,\ga\},\quad  \Ga_9=\{m,a,b,c\}$$
It has $d=9$ hyperedges with $4$ points on each hyperedge, with $12$ vertices.
Note that $g=16$ and the expected relative dimension of $W^1$ is $0$.

Let $\Ga$ be the hypergraph $\{\Ga_1,\ldots,\Ga_9,\Ga_0\}$ where:
$$\Ga_0=\{m,n,p,1,a\}$$
(this corresponds to adding a conic $C$ through $5$ points in Prop.~\ref{propermap}).
\end{Review}

\begin{thm}\label{HesseMain} 
The hypergraph morphism  
$$\pi:={\pi_{\Gamma}}_{|M_{0,12}}:M_{0,12}\ra M_{\Ga}=(M_{0,4})^9\times M_{0,5}$$ 
has a $1$-dimensional connected component in the closure of its exceptional locus in $M_{0,12}$.
This connected component is in fact irreducible and is 
the proper transform $C$ in $\Bl_{12}\bP^2$ of the conic in $\bP^2$
passing through $5$ points $\{m,n,p,1,a\}$ of the dual Hesse configuration.
\end{thm}

\begin{proof}
Let $\rho$ be a closed point of $M_{0,12}=G^1$. Then $\rho$ gives rise to the morphism $\Sigma\to\bP^1$
and we let $x'=\rho(x)$ for any singular point $x$ of $\Sigma$.
Without loss of generality we can assume that 
$$1'=\infty,\quad m'=0,\quad a'=1,$$
and we let
$$b'=t,$$
where $t\in k$ is a parameter.

In these coordinates the morphism $\pi$ has the following form:
$$w_1=[p', 1',b',\ga'],\quad w_2=[p',2',c',\be'],\quad w_3=[p',3',a',\al'],$$
$$w_4=[n', \ga',a',2'],\quad w_5=[n',\be',b',3'],\quad w_6=[n',\al',c',1'],$$
$$w_7=[m', 1',2',3'],\quad w_8=[m',\ga',\be',\al'],\quad  w_9=[m',b',c',a'],$$
$$u=p',\quad v=n',$$
where $$[x,y,z,s]=\frac{(s-x)(y-z)}{(y-x)(s-z)}$$ is the cross-ratio
and $(u,v)$ are coordinates on $M_{0,5}$.

\begin{claim}
The natural morphism $M_{0,12}\to (M_{0,4})^9\times M_{0,\{1,m,a,b,p,n\}}$
is injective on closed points. In particular, $\pi_\Gamma$ has at most one-dimensional fibers.
\end{claim}

\begin{proof}
We will show how to recover all points $x'$ starting from $1'$, $m'$, $a'$, $b'$, $p'$, $n'$
and using    coordinates on $M_\Gamma$.  From the cross-ratio $w_9$ we find that:
$$c'=\frac{(w_9-1)t}{w_9t-1}.$$
From the cross-ratio $w_1$ we find that:
$$\ga'=\frac{w_1t-u}{w_1-1}.$$
From the cross-ratio $w_4$ we find that:
$$2'=\frac{-w_4v+v+\ga'(w_4-v)}{-w_4v+1+\ga'(w_4-1)}=\frac{-v(w_4-1)(w_1-1)+(w_4-v)(w_1t-u)}{(1-w_4v)(w_1-1)+(w_4-1)(w_1t-u)}.$$
For simplicity, we think of this as $2'=\frac{C}{D}$ where
\begin{equation}\label{C}
C=-v(w_4-1)(w_1-1)+(w_4-v)(w_1t-u),
\end{equation}
\begin{equation}\label{D}
D=(1-w_4v)(w_1-1)+(w_4-1)(w_1t-u).
\end{equation}
From the cross-ratio $w_6$ we find that:
$$\al'=\frac{w_6v-c'}{w_6-1}=\frac{w_6v(w_9t-1)-(w_9-1)t}{(w_6-1)(w_9t-1)}.$$
For simplicity, we think of this as $\al'=\frac{A}{B}$ where
\begin{equation}\label{A}
A=w_6v(w_9t-1)-(w_9-1)t,
\end{equation}
\begin{equation}\label{B}
B=(w_6-1)(w_9t-1).
\end{equation}
From the cross-ratio $w_7$ we find that:
$$3'=\frac{w_72'}{w_7-1}=\frac{w_7C}{(w_7-1)D}.$$
Finally, from the cross-ratio $w_8$ we find that:
$$\be'=\frac{M}{N}$$ where we denote:
\begin{equation}\label{M}
M=(1-w_8)(w_1t-u)A,
\end{equation}
\begin{equation}\label{N}
N=(w_1-1)A-w_8(w_1t-u)B.
\end{equation}
This shows the claim.
\end{proof}

\begin{lemma}
The locus in $M_{\Ga}$ where the fiber of the hypergraph map is positive dimensional 
is given by those points for which the following polynomials in $t$
with coefficients in $k[w_1,\ldots,w_9,u,v]$  
are identically zero:
\begin{equation}\label{C3}
(A-uB)[w_7C-(w_7-1)D]-w_3(A-B)[w_7C-u(w_7-1)D]=\\
f_1t^2+f_2t+f_3,
\end{equation}

$$
[w_7C-v(w_7-1)D](M-tN)-w_5(M-vN)[w_7C-t(w_7-1)D]\\
$$
\begin{equation}\label{C5}
=f_4t^4+f_5t^3+\ldots+f_8,
\end{equation}

$$
[(w_9t-1)C-(w_9-1)tD](M-uN)-w_2[(w_9t-1)M-(w_9-1)tN](C-uD)
$$\begin{equation}\label{C2}
=f_9t^4+f_{10}t^3+\ldots+f_{13},
\end{equation}
where $A,B,C,D,M,N$ are as in \eqref{C} -- \eqref{N}.
\end{lemma}

\begin{proof}
We get equations on $t$ by utilizing the cross-ratios not used in the proof of the previous Claim.
Namely, we get \eqref{C3} from the points $3'$, $p'$, $a'$, $\al'$ and $w_3$;
we get \eqref{C5} from the points $n'$, $3'$, $b'$, $\be'$ and $w_5$;
we get \eqref{C2} from the points $p'$, $2'$, $c'$, $\be'$ and $w_2$.
For example: we require that $[p', 3', a', \al']=\om_3$. This is equivalent to:
$$(\al'-p')(3'-a')=\om_3(3'-p')(\al'-a').$$ 

As $a'=1$, $p'=u$, $\al'=\frac{A}{B}$ and $3'=\frac{\om_7C}{(\om_7-1)D}$, this implies:
$$(A-uB)[w_7C-(w_7-1)D]-w_3(A-B)[w_7C-u(w_7-1)D]=0.$$

Note that $A, B, C, D$ are linear polynomials in $t$. Note that equality must hold for all $t$ (remember, we are looking for one-dimensional fibers of the map $\pi$).  This implies that the degree two polynomial in \eqref{C3} must be identically zero. 
\end{proof}

Let $m_0\in M_\Gamma$ be the point that corresponds to the dual Hesse configuration in $\bP^2$. It is not realizable over $\bR$, so we can give 
only its vague sketch, see Fig.~\ref{dHesse}.

\begin{figure}[htbp]
\includegraphics[width=4in]{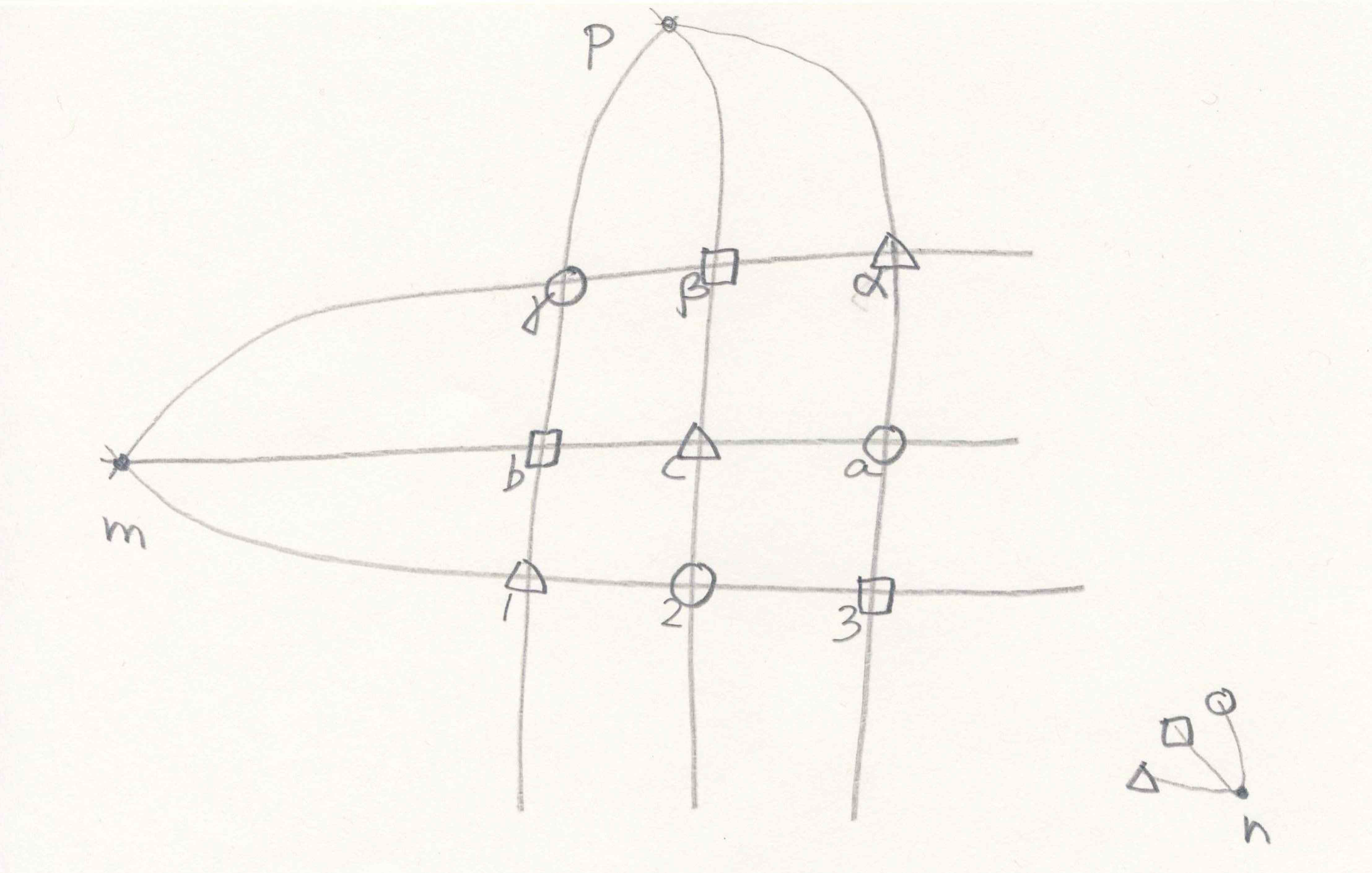}
\caption{\small A dual Hesse hypergraph.
}\label{dHesse}
\end{figure}
 
Note that ``circles'' (resp.~``squares'', resp.~``triangles'') span lines $\Ga_4$, $\Ga_5$, and $\Ga_6$.
Alternatively, one can choose coordinates in $\bP^2$ such that
$$\Ga_1\Ga_2\Ga_3=X^3-Z^3,\quad  \Ga_4\Ga_5\Ga_6=X^3-Y^3,
\quad \Ga_7\Ga_8\Ga_9=Y^3-Z^3.$$

\begin{lemma} Let $\omega$ be the primitive cubic root of~$1$. The point 
$m_0$ has coordinates 
$$w_1^0=\ldots=w_9^0=-\omega^2,\quad  u^0=1-\omega,\quad v^0=1-\omega^2.$$
The differentials of functions $f_1,\ldots,f_{13}$ at $m_0$ do not depend on $u$ and $v$
and the Jacobian matrix $[\partial f_i/\partial w_j]$ at $m_0$
is given by  
$$\left[\small\begin{array}{ccccccccccccccccccccccc}
0& 0& 0& 1& 0& \omega+1& -\omega-1& 0& \omega+1\cr
-1& 0& \omega& \omega& 0& -\omega-3& 2 \omega+3& 0& -\omega-1\cr
-\omega+1& 0& 0& 0& 0& -\omega+1& -\omega-2& 0& 0\cr
0& 0& 0& 1& -\omega-1& \omega& -\omega-1& 1& \omega\cr
0& 0& 0& 2 \omega-1& 3 \omega+4& -5 \omega-3& 3 \omega+5& 3 \omega-1& -3 \omega-1\cr
0& 0& 0& -5 \omega-1& -3 \omega-9& 8 \omega+10& -2 \omega-10& -9 \omega-3& 5 \omega+4\cr
0& 0& 0& 3 \omega+3& 9& -3 \omega-12& -3 \omega+9& 9 \omega+9& -3 \omega-3\cr
0& 0& 0& 0& 3 \omega-3& -3 \omega+3& 3 \omega-3& -3 \omega-6& 0\cr
0& 0& 0& -2 \omega& 0& 2& 0& -2 \omega-2& 2 \omega+2\cr
-2& 4 \omega+4& 0& 6 \omega+7& 0& 9 \omega-1& 0& -\omega+8& -\omega-6\cr
-7 \omega+1& -12& 0& \omega-7& 0& -20 \omega-16& 0& 15 \omega& -4 \omega+4\cr
12 \omega+9& 3-12 \omega& 0& -3 \omega& 0& 6 \omega+18& 0& -15 \omega-12& 3 \omega\cr
-3 \omega-6& 6 \omega+3& 0& 0& 0& 3 \omega-3& 0& 3 \omega+6& 0\cr
\end{array}\right]$$
It has rank~$9$ (rows $1,2,3,6,7,8,11,12,13$ are linearly independent). 
Consider the following functions:
$$g_1=45f_4+27f_5+(3-6\omega)f_6-(10\omega+5)f_7-(6\omega+3)f_8$$
$$g_2=-18f_4+(6\omega-6)f_5+6\omega f_6+(4\omega+2)f_7+(2\omega+2)f_8$$
$$g_3=126f_9+(63\omega+126)f_{10}+(105\omega+126)f_{11}+(161\omega+112)f_{12}+(189\omega+42)f_{13}$$
Their differentials at $m_0$ are identically $0$ and the Hessians
$$\left[\begin{matrix}{\partial^2g_k\over\partial u\partial u}&{\partial^2g_k\over\partial u\partial v}\cr
{\partial^2g_k\over\partial v\partial u}&{\partial^2g_k\over\partial v\partial v}\end{matrix}\right],\qquad k=1,2,3$$
at $m_0$ are equal to 
$$
\left[\begin{matrix}-18\omega-18&-30\omega-12\cr
-30\omega-12&-12\omega+54\end{matrix}\right],\quad
\left[\begin{matrix}4\omega+8&16\omega+8\cr
16\omega+8&16\omega-16\end{matrix}\right],\quad
\left[\begin{matrix}-126\omega+42&42\omega+84\cr
42\omega+84&42\omega+42\end{matrix}\right].$$
These three matrices are linearly independent.
\end{lemma}

\begin{proof}
This is a straightforward calculation and a joy of substitution.
\end{proof}

Now we can finish the proof of the Theorem.
It suffices to show that the scheme $\cZ$ cut out by 
the ideal $\langle f_1,\ldots,f_{13}\rangle$ is zero-dimensional at $m_0$.
This would follow at once if the tangent cone of $\cZ$ at $m_0$ is zero-dimensional.
By the Lemma, the ideal of the tangent cone contains functions $w_i-w_i^0$ (for $i=1,\ldots,9$),
$(u-u^0)^2$, $(u-u^0)(v-v^0)$, and $(v-v^0)^2$, which clearly cut out $m_0$ set-theoretically.
\end{proof}

\begin{rmk}
The dual Hesse configuration is a $q=3$ case of the $\Ceva(q)$ arrangement
with $3q$ lines that satisfy 
$$\Ga_1\ldots\Ga_q=X^q-Y^q,\quad \Ga_{q+1}\ldots\Ga_{2q}=Y^q-Z^q,\quad \Ga_{2q+1}\ldots\Ga_{3q}=Z^q-X^q.$$
We think it is plausible that these hypergraphs also give rise to $1$-dimensional exceptional
loci (on $\oM_{0,q^2+3}$).
\end{rmk}

\begin{notn}\label{formal}
We denote by $\De_I$ a formal curve class that has intersection $1$ with $\de_I$ and
$0$ with the rest of  boundary divisors.
\end{notn}

\begin{review}{\bf Class of $C$.}\label{Hesse curve class}
In the setup of Th.~\ref{HesseMain}, the numerical class of $C$ can be computed using Cor. \ref{class1}:
$$
\De_{1,b,\ga}+\De_{p, b,\ga}+
\De_{p,2,c,\be}+\De_{2,c,\be}+
\De_{3,a,\al}+\De_{p,3,\al}+
\De_{2,a,\ga}+\De_{n, 2,\ga}+
\De_{n,3,b,\be}$$
$$+\De_{3,b,\be}+
\De_{1,c,\al}+\De_{n,c,\al}+
\De_{1,2,3}+\De_{m, 2,3}+
\De_{m,\al,\be,\ga}+\De_{\al,\be,\ga}+
\De_{a,b,c}+\De_{m,b,c}$$
$$+\De_{1,\be}+2\De_{2,b}+2\De_{2,\al}+
2\De_{3,c}+2\De_{3,\ga}+\De_{a,\be}+
2\De_{b,\al}+2\De_{c,\ga}.$$
\end{review}


\section{The ``Two Conics"  construction}\label{deJ-K}

\begin{defn}
Recall that any configuration of lines $\{L_1,\ldots, L_k\}$ in $\PP^2$ has an associated {\em matroid}. 
This is a collection of subsets of the set $\{1,\ldots,k\}$ representing 
linearly independent subsets of the set of linear equations of lines $\{L_1,\ldots, L_k\}$.
We say that a configuration of distinct lines $\{L_1,\ldots, L_k\}$ in~$\PP^2$ is a \emph{rigid configuration} if any 
configuration of lines with the same matroid can be obtained from $\{L_1,\ldots, L_k\}$ via 
an automorphism of $\PP^2$. 
\end{defn}

\begin{review}{\bf The ``Two Conics" Construction.}
Let $\{L_1,\ldots, L_{n-3}\}$ be a rigid configuration of lines in $\PP^2$ and let 
$\{p_1,\ldots, p_k\}$ be the set of intersection points. Assume that there are two smooth, non-tangent, conics $C_1$ and $C_2$, each passing through five points in $\{p_1,\ldots, p_k\}$, with the  intersection $C_1\cap C_2$ containing exactly three points from $\{p_1,\ldots, p_k\}$. Let $p$ be the fourth point of intersection of $C_1$, $C_2$. For simplicity, we'll assume none of the lines is tangent to any of the conics. Let $\{p_{k+1},\ldots, p_l\}$ (for some $l>k$) be the points of intersection of  $C_1$, $C_2$ with the lines $L_1,\ldots, L_{n-3}$. 

The pencil of lines through $p$ gives a family of $n$-pointed rational curves as follows: Each line through $p$ is marked by the intersections with the lines $L_i$, the second intersection points with $C_1$ and $C_2$ and the point $p$ itself. More precisely, let $S_0$ be the blow-up of $\PP^2$ at $p$ and let $E_p$ be the exceptional divisor. Then $E_p$ together with the proper transforms of the  lines and conics give $n$ sections of the $\PP^1$-bundle $S_0\ra\PP^1$. These sections are pairwise transversal
and therefore can be disconnected by simple blow-ups as follows. 
Let $S$  be the blow-up of $S_0$ along $p_1,\ldots, p_l$ and 
the points $q_i:=\tilde{C}_i\cap E_p$, $i=1, 2$.  Let $E_i$ (resp., $E'_1, E'_2$) 
denote the exceptional divisors corresponding to the points $p_i$, $i=1,\ldots, l$ (resp., $q_1, q_2$). 
Since none of the conics is tangent to any of the lines, the proper transforms 
$$\tilde{L}_1,\ldots, \tilde{L}_{n-3}, \tilde{C}_1, \tilde{C}_2, \tilde{E}_p$$ form 
$n$ disjoint sections of the family $\pi: S\ra\PP^1$. 
\end{review}

\begin{notn}\label{map f} Let $f: \PP^1\ra\MM_{0,n}$
be the map induced by the family 
$$(\pi: S\ra\PP^1, \tilde{L}_1,\ldots, \tilde{L}_{n-3}, \tilde{C}_1, \tilde{C}_2, \tilde{E}_p).$$ 
We will denote by  $u, v$ the markings corresponding to $\tilde{C}_1, \tilde{C}_2$, i.e,  we have: $$\MM_{0,n}=\MM_{0,\{1,\ldots, n-3, u, v, p\}}.$$
Recall that the forgetful map $\MM_{0,n+1}\to \MM_{0,n}$
is the universal family. So we have 
$S\cong\PP^1\times_{\MM_{0,n}}\MM_{0,n+1}$. Let $h: S\ra\MM_{0,n+1}$ be the pull-back map.
\end{notn}

\begin{prop}\label{class2}
The maps $f: \PP^1\ra\MM_{0,n}$  and  $h: S\ra\MM_{0,n+1}$ of (\ref{map f}) are closed immersions. The boundary divisors of $\MM_{0,n+1}$ pull-back as follows: For each  point $p_j$ ($j=1,\ldots, l$) which is the intersection point of the lines and conics indexed by the subset 
$I\subseteq\{1,\ldots, n-3, u, v\},$ we have $h^*\de_{I\cup\{n+1\}}=E_j$. Moreover, 
$h^*\de_{\{u,p,n+1\}}=E'_1$ and $h^*\de_{\{v,p,n+1\}}=E'_2$. Other boundary divisors pull-back trivially.
\end{prop}

\begin{proof}
We renumber the lines so that the lines $L_1, L_2, L_3$ do not pass through $p$. For the first claim, it suffices
to check that the composition $\bP^1\arrow^f\MM_{0,n}\to\MM_{1,2,3,p}$ 
(where the last map is the forgetful map) is an isomorphism. This follows from the fact that $\MM_{0,5}$
is isomorphic to the blow-up of $\bP^2$ in $4$ general points, say $p, p_1, p_2, p_3$, and
the forgetful map ($=$ universal family) $\MM_{0,5}\to\MM_{0,4}$ is obtained by choosing a pencil of lines through $p$.
The four sections are $E_p$ and the proper transforms of lines through $p_1$, $p_2$, and $p_3$.
Our construction gives the same family. The claim about pull-backs of boundary divisors follows by definition
of the boundary divisors (see also similar Theorem~\ref{blowupdescr}).
\end{proof}

\begin{thm}\label{rigid map}
The map $f: \PP^1\ra\MM_{0,n}$ is rigid.
\end{thm}

\begin{proof}
Assume that there is a smooth curve germ $(t_0\in T)$ and 
maps $\Pi: \cC\ra T$, $F:\cC\ra\MM_{0,n}$ such that we have: 
$$(F_{t_0}: \cC_{t_0}\ra\MM_{0,n})\cong (f:\PP^1\ra\MM_{0,n}).$$ 

Let $\cU\ra\cC$ be the pull-back of the universal family of $n$-pointed stable curves over $\cC$,  with sections $\sigma_1,\ldots, \sigma_{n-3}, \sigma_u, \sigma_v, \sigma_p:\cC\ra\cU$.
The family $\cU\ra T$ gives a deformation of the surface $S$ and we may assume (by shrinking $T$)
that $\cU$ is smooth over $T$. By applying repeatedly \cite[Prop. IV. (3.1), p. 121]{BPV}, it follows 
(after shrinking $T$) that the surface $\cU_t$ is a blow-up of $\PP^2$, with the exceptional divisors fitting in a flat family over $T$. More precisely, for every $t\in T$ the surface $\cU_t$ is a blow-up of $\PP^2$ along distinct points $p^t$, $p_j^t$ ($j=1\ldots, l$) and two infinitely near points $q_1^t$, $q_2^t$, such for $t=t_0$ this coincides with our initial configuration
$$p^{t_0}=p, \quad p_j^{t_0}=p_j, \quad q_i^{t_0}=q_i.$$
Moreover, there are divisors $\cE_1,\ldots, \cE_l, \cE_p, \cE'_1,\cE'_2$ in $\cU$ such that for each $t\in T$, 
$$\cE_j^t=\cE_j\cap\cU_t,\quad  \cE_p^t=\cE_p\cap\cU_t,\quad {\cE'}_i^t=\cE'_i\cap\cU_t$$
are the exceptional divisors corresponding to the points $p^t$, $p_j^t$, $q_i^t$.

For each $t\in T$, let $L_i^t$, $C_1^t$, $C_2^t$ be the images in $\PP^2$ of the sections $\sigma_i^t$, $\sigma_u^t$, $\sigma_v^t$. The intersection numbers $(\sigma_i^t.\cE_j^t)$,
$(\sigma_u^t.\cE_j^t)$, $(\sigma_v^t.\cE_j^t)$ do not depend on $t$, hence, each of the curves 
$L_i^t$, $C_1^t$, $C_2^t$ contains the point $p_j^t$ if and only if this happens for $t=t_0$. (Moreover, the multiplicity is $1$ if this happens.) Moreover, as $\cO(\sigma_i)$ is flat over $T$,  the self-intersection number $(\sigma_i^t)^2$ is constant in the family and it follows that  $L_i^t$ is a line and $C_1^t$ and $C_2^t$ are conics. It is clear now that 
 the lines $L_1^t,\ldots, L_{n-3}^t$ have the same incidence, for all $t\in T$. 
\end{proof}



\section{Arithmetic break of a hypergraph curve}\label{break of hypergraph curve}

We will show how the rigid curve $C$ constructed in Section \ref{HesseSection}, breaks in 
characteristic $3$ into several components. We compute the numerical classes of these components and use this to prove that the class of $C$ is a sum of $F$-curves. We keep the notations from Section \ref{HesseSection}.
Starting with this section, all schemes will be $\bZ$-schemes (including $\MM_{0,n}$).

\Trick{Set-up.} 

Let $\om\in\bC$ be a primitive cubic root of $1$ and let $R=\ZZ[\om]$ 
be the ring of Eisenstein integers. Let $K=\QQ[\om]$.
The Hesse configuration is defined over $K$ and we can choose coordinates $X,Y,Z$ in $\PP^2$ such that the $12$ points have coordinates:
$$m=(1,0,0), \quad n=(0,0,1), \quad p=(0,1,0),$$
$$a=(\om,\om, 1), \quad b=(1,\om,1), \quad c=(\om^2,\om,1),$$
$$1=(1,\om,1),\quad 2=(\om^2,\om^2,1), \quad 3=(\om,\om^2,1),$$
$$\al=(\om,1,1),\quad \be=(\om^2,1,1), \quad \ga=(1,1,1).$$
We view these points as sections of $\bP^2_R$.

Let $\cC$ denote the smooth conic bundle (over $R$) 
$$\om XY+ \om XZ+ YZ=0.$$
It contains sections $m,n,p,1,a$.
Note that $\cC$ has a parametrization given by:  
$$\PP^1_R\cong\cC,\quad (u,v)\mapsto (\om u^2+uv, \om uv+v^2, -\om uv).$$

Our curve $C$ is the characteristic $0$ fiber of $\cC$ (base-changed to $\CC$).

\Trick{Break of the curve $C$ in characteristic $3$ (outline).}
Consider the characteristic $3$ fiber $\PP^1_{\FF_3}$ of $\PP^1_R\ra\Spec R$ at $(\om-1)\in\Spec R$.
Note that at this fiber sections $a,b,c,1,2,3,\al,\be,\ga$ pass through $(1,1,1)$.
Consider the rational map: 
$$\PP^1_R\cong\cC\dra\M_{0,12}.$$
In order to resolve this map, one has to blow-up the arithmetic surface 
$\PP^1_R$ several times along $\PP^1_{\FF_3}$.
We now outline the strategy. First, we blow-up $\PP^1_R$ at the point $u=v=1$ in the fiber $\PP^1_{\FF_3}$. Let the corresponding exceptional divisor be $E_1$ and let $F$ denote the proper transform of $\PP^1_{\FF_3}$. We blow-up one more point on $E_1$, resulting in an exceptional curve $E_2$. We also blow-up the intersection point of $E_1$ and $F$ and let $E_3$ be the exceptional curve (see Fig.~\ref{pic_a}).
\begin{figure}[htbp]
\includegraphics[width=4in]{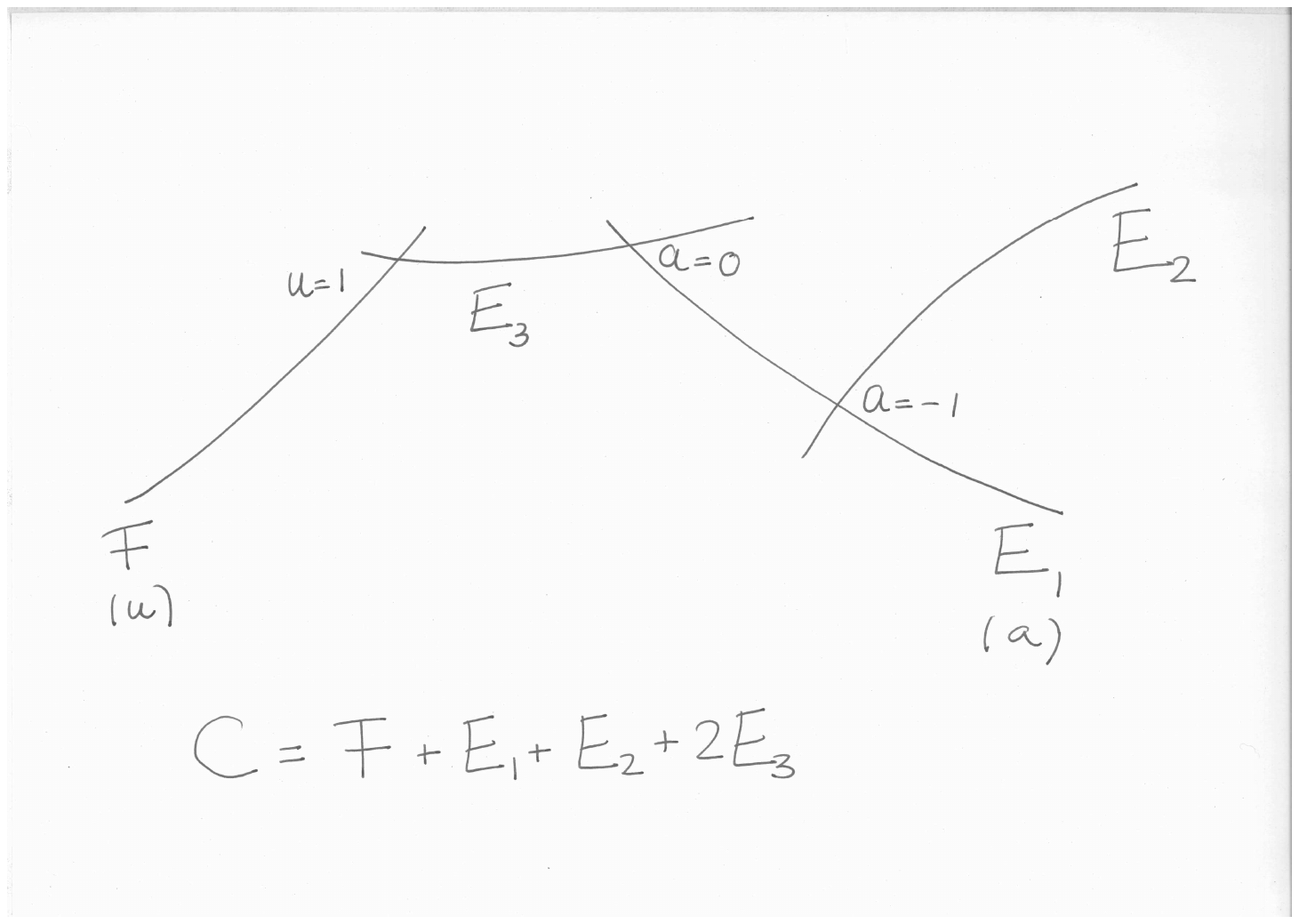}
\caption{Components of the characteristic $3$ fiber.}\label{pic_a}
\end{figure}

We let $T$ be the resulting arithmetic surface. We will abuse notations and denote by $F$ and $E_1$ the proper transforms of the respective curves in $T$.

We construct a family $\pi: S\ra T$ with twelve sections, such that over an open set 
$T^0\subseteq T$ the sections are disjoint, and thus define a map $T^0\ra\M_{0,12}$. 
Moreover, we can enlarge $T^0$ such that that its intersection with each of $F$, $E_i$ is non-empty. Simply blow-up the total space $S$ along the sections that become equal over the generic points of these curves. The maps $T^0\cap F\ra\MM_{0,12}$, $T^0\cap E_i\ra\MM_{0,12}$ extend uniquely to morphisms: 
$$F\ra\MM_{0,12},\quad E_i\ra\MM_{0,12}.$$

From the universal family $S\ra T^0$ restricted to $T^0\cap F$, $T^0\cap E_i$ we can determine the classes of the curves $F$, $E_i$. (We use here that the class of a curve $B\ra\M_{0,n}$ is determined by the universal family over an open set of $B$). One will eventually have to do further blow-ups to resolve the map $\PP^1_R\dra\M_{0,12}$, but since one can check that one has an equality of numerical classes:
$$C=F+E_1+E_2+2E_3,$$ 
this proves that any other extra components in the characteristic $3$ fiber will map constantly to $\M_{0,12}$. Note that the characteristic $3$ fiber of $T\ra\Spec R$ contains $E_3$ with multiplicity $2$, since we blow-up a node of the fiber. We then prove that each of the classes of these curves is a sum of $F$-curves by using Prop. \ref{lines+exc}.

\Trick{Universal family over an open set in $\PP^1_R$.\label{equations I}}
Let $X', Y', Z'$ be coordinates on the dual projective plane $\cPP^2_R$. The incidence variety in $\PP^2_R\times\cPP^2_R$, with equation$XX'+YY'+ZZ'=0$
parametrizes pencils of lines through points in $\PP^2_R$. We consider 
the subvariety $\cI\subseteq \PP^1_R\times\cPP^2_R$ 
of pencils of lines through points of $\cC\cong\PP^1_R$: 
$$\cI: \quad (\om u^2+uv)X'+(\om uv+v^2)Y' -\om uvZ'=0.$$

The first projection $\cI\ra\PP^1_R$ is a $\PP^1$-bundle. Each point $q$ of the $12$ points in the dual Hesse configuration defines a rational section $s_q$. If $q\neq m,n,p,1,a$, then one simply has $s_q=\cI\cap (\PP^1_R\times L_q)$, where $L_q\subseteq\cPP^2$ is the line dual to the point $q$. If $q\in\{m,n,p,1,a\}$, then one has to discard the fiber at $q$.
Note that over a general point in $\PP^1_K$ the sections are disjoint.  

One obtains a simpler description of the universal family as follows. From now on we will work in the chart $v=1$ on $\PP^1_R$. Each section $s_q$ gives a rational map $\AA^1_R\dra\PP^2_R$. Composing with the projection $\PP^2_R\dra\PP^1_R$ from the point $(0,0,1)$, 
one obtains a family over $\AA^1_R$ that defines the same map $\PP^1_R\dra\MM_{0,12}$. This 
is simply
$$S=\AA^1_R\times\PP^1, \quad \text{  with the projection }\quad \pi: S\ra \AA^1_R,$$ 
with sections given by the following equations. We denote by $X', Y'$ the coordinates on $\PP^1$ (with $u$ as before the coordinate on $\AA^1_R$):
\begin{align*}
m &:\quad X'=0,\\
n &:\quad Y'=-uX',\\
p &:\quad Y'=0,\\
a &:\quad Y'=\om uX',\\
b &:\quad (1-u)Y'=\om u(\om-u)X',\\
c &:\quad (1-u)Y'=-u(\om u+2)X',\\
1 &:\quad Y'=\om^2 u X',\\
2 &:\quad (\om^2 u-1)Y'=u(\om u+2)X',\\
3 &:\quad  (\om^2 u-1)Y'=\om u(u-1)X',\\
\al &:\quad (2\om u+1)Y'=\om u(1-u)X',\\
\be &:\quad (2\om u+1)Y'=-u(\om u+2)X',\\
\ga &:\quad (2\om u+1)Y'=\om u(\om-u)X'.
\end{align*}

\Trick{The component $F$ of the characteristic $3$ fiber.}
As $\om=1$, all sections but $m, n, p$, become equal,  given by equation $Y'=uX'$. We blow-up the total space $S$ along $\om=1, Y'=uX'$. Locally, in coordinates, we have: $$Y'=uX'+(\om-1)Y_1,$$ 
with the exceptional divisor cut by $\om=1$ and introducing a new coordinate $Y_1$. The proper transforms of the sections that intersect this chart have equations:
\begin{align*}
a &:\quad Y_1=uX',\\
b &:\quad (1-u)Y_1=u(\om+1-u)X',\\
c &:\quad (1-u)Y_1=u(\om+2-u)X',\\
1 &:\quad Y_1=(\om+1)u X',\\
2 &:\quad (\om^2 u-1)Y_1=-u(\om u+\om+2)X',\\
3 &:\quad  (\om^2 u-1)Y_1=-u(\om u+1)X',\\
\al &:\quad (2\om u+1)Y_1=u\big(1+u(\om-1)\big)X',\\
\be &:\quad (2\om u+1)Y_1=u(\om+2)(\om u+1)X',\\
\ga &:\quad (2\om u+1)Y_1=u\big(\om+1+u(\om-1)\big)X'.
\end{align*}
The ``attaching" section (call this $y$) has equation $X'=0$. For general $u$, the twelve sections are distinct, and we obtain the universal family over the proper transform $F$ of $\PP^1_{\FF_3}$.
A general point in $F$ parametrizes a curve in the boundary 
$$\de_{mnp}=\MM_{0,10}\times\MM_{0,4}=
\MM_{0,\{a,b,c,1,2,3,\al,\be,\ga,y\}}\times\MM_{0,\{m,n,p,y\}}.$$ 

It is easy to see that the cross-ratio of sections $m,n,p, y$ do not depend on $u$, thus 
the projection of $F$ onto $\MM_{0,4}$ is constant. Thus the class of $F$ is given by the 
class of the curve in $\MM_{0,10}$ obtained by making $\om=1$ in the above equations. 

An easy computation gives that the class of $F$ is given by: 
\begin{align*}
F& =\De_{2,c,\be}+\De_{3,a,\al}+\De_{1,b,\ga}+\De_{1,a,\be}+\De_{2,b,\al}+\De_{3,c,\ga}+\\
&+\De_{3,b,\be}+\De_{1,c,\al}+\De_{2,a,\ga}+\De_{1,2,3}+\De_{a,b,c}+\De_{\al,\be,\ga}-\De_{m,n,p}.
\end{align*}
(where we use Notation \ref{formal}). Note that $F\cdot\de_{m,n,p}=-1$ since 
$F\cdot\psi_y=1$ on $\MM_{0,10}$.

\Trick{The first blow-up.\label{Hesse 1st blow-up}}
We blow-up $\PP^1_R$ at the point $u=1$ in $\AA^1_{\FF_3}$, i.e., along $\om=1, u=1$. Locally, in coordinates, we have 
$u-1=(\om-1)a$, with exceptional divisor $E_1: \om=1$ and new coordinate $a$. 

The following is an argument that we will repeat several times in what follows. The family $S\ra\PP^1_R$ pulls back to give an arithmetic threefold over $T$, which is itself a blow-up of $S$. By abuse of notations, we will keep denoting this with $S$. The proper transforms of the twelve sections in (\ref{equations I}) give (rational) sections of the new map $S\ra T$, with equations:
\begin{align*}
m &:\quad X'=0,\\
n &:\quad Y'=-X'\big(1+(\om-1)a\big),\\
p &:\quad Y'=0,\\
a &:\quad Y'=X'\om\big(1+(\om-1)a\big),\\
b &:\quad Y'a=X'\om(a-1)\big(1+(\om-1)a\big),\\
c &:\quad Y'a=X'\om(\om+a)\big(1+(\om-1)a\big),\\
1 &:\quad Y'=X'\om^2\big(1+(\om-1)a\big),\\
2 &:\quad Y'\om(a-1)=X'(a+\om)\big(1+(\om-1)a\big),\\
3 &:\quad Y'\om(a-1)=X'a\big(1+(\om-1)a\big),\\
\al &:\quad Y'(1-2a)=X'a\big(1+(\om-1)a\big),\\
\be &:\quad Y'(1-2a)=X'(a+\om)\big(1+(\om-1)a\big),\\
\ga &:\quad Y'(1-2a)=X'(1-a)\big(1+(\om-1)a\big).
\end{align*}

Along $E_1: \om=1$, the sections $a, 1, \be$ become equal to $Y'=X'$. By blowing-up the total space along  $\om=1, Y'=X'$, the $12$ sections become distinct above the generic point of $E_1$. The curve $E_1$ thus lies in the boundary $$\de_{a1\be}\cong\MM_{0,10}\times\MM_{0,4}.$$ Making $\om=1$ in the above equations, allows one to compute the class of its first projection $E'_1\subset\MM_{0,10}$. As a curve in
$\M_{0,12}$ it is given by:
\begin{align*}
E'_1&=\De_{m,b,c}+\De_{p,3,\al}+\De_{n,2,\ga}+\De_{m,2,3}+\De_{p,b,\ga}+
\De_{n,c,\al}+\\
&+\De_{m,\al,\ga}+\De_{p,2,c}+\De_{n,3,b}+
\De_{2,b,\al}+\De_{3,c,\ga}+\De_{mnp}-\De_{1,a,\be}.
\end{align*}

The second projection $E''_1\subset\MM_{0,4}$ is an $F$-curve with class:
$$E''_1=-\De_{1,a,\be}+\De_{1,a}+\De_{1,\be}+\De_{a,\be}.$$

Then $E_1=E'_1+E''_1$ has numerical class:
\begin{align*}
E_1&=\De_{m,b,c}+\De_{p,3,\al}+\De_{n,2,\ga}+\De_{m,2,3}+\De_{p,b,\ga}+
\De_{n,c,\al}+\De_{m,\al,\ga}+\\
&+\De_{p,2,c}+\De_{n,3,b}+
\De_{2,b,\al}+\De_{3,c,\ga}+\De_{mnp}-2\De_{1,a,\be}+\De_{1,a}+\De_{1,\be}+\De_{a,\be}.
\end{align*}

\Trick{The second blow-up.} In the notations of 
(\ref{Hesse 1st blow-up}) we further blow-up $\PP^1_R$ at the point $a=-1$ in $E_1$, i.e., along $\om=1, a=-1$. Locally, in coordinates, we have $a+1=(\om-1)b$, with exceptional divisor $E_2: \om=1$ and new coordinate $b$. The proper transforms of the sections have equations:
\begin{align*}
m &:\quad X'=0,\\
n &:\quad Y'=-X'\big((\om-1)^2b-\om+2\big),\\
p &:\quad Y'=0,\\
a &:\quad Y'=X'\om\big((\om-1)^2b-\om+2\big),\\
b &:\quad Y'\big((\om-1)b-1\big)=X'\om\big((\om-1)b-2\big)\big((\om-1)^2b-\om+2\big),\\
c &:\quad Y'\big((\om-1)b-1\big)=X'\om(\om-1)(b+1)\big((\om-1)^2b-\om+2\big),\\
1 &:\quad Y'=X'\om^2\big((\om-1)^2b-\om+2\big),\\
2 &:\quad Y'\om\big((\om-1)b-2\big)=X'(\om-1)(b+1)\big((\om-1)^2b-\om+2\big),\\
3 &:\quad Y'\om\big((\om-1)b-2\big)=X'\big((\om-1)b-1\big)\big((\om-1)^2b-\om+2\big),\\
\al &:\quad Y'\big(3-2(\om-1)b\big)=X'\big((\om-1)b-1\big)\big((\om-1)^2b-\om+2\big),\\
\be &:\quad Y'(-2b+\om^2-1)=X'(b+1)\big((\om-1)^2b-\om+2\big),\\
\ga &:\quad Y'\big(3-2(\om-1)b\big)=X'\big((\om-1)b-2\big)\big((\om-1)^2b-\om+2\big).
\end{align*}

Along $E_2: \om=1$, one has:
\begin{align*}
m=\al=\ga &:\quad X'=0,\\
n=b=3 &:\quad Y'=-X',\\
p=c=2 &:\quad Y'=0,\\
a=1 &:\quad Y'=X',\\
\be &:\quad Y'b=X'(b+1).
\end{align*}
\begin{figure}[htbp]
\includegraphics[width=4in]{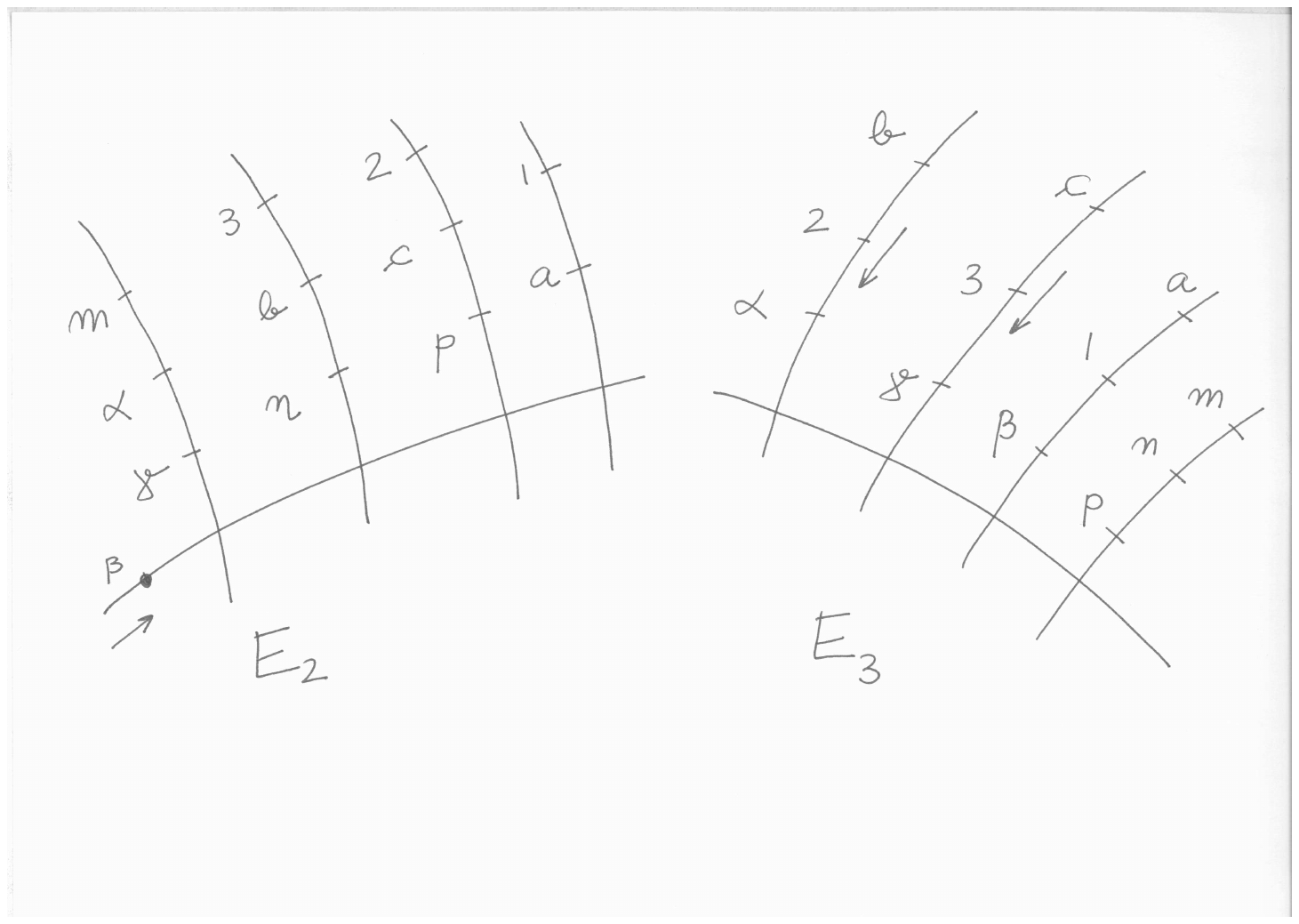}
\caption{Components $E_2$ and $E_3$ of the characteristic $3$ fiber.}\label{pic_b}
\end{figure}

Blowing-up the total space along the above loci (where some of the sections become equal along $E_2$), the twelve sections become disjoint above the generic point of $E_2$. See also Fig.~\ref{pic_b}. The curve $E_2$ has numerical class:
$$E_2=\De_{m,\al,\be,\ga}+\De_{p,2,c,\be}+\De_{3,b,n,\be}+
\De_{1,a,\be}-\De_{m,\al,\ga}-\De_{p,2,c}-\De_{n,3,b}-\De_{1,a}.$$

\Trick{The third blow-up.} In the notations of 
(\ref{Hesse 1st blow-up}) we further blow-up $\PP^1_R$ at the point $E_1\cap F$, i.e., along $\om=1, a=\infty$. By passing to the other chart of the blow-up in (\ref{Hesse 1st blow-up}), if we let $s=\frac{1}{a}$  (thus $\om-1=(u-1)s$), we blow-up the point $u=1, s=0$. Locally, in coordinates, we have $s=(u-1)t$, with exceptional divisor $E_3: u=1$ and new coordinate $t$. The proper transforms of the twelve sections have equations:
\begin{align*}
m &:\quad X'=0,\\
n &:\quad Y'=-X'u,\\
p &:\quad Y'=0,\\
a &:\quad Y'=X'\om u,\\
b &:\quad Y'=X'\om u\big(1-(u-1)t\big),\\
c &:\quad Y'=X'u\big(u(u-1)t-(\om+2)(u-1)t+1\big),\\
1 &:\quad Y'=X'\om^2 u,\\
2 &:\quad Y'\big(1-(u-1)t\big)\om=X'u\big(1+\om(u-1)t\big),\\
3 &:\quad Y'\big(1-(u-1)t\big)\om=X'u,\\
\al &:\quad Y'\big((u-1)t-2\big)=X'u,\\
\be &:\quad Y'\big((u-1)t-2\big)=X'u\big(1+\om(u-1)t\big),\\
\ga &:\quad Y'\big((u-1)t-2\big)=X'u\big(1-(u-1)t\big).\\
\end{align*}

Along $E_3: u=1$, the sections $a,b,c,1,2,3,\al,\be,\ga$ become equal to $Y'=X'$. We blow-up the total space along $u=1, Y'=X'$. Locally, in coordinates, we have $Y'=X'+(u-1)Y_1$, 
with the exceptional divisor cut by $\om=1$ and new coordinate $Y_1$.
The proper transforms of the nine sections have equations:
\begin{align*}
a &:\quad Y_1=X'\big(u(u-1)t+1\big),\\
b &:\quad Y_1=X'\big(-\om u+(u-1)t+1\big),\\
c &:\quad Y_1=X'\big(u^2t-(\om+2)ut+1\big),\\
1 &:\quad Y_1=X'\big(\om^2+(\om+1)(u-1)t\big),\\
2 &:\quad Y_1\big(1-(u-1)t\big)\om=X'\big(\om ut+\om t+1-(u-1)t\big),\\
3 &:\quad Y_1\big(1-(u-1)t\big)\om=X'\big(\om t-(u-1)t+1\big),\\
\al &:\quad Y_1\big((u-1)t-2\big)=X'\big(1-t+(\om+2)(u-1)t\big),\\
\be &:\quad Y_1\big((u-1)t-2\big)=X'\big(\om ut-t+1-(\om+2)(u-1)t\big),\\
\ga &:\quad Y_1\big((u-1)t-2\big)=-X'\big(ut+t-1+(\om+2)(u-1)t\big).\\
\end{align*}
The ``attaching section" is $X'=0$.  Along $u=1$ the sections become:
\begin{align*}
b=2=\al &:\quad Y'=(1-t)X',\\
c=3=\ga &:\quad Y'=(1+t))X',\\
a=1=\be &:\quad Y'=X'.
\end{align*}

By blowing-up the total space along the above loci (where some of the sections become equal along $E_3$), the twelve sections become disjoint above the generic point of $E_3$. See also Fig.~\ref{pic_b}. The curve $E_3$ is thus containd in several boundary components:
$$E_3\subset \de_{2,b,\al}\cap\de_{3,c,\ga}\cap\de_{1,a,\be}\cap\de_{m,n,p}.$$

From the blow-up of thes above loci, one can see that the only cross-ratios that 
change with $t$ are the ones coming from the triples $b, 2, \al$ and $c, 3, \ga$.
The curve $E_3$ is thus the sum of two $F$-curves:
$$E_3=\big(-\De_{2,b,\ga}+\De_{2,b}+\De_{2,\ga}+\De_{b,\ga}\big)+
\big(-\De_{3,c,\ga}+\De_{3,c}+\De_{3,\ga}+\De_{c,\ga}\big).$$

\Trick{The classes $F$, $E_i$ are sums of $F$-curves.\label{example: sums of F-curves}} 

Recall that $F$-curve classes are represented by curves as in Figure \ref{pic_k}.
\begin{figure}[htbp]
\includegraphics[width=2in]{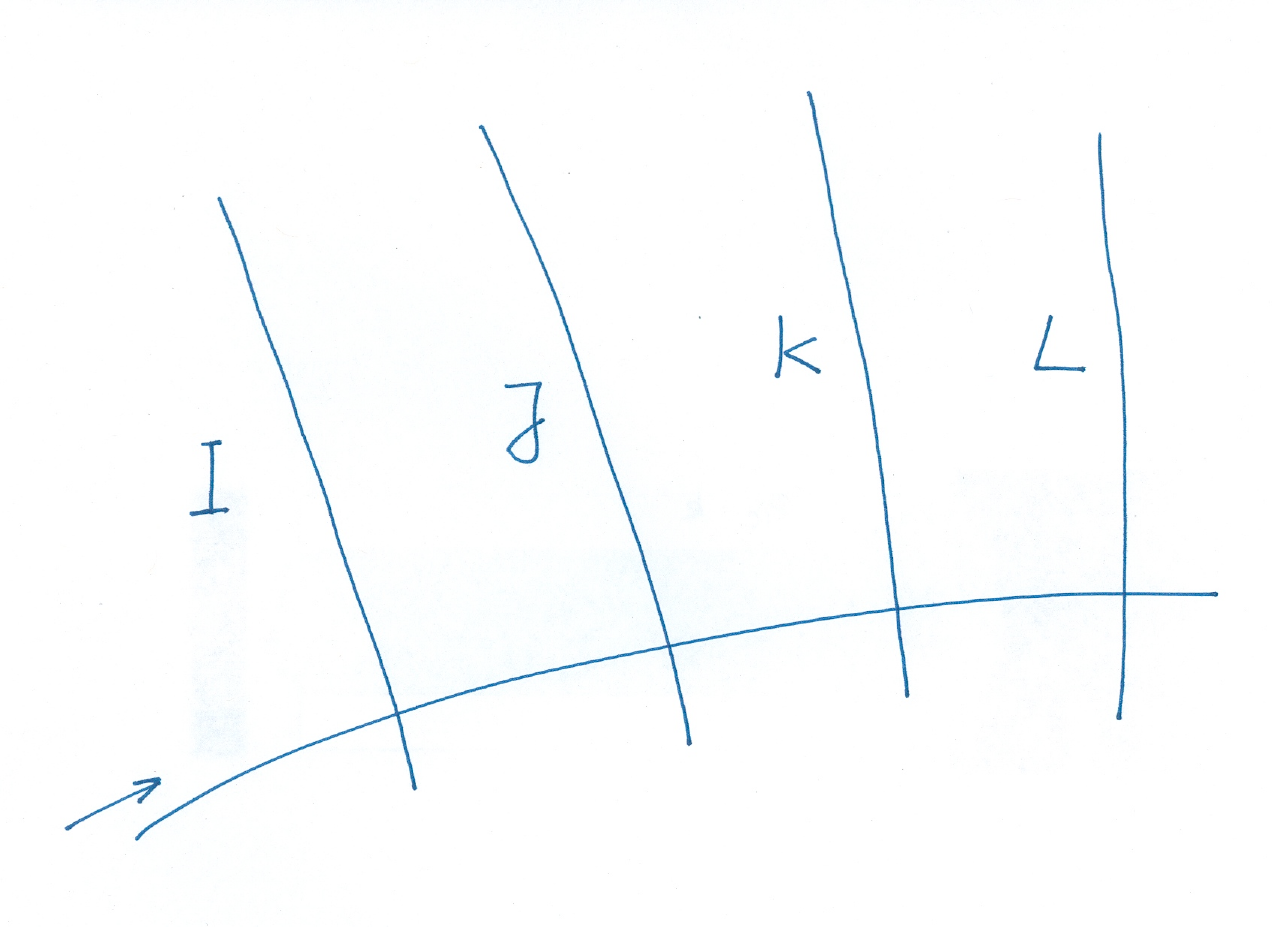}
\caption{$F$-curves are given by a partition $I, J, K, L\neq\emptyset$ of 
$\{1,\ldots, n\}$.}\label{pic_k}
\end{figure}

Note that the markings from $I, J, K, L$ stay fixed. Such a curve has class:
$$-\De_I-\De_J-\De_K-\De_L+\De_{I\cup J}+\De_{I\cup K}+\De_{I\cup L},$$
(with the convention that we omit the terms $\De_I$ if $|I|=1$). Perhaps the easiest curves that
are sums of $F$-curves are components of fibers of forgetful maps $\M_{0,n}\ra\M_{0,n-1}$ that forget one marking $p\in\{1,\ldots, n\}$: 

\begin{figure}[htbp]
\includegraphics[width=4in]{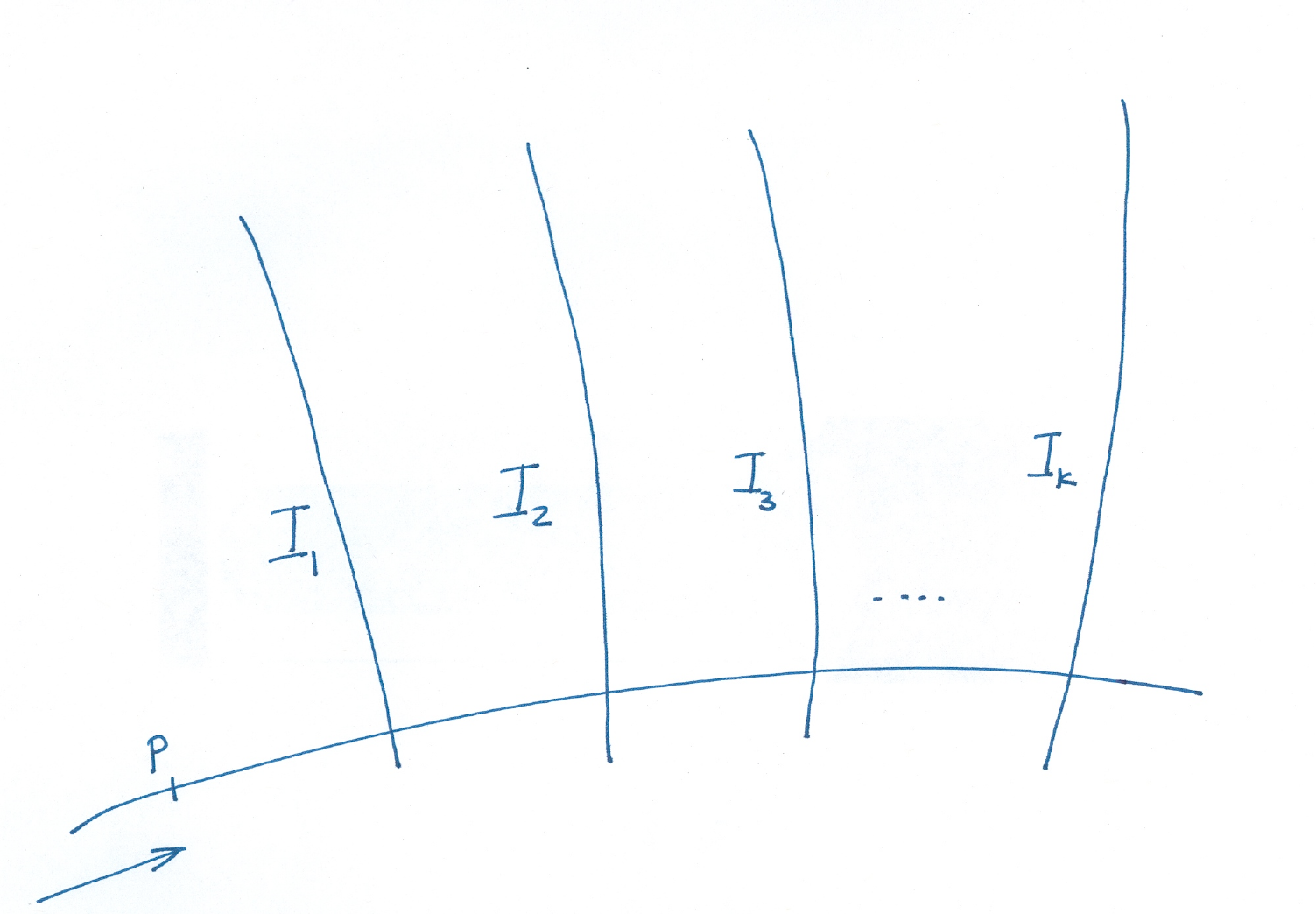}
\caption{}\label{pic_l}
\end{figure}

\begin{claim}\label{fiber comp}
Let $\{p\}\cup I_1\cup\ldots \cup I_k$ ($k\geq3$) be a partition of the set $\{1,\ldots, n\}$
and let $B$ be the curve in $\M_{0,n}$ given as in Figure \ref{pic_l} (where the markings in  $I_1,\ldots, I_k$ stay fixed, as do their attaching points,  and $p$ is the only moving point). 
Then $B$  has class:
$$B=-\De_{I_1}-\ldots-\De_{I_k}+\De_{I_1\cup \{p\}}+\ldots+\De_{I_k\cup \{p\}},$$
(we omit the terms $\De_{I_j}$ if $|I_j|=1$). Moreover, $B$ is a sum of $(k-2)$ $F$-curves.
\end{claim}

\begin{proof}
This is a straightforward computation. If we denote by $y_j$ the attaching point corresponding to the component with markings from $I_j$, note that $B$ comes from a curve in 
$\M_{0,k+1}=\M_{0,\{p,y_1,\ldots, y_k\}}$.  Then $B\cdot\De_{I_j}$ equals 
$-B\cdot\psi_{y_j}$ (on $\M_{0,k+1}$), and thus $B\cdot\De_{I_j}=-1$. One can check directly that $B$ is the sum of the following $F$-curves corresponding to the partitions:
$$\{p\},\quad I_1,\quad I_2,\quad I_3\cup\ldots\cup I_k,$$
$$\{p\},\quad I_1\cup I_2,\quad I_3,\quad I_4\cup\ldots\cup I_k,$$
$$\ldots$$
$$\{p\},\quad I_1\cup\ldots\cup I_{k-2},\quad I_{k-1},\quad I_k.$$
\end{proof}

We now prove that the classes $F, E_i$ are sums of $F$-curves. Clearly, $E_3$ is the sum of two $F$-curves. The curve $E_2$ is a sum of two $F$-curves by Claim \ref{fiber comp} (note also that $E_2$ comes from a curve in $\MM_{0,5}$, thus a sum of $F$-curves by Cor. \ref{F-conj}).

For the curves $F$ and $E_1$, we will use Prop. \ref{lines+exc}. Note that the two curves have numerical classes equal to classes of (proper transforms of) lines in surfaces $\Bl_{p_1,\ldots,p_n}\PP^2$ as in Thm. \ref{blowupdescr}. To see this, consider the
configuration of all $\FF_3$-rational points in $\PP^2_{\FF_3}$ except for $(2,1,0)$:
$$m=(1,0,0), \quad n=(1,1,0), \quad p=(0,1,0),$$
$$a=(1,1, 1), \quad b=(0,1,1), \quad c=(2,1,1),$$
$$1=(0,2,1),\quad 2=(2,2,1), \quad 3=(1,2,1),$$
$$\al=(1,0,1),\quad \be=(2,0,1), \quad \ga=(0,0,1).$$

The configuration has the same pairs of points collinear as the Hesse configuration. In addition, the following points give concurrent lines (see Figure \ref{pic_s}):
$$mnp, 1a\be, 2b\al, 3c\ga.$$
\begin{figure}[htbp]
\includegraphics[width=4in]{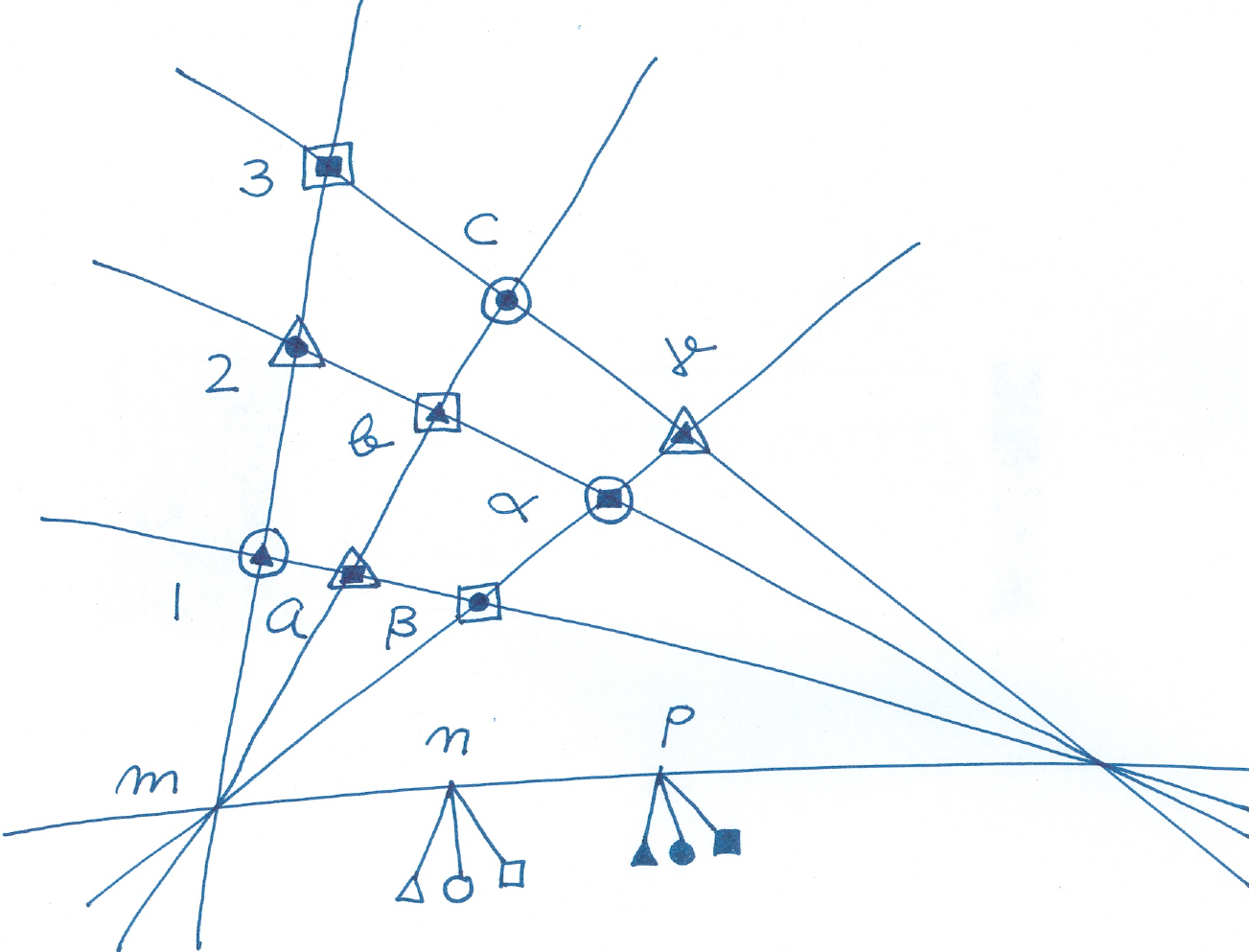}
\caption{ A configuration of $12$ points in $\PP^2_{\FF_3}$. }\label{pic_s}
\end{figure}

Let $S$ be the blow-up $S$ of $\PP^2_{\overline{\FF}_3}$ at the above twelve points.
Thm. \ref{blowupdescr} gives a map:
$$S\ra\MM_{0,12;\overline{\FF}_3},$$ 

Theorem \ref{blowupdescr} allows one to compute the class of any curve in $S$. 
It is straightforward to check that the class of the proper transform of the line
$1a\be$ equals the class of $E_1$. We will use this to prove that $E_1$ is a sum of $F$-curves.

Moreover, the line $mnp$ lies in the boundary component
$$\de_{mnp}\cong\MM_{0,10}\times\MM_{0,4}.$$
Taking its projection onto $\MM_{0,10}$ and embedding it again in $\MM_{0,12}$ (attach a fixed  $\PP^1$ marked by $m,n,p$) gives a curve with the same class as $F$. 
One can also argue geometrically: if we blow-up $\PP^2_R$ at the point $(1,1,1)\in\PP^2_{\FF_3}$, the exceptional divisor $A$ is isomorphic to $\PP^2_{\FF_3}$ and the proper transforms of the $R$-points of the Hesse configuration intersect $A$ at the above points. When resolving the map $$\PP^2_R\dra\MM_{0,12},$$ the line $1a\be$ is  the component $E_1$ in the proper transform of our conic $\cC$. Similarly, one can express the class of the line $mnp$ in terms of the class of $F$ using the geometry of the ruled surface which is the proper transform of $\PP^2_{\FF_3}$.

\Trick{The class $E_1$ is a sum of $F$-curves.} 

We now use the proof of Prop. \ref{lines+exc} (Case II), as $E_1$ is the class of the line $1a\be$ in $\PP^2_{\overline{\FF}_3}$. 
Since $E_1\subset \de_{1a\be}$,  the curve $E_1$ is a sum of its projections $E'_1$ in $\M_{0,10}$ and $E''_1$ in $\M_{0,4}$ (see \ref{Hesse 1st blow-up}). Note that $E''_1$ is an $F$-curve; hence, we are left to write $E'_1$ as a sum of $F$-curves. 
As indicated in the proof of Prop. \ref{lines+exc}, we remove the points $1, a, \be$ from $\PP^2_{\overline{\FF}_3}$ and place an extra point $x_1$ at a general point of the line $1a\be$. Repeating the construction of 
Theorem \ref{blowupdescr} for this new configuration, we obtain a curve in 
$$\M_{0,10}=\M_{0,\{2,3,b,c,\al,\ga, m, n, p, x_1\}},$$ 
corresponding to the proper transform of a general line $L_1$ through $x_1$. The curve $E'_1$ in $\M_{0,12}$ is obtained from $L_1$ by adding at $x_1$ an extra component $\PP^1$ with markings $1, a, \be$ (and no moduli). The class of $L_1$ in the blow-up of $\PP^2_{\overline{\FF}_3}$ at the points $2, 3, b, c,\al,\ga, m, n, p, x_1$ is the sum of the class of the (proper transform of the) line $2x_1$ and the exceptional divisor $B_2$ corresponding to the point $2$. The class of $B_2$ in $\M_{0,10}$ is given by:
$$-\De_{m,3}-\De_{p,c}-\De_{n,\ga}-\De_{b,\al}+\De_{m,2,3}+
\De_{2,p,c}+\De_{2,n,\ga}+\De_{2,b,\al}+\De_{2,x_1}.$$

To see this, we use again the proof of Prop  \ref{lines+exc} (Case I): The curve $B_2$ is a component of the forgetful map $\M_{0,10}\ra\M_{0,9}$ that forgets the marking $2$. 
The point in $\M_{0,9}$ to which $B_2$ maps is determined by the cross-ratio of the lines joining $2$ with the other points. 
The class of $B_2$ in $\M_{0,12}$ is given by (see Figure \ref{pic_m}):
$$B_2=-\De_{m,3}-\De_{p,c}-\De_{n,\ga}-\De_{b,\al}+\De_{m,2,3}+
\De_{2,p,c}+\De_{2,n,\ga}+\De_{2,b,\al}+\De_{1,2,a,\be}-\De_{1,a,\be}.$$
\begin{figure}[htbp]
\includegraphics[width=3in]{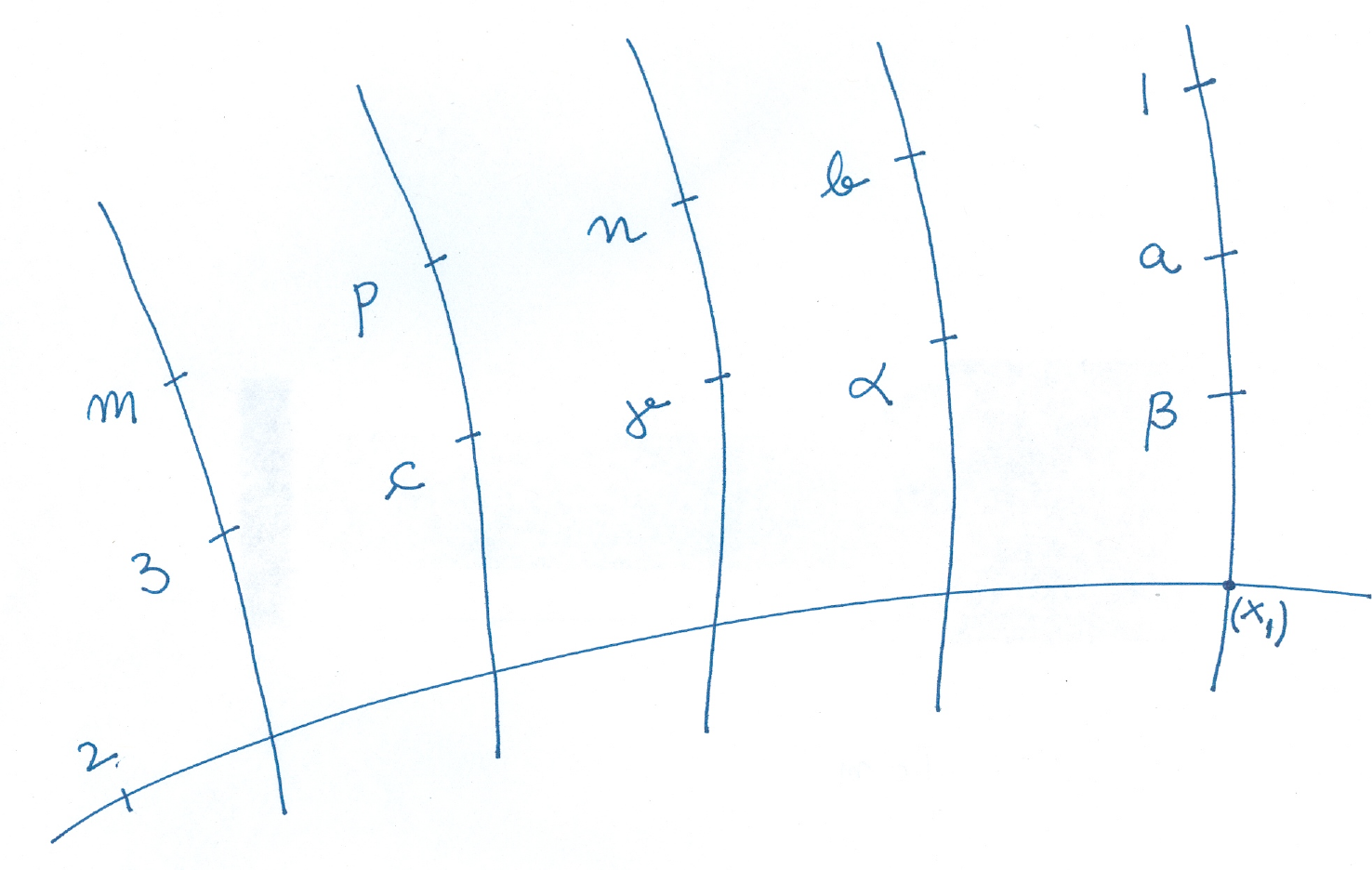}
\caption{ The curve $B_2$ in $\M_{0,12}$.}\label{pic_m}
\end{figure}

In order to find the class of the line $2x_1$, we repeat the argument. As before, we further remove the points $2, x_1$ and place an extra point $x_2$ at a general point of the line $2x_1$. Repeating the construction of Theorem \ref{blowupdescr} for this new configuration, we obtain a curve in 
$$\M_{0,9}=\M_{0,\{3,b,c,\al,\ga, m, n, p, x_2\}},$$ 
corresponding to the proper transform of a general line $L_2$ through $x_2$. The class of $L_2$ in the blow-up of $\PP^2_{\overline{\FF}_3}$ at the points $3, b, c,\al,\ga, m, n, p, x_2$ is the sum of the class of the (proper transform of the) line $3x_2$ and the exceptional divisor $B_3$ corresponding to the point $3$. The class of $B_3$ in $\M_{0,9}$ is given by:
$$-\De_{p,\al}-\De_{b,n}-\De_{c,\ga}+\De_{3,p,\al}+
\De_{3,b,n}+\De_{3,c,\ga}+\De_{3,m}+\De_{3,x_2}.$$

\begin{figure}[htbp]
\includegraphics[width=4in]{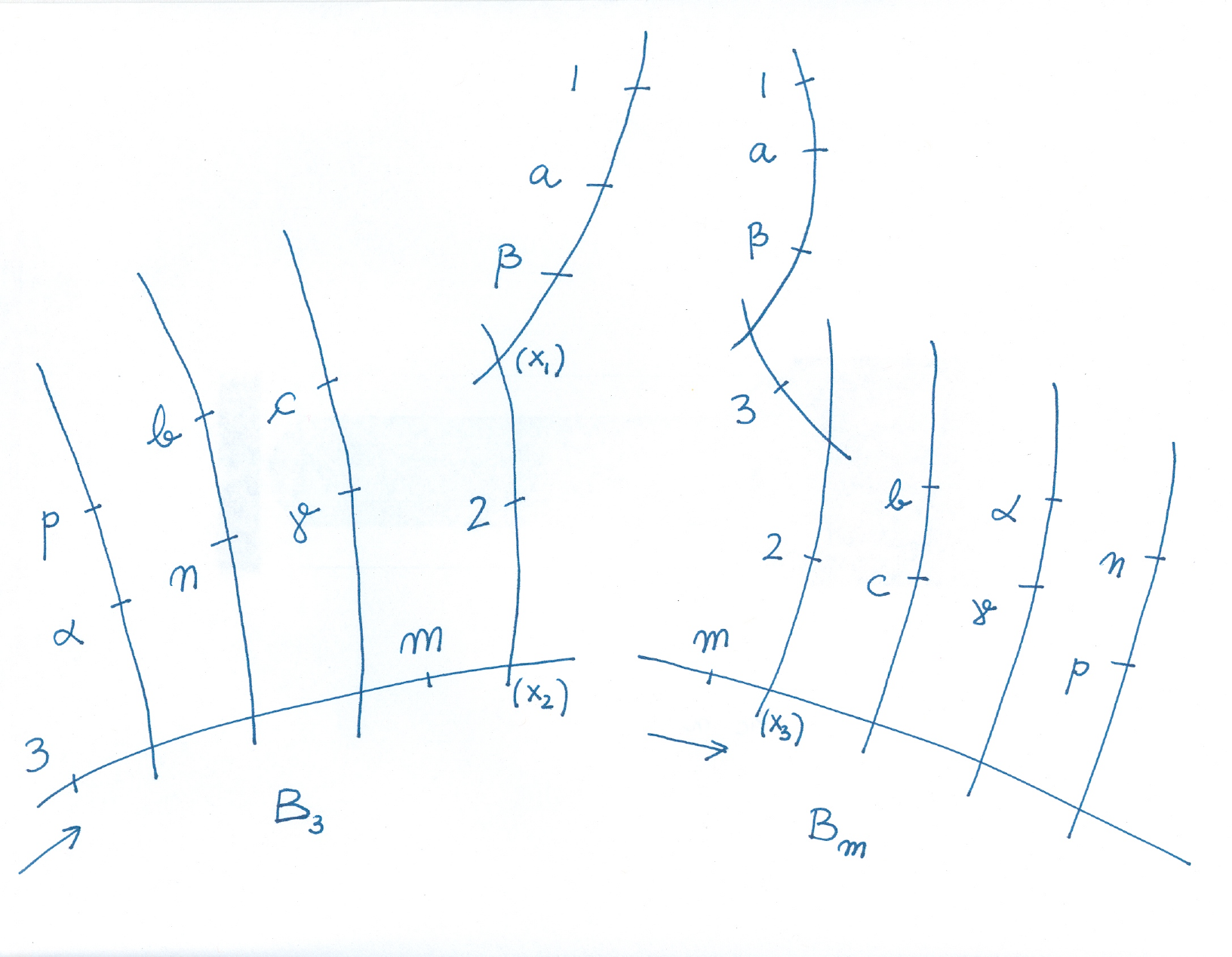}
\caption{ The curves $B_3$ and $B_m$ in $\M_{0,12}$.}\label{pic_n}
\end{figure}
The class of $B_3$ in $\M_{0,12}$ is given by:
$$B_3=-\De_{p,\al}-\De_{b,n}-\De_{c,\ga}+\De_{3,p,\al}+
\De_{3,b,n}+\De_{3,c,\ga}+\De_{3,m}+\De_{1,2,3,a,\be}-\De_{1,2,a,\be}.$$

In order to find the class of the line $3x_2$, we further remove the points $3, x_2$ and place an extra point $x_3$ at a general point of the line $3x_2$. Repeating the construction of Theorem \ref{blowupdescr} for this new configuration, we obtain a curve in 
$$\M_{0,8}=\M_{0,\{b,c,\al,\ga, m, n, p, x_3\}},$$ 
corresponding to the proper transform of a general line $L_3$ through $x_3$.  The class $L_3$ in the blow-up of $\PP^2_{\overline{\FF}_3}$ at the points 
$b, c,\al,\ga, m, n, p, x_3$ is the sum of the class of the (proper transform of the) line $mx_3$ and the exceptional divisor $B_m$ corresponding to the point $m$. 
The class of $B_m$ in $\M_{0,12}$ is given by:
$$B_m=-\De_{b,c}-\De_{\al,\ga}-\De_{n,p}+\De_{m,b,c}+
\De_{m,\al,\ga}+\De_{m,n,p}+\De_{1,2,3,a,\be,m}-\De_{1,2,3,a,\be}.$$

In order to find the class of the line $mx_3$, we further remove the points $m, x_3$ and place an extra point $x_4$ at a general point of the line $mx_3$. Repeating the construction of Theorem \ref{blowupdescr} for this new configuration, we obtain a curve in 
$$\M_{0,7}=\M_{0,\{b,c,\al,\ga, n, p, x_4\}},$$ 
corresponding to the proper transform of a general line $L_4$ through $x_4$. The class $L_4$ in the blow-up of $\PP^2_{\overline{\FF}_3}$ at the points 
$b, c,\al,\ga, n, p, x_3$ is the sum of the class of the (proper transform of the) line $\ga x_4$ and the exceptional divisor $B_{\ga}$ corresponding to the point $\ga$. 
The class of $B_{\ga}$ in $\M_{0,12}$ (see Figure \ref{pic_p}) is given by:
$$B_{\ga}=-\De_{p,b}+\De_{\ga,p,b}+\De_{\ga,\al}+\De_{\ga,n}+\De_{\ga,c}+
\De_{1,2,3,a,\be,\ga,m}-\De_{1,2,3,a,\be,m}.$$
\begin{figure}[htbp]
\includegraphics[width=4in]{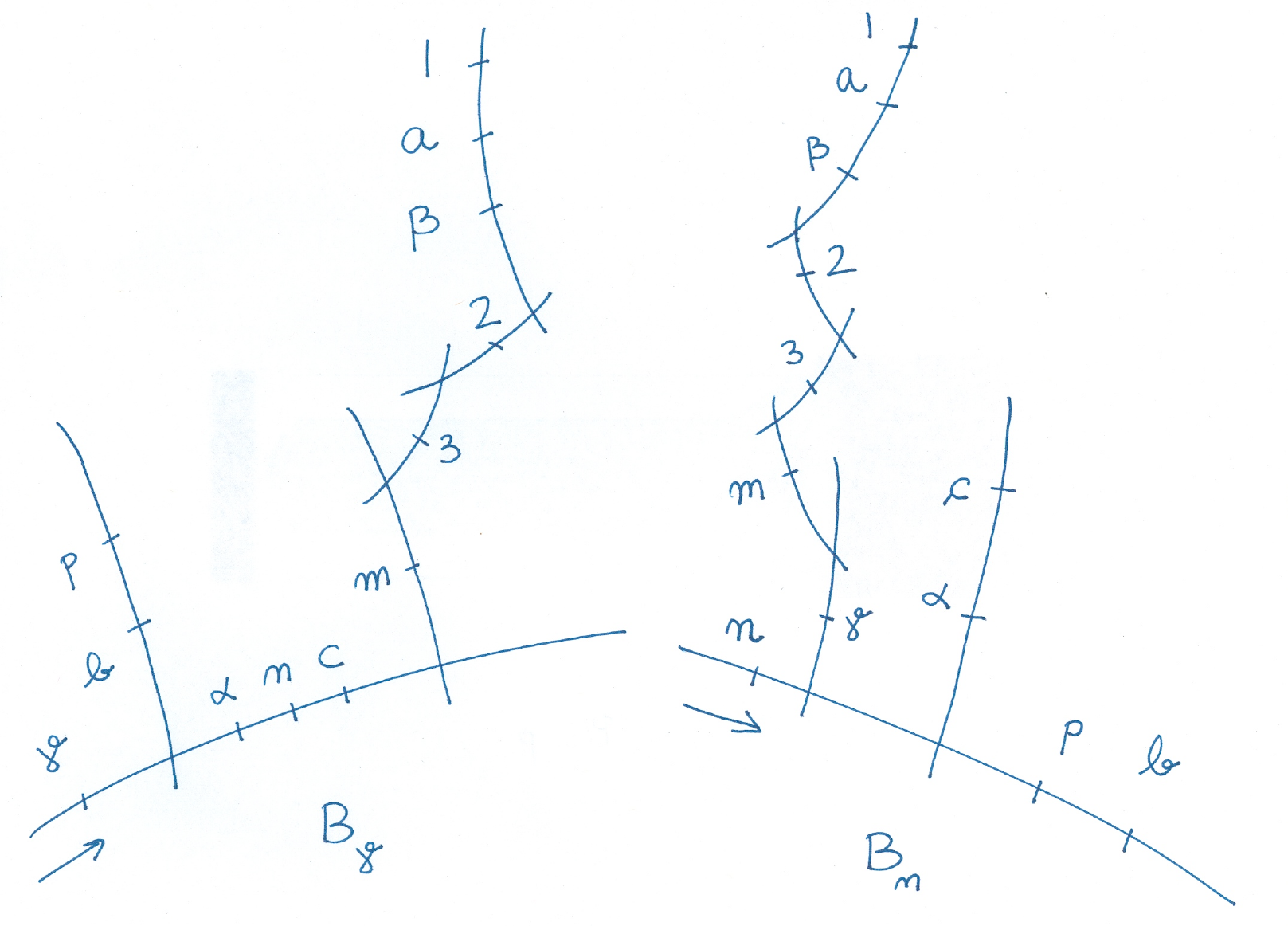}
\caption{ The curves $B_{\ga}$ and $B_n$ in $\M_{0,12}$.}\label{pic_p}
\end{figure}

In order to find the class of the line $\ga x_4$, we further remove the points $\ga, x_4$ and place an extra point $x_5$ at a general point of the line $\ga x_4$. Repeating the construction of Theorem \ref{blowupdescr} for this new configuration, we obtain a curve in 
$$\M_{0,6}=\M_{0,\{b,c,\al, n, p, x_5\}},$$ 
corresponding to the proper transform of a general line $L_5$ through $x_5$. The class $L_5$ in the blow-up of $\PP^2_{\overline{\FF}_3}$ at the points 
$b, c,\al, n, p, x_3$ is the sum of the class of the (proper transform of the) line $nx_5$ and the exceptional divisor $B_n$ corresponding to the point $n$. 
The class of $B_n$ in $\M_{0,12}$ (see Figure \ref{pic_p}) is given by:
$$B_n=-\De_{c,\al}+\De_{n,c,\al}+\De_{n,p}+\De_{n,b}+
\De_{1,2,3,a,\be,\ga,m,n}-\De_{1,2,3,a,\be,m,n}.$$

Note that all of the curves $B_2, B_3, B_m, B_{\ga}, B_n$ are sums of $F$-curves by Claim \ref{fiber comp}. At this point we can continue to follow the algorithm, or just notice that 
\begin{gather*}
E'_1-B_2-B_3-B_m-B_{\ga}-B_n=\big(-\De_{b,c,p}+\De_{b, c}+\De_{b, p}+\De_{c,p}\big)+\\
\big(-\De_{b,c,p,\al}+\De_{\al, b}+\De_{\al, c}+
\De_{\al, p}+\De_{b,c,p}\big)
\end{gather*} 
that is, the difference is the sum of an $F$-curve and a curve as in Claim \ref{fiber comp} (see Figure \ref{pic_r}). It follows that $E'_1$ (and hence, $E_1$) is a sum of $F$-curves. 
\begin{figure}[htbp]
\includegraphics[width=4in]{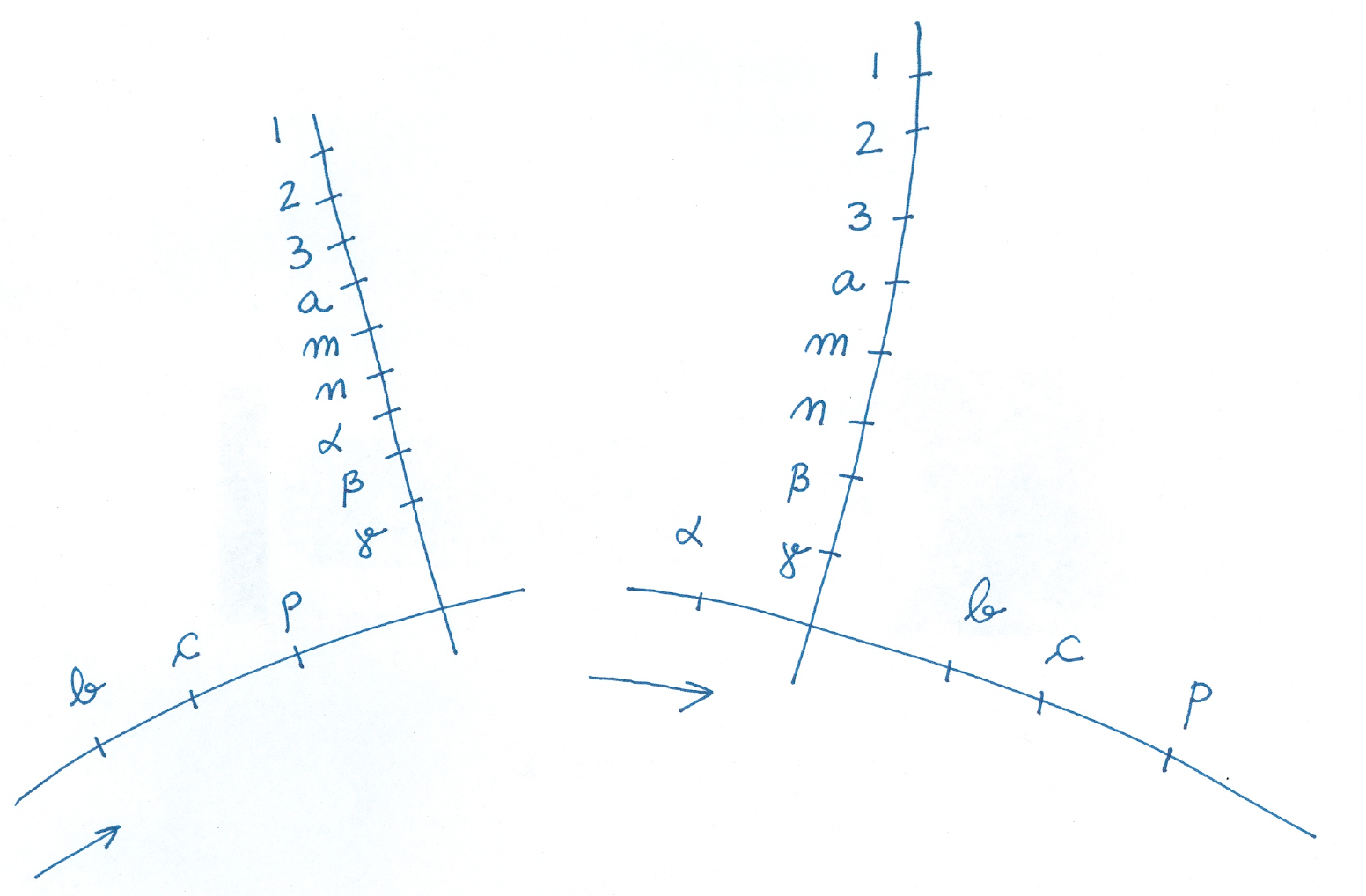}
\caption{}\label{pic_r}
\end{figure}

\Trick{The class $F$ is a sum of $F$-curves.} 
The class of $F$ can be obtained from the class of $E'_1$ (see \ref{Hesse 1st blow-up})
by interchanging:
$$m \leftrightarrow 1, \quad n \leftrightarrow \be,\quad p \leftrightarrow a,\quad 
c \leftrightarrow \ga.$$ 
(with $2, 3, \al, b$ not changed). Therefore, $F$ is also a sum of $F$-curves.

\begin{rmk}\label{K1}
One can see that the intersections with $(K+\De)$ add up. We have:
$$K_{\MM_{0,12}}=\frac{1}{11}\big(-2\de_2+5\de_3+10\de_4+13\de_5+14\de_6\big),$$
$$(K+\De)\cdot C=39,\quad (K+\De)\cdot F=16,\quad (K+\De)\cdot E_1=17,\quad$$
$$(K+\De)\cdot E_2=2,\quad (K+\De)\cdot E_3=2.$$

Note that $K\cdot C=7$, and thus the usual lower bound for the dimension of the Hom scheme $\Hom(\PP^1,\MM_{0,12})$ at $[C]$ (\ref{lower bound}) is $-1$. This shows that $K\cdot C$ satisfies the necessary lower-bound for $C$ to be rigid (although not by a large margin). The same computation also shows that the components of the characteristic $3$ fiber are in fact not rigid.  This happens also in our second example (see Rmk. \ref{K2}). 
\end{rmk}


\section{Rigid matroids}\label{rigid matroids}

The calculation above shows that many rigid curves defined using configurations of points
often break arithmetically simply because the configuration itself has primes of bad reduction. Here we use the following definition:

\begin{defn}
Let $\cL$ be a finite connected matroid of rank $r$ (see \cite{KT} for the definition of connected matroids), let $R$ be a domain with the field of fractions $K$, 
and let $\p\subset R$ be a prime ideal. 
We say that $\cL$ has $\p$ as a prime of bad reduction over $R$ if there exists a family of sections
$$p_i:\,\Spec R_\p\to\bP^{r-1}_{R_\p}$$ 
such that the matroid of the configuration of points
$\{p_i(0)\}\subset\bP^{r-1}_K$ is isomorphic to~$\cL$, while the matroid of the specialization
$\{p_i(\p)\}\subset\bP^{r-1}_{R/\p}$ has rank $r$, is connected, and is a strict subset of $\cL$
(i.e.,~the specialization has less linearly independent subsets).

A matroid $\cL$ is called {\em arithmetically rigid} if it has no primes of bad reduction.
\end{defn}

\begin{ex}\label{SB}
Let $\cL$ be a matroid of rank~$2$. Without loss of generality we can assume that $\cL$
is uniform, i.e.~any two points are linearly independent. 
Of course $\cL$ can not be arithmetically rigid (unless it has at most three points), so let's fix a realization of $\cL$ over a field
of fractions~$K$ of a Dedekind domain $R$, i.e.~a collection of elements $p_1,\ldots,p_n\in K$. 
We can extend them to a collection of sections
$p_i:\,\Spec R\to \bP^1_R$. Primes of bad reduction correspond to places where two sections intersect.
If there are no places of bad reduction then we can arrange so that $p_0=0$, $p_1=1$, $p_2=\infty$.
Then the remaining points $p_3,\ldots,p_n$ form what's known as a {\em clique of exceptional units}:
we have 
$$p_i,\ 1-p_i,\ p_i-p_j\in R^*\quad\hbox{\rm for any $i,j$}.$$
\end{ex}

How about rank~$3$? We can obtain examples by simply considering rigid matroids which are realizable
only in prime characteristic, e.g.~the Fano matroid. So let's impose an extra condition that a matroid
is realizable in characteristic~$0$. 

It is easy to see using Lafforgue's theory \cite{La}
of compact moduli spaces of hyperplane arrangements that if $\cL$ is not rigid then $\cL$ is not arithmetically rigid.
In other words, if there exists an algebraic curve $B$ over an algebraically closed field of characteristic $0$
and sections $p_1,\ldots,p_n:\,B\to \bP^{r-1}_B$, such that the matroid of $\{p_i(b)\}$ is isomorphic to $\cL$ for any $b\in B$
and yet configurations $\{p_i(b)\}$ and $\{p_i(b')\}$ are not projectively equivalent for some $b,b'$ then $B$ is not proper
and one of the infinite points is a prime of bad reduction.

Quite surprisingly, we know only two examples of arithmetically rigid matroids of rank~$3$
realizable in characteristic~$0$.
One is a uniform matroid ($4$ general points in~$\bP^2$),
which is useless for our purposes.
 Another is a quite remarkable configuration that represents the golden ratio. 
We learned about it from the book \cite{G}.

\begin{ex}
Let $\tau$ be  a root of $\tau^2-3\tau+1=0$. Let $R=\ZZ[\tau]$ and $K=\QQ(\tau)$.
Consider the following configuration of nine points of $\PP^2_R$ (in coordinates $X,Y,Z$): 
$$a=(1,0,0), \quad   b=(0,1,0), \quad  c=(0,0,1),$$
$$d=(1,1,1), \quad   e=(2-\tau,1-\tau,1), \quad  f=(1,1,0),$$
$$g=(0,1-\tau,1), \quad   h=(1,0,1), \quad  i=(1,\tau,0).$$

Consider the  following lines (see Fig.~\ref{adfgsdfsfb}):
\begin{align*}
L_1=abif: &\quad Z=0,\\
L_2=ach: & \quad Y=0,\\ 
L_3=age: &\quad Y=(1-\tau)Z,\\
L_4=bcg: &\quad X=0,\\ 
L_5=bdh: &\quad X=Z,\\ 
L_6=cei: & \quad Y=\tau X,\\ 
L_7=cdf: &\quad Y=X,\\ 
L_8=dgi: &\quad Y=\tau X-(\tau-1)Z=0,\\
L_9=efh: &\quad Y=X-Z.
\end{align*}

\begin{figure}[htbp]
\includegraphics[width=5in]{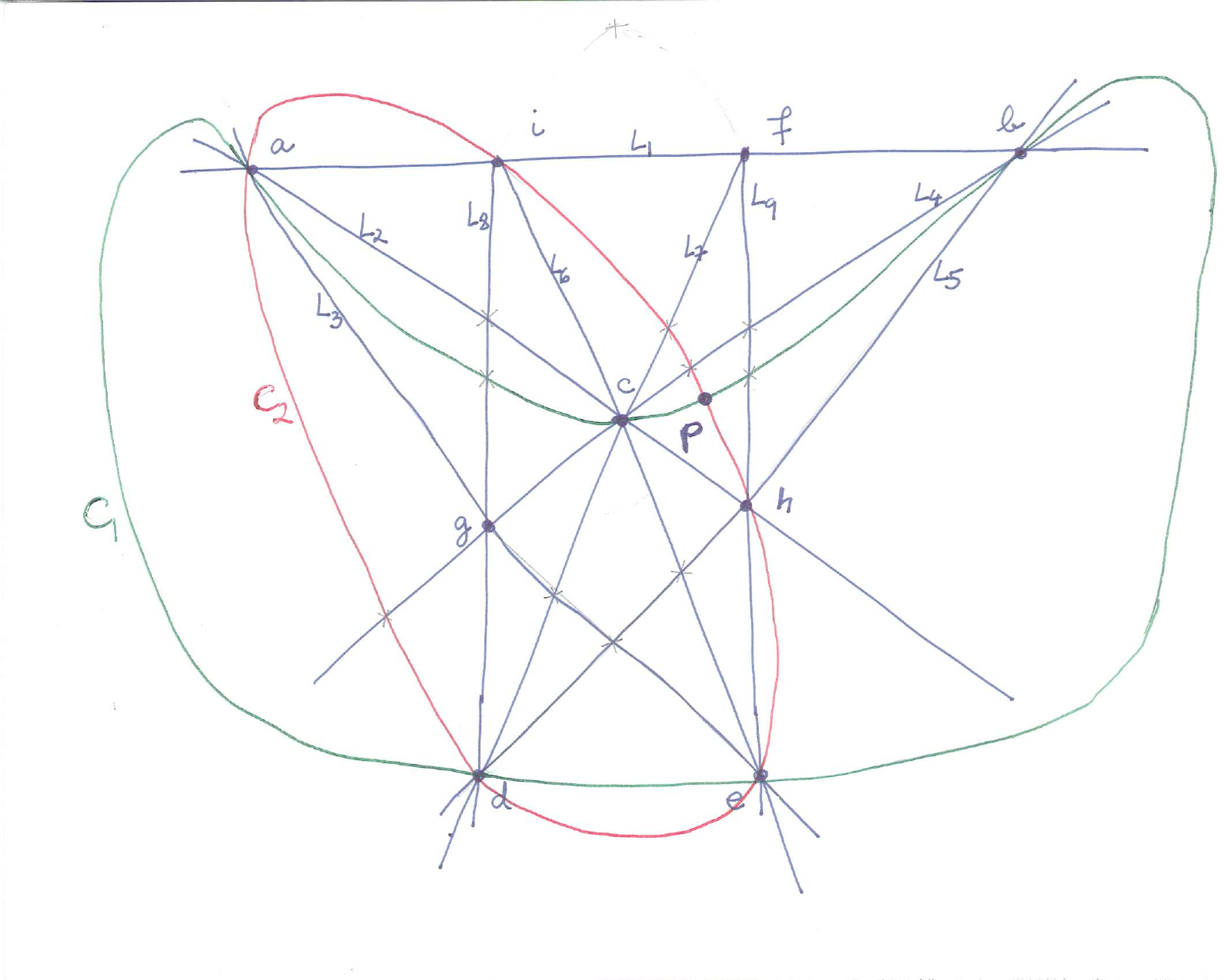}
\caption{Gr\"unbaum configuration (with added conics $C_1$ and $C_2$)}\label{adfgsdfsfb}
\end{figure}

The only possible non-trivial $3\times3$ minors of the $3\times 9$ matrix of coordinates of points $a,\ldots,i$
are $\pm1$, $\pm \tau$, $\pm(\tau-1)$, and $\pm(\tau-2)$.
All these minors are units in $\bZ[\tau]$, and therefore 
the Gr\"unbaum configuration has no primes of bad reduction in $\bZ[\tau]$.
It is quite easy to check (see \cite{G}) that the Gr\"unbaum configuration is rigid
and $\bQ[\tau]$ is its field of definition. Therefore, this matroid is arithmetically rigid.

We find it remarkable that if we pick any smooth conic through five of the nine points, the construction in Thm. \ref{blowupdescr} and Cor. \ref{class1} gives a morphism:
$$\PP^1_R\ra\MM_{0,9;R},$$
whose generic fiber $\PP^1_K$ (when base-changed to $\CC$) is a moving curve on $\MM_{0,9}$. For example, if we let $C_1$ be the conic through the points $a, b, c, d, e$, by Cor. \ref{class1}, we have that the numerical class of $C_1$ is given by:
\begin{align*}
C_1&=\De_{b,f,i}+\De_{a,f,i}+\De_{c,h}+\De_{a,h}+\De_{e,g}+\De_{a,g}+\De_{b,g}+\\
&+\De_{c,g}+\De_{b,h}+\De_{d,h}+\De_{c,f}+\De_{d,f}+\De_{c,i}+\De_{e,i}+\\
&+\De_{d,g,i}+\De_{e,f,h}+\De_{g,i}+\De_{f,h}+2\De_{f,g}+2\De_{g,h}+2\De_{h,i}
\end{align*}
Since $$K_{\MM_{0,9}}=\frac{1}{4}\big(-\de_2+\de_3+2\de_4\big)$$ it follows that $K\cdot C_1=-4$ and thus by (\ref{lower bound}), the curve $C_1$ moves on $\MM_{0,9}$.

Similarly, if we let $C_2$ be the conic through the points $a, d, e, h, i$, by Cor. \ref{class1}, we have that the numerical class of $C_2$ is given by:
\begin{align*}
C_2&=\De_{b,f,i}+\De_{a,b,f}+\De_{a,c}+\De_{c,h}+\De_{a,g}+\De_{e,g}+\De_{b,d}+\\
&+\De_{b,h}+2\De_{b,c,g}+\De_{b,e}+\De_{c,d,f}+\De_{c,f}+\De_{c,e}+\\
&+\De_{c,i}+\De_{d,g}+\De_{g,i}+\De_{e,f}+\De_{f,h}+2\De_{f,g}+\De_{g,h}
\end{align*}
It follows that $K\cdot C_2=-3$ and thus by (\ref{lower bound}), the curve $C_2$ moves as well. 

\end{ex}

\section{Arithmetic break of a ``Two Conics'' curve - part I}\label{break of 2-conics curve}

Now we are going to construct a curve in $\oM_{0,12}$ by applying a ``Two conics'' construction. We will use the configuration in Example \ref{SB}.
Up to symmetries, there is only one choice:
consider the following two (smooth) conics:
\begin{align*}
C_1=abcde:& \quad XY-\tau XZ+(\tau-1)YZ=0,\\
C_2=adehi: &\quad (\tau-2)Y^2+Z^2-(\tau-1)XY-XZ+YZ=0.
\end {align*}

Let $p$ be the fourth intersection point of $C_1$ and $C_2$:
$$p=(\frac{1-\tau}{2},-\tau,1)=(1,2(\tau-1), 2(2-\tau)).$$

Let $S$ be the blow-up of $\PP^2_{R}$ at $p$ and let $E_p\cong\PP^1_R$ be the exceptional divisor. Consider the natural fibration $\pi: S\ra\PP^1_R$ that resolves the projection from $p$. 
The proper tranforms of the nine lines, the two conics and the exceptional divisor $E_p$ give twelve sections of $\pi$. After blowing up the $R$-points where the sections intersect, one obtains a family of stable $12$-pointed rational curves over $\PP^1_K$ that intersects the interior of $\M_{0,12}$. Denote by $C$ this curve in $\M_{0,12}$. According to Theorem~\ref{rigid map}, the corresponding map $\PP^1_K\ra\MM_{0,12}$ is rigid.
Despite the fact that the Gr\"unbaum arrangement is arithmetically rigid,
we will prove that $C$ breaks in characteristic $5$ into several components, 
each a sum of $F$-curves. 

\begin{notn}
We will denote $1,\ldots,9$, $u$, $v$, $p$ the markings corresponding to the sections given by $L_1,\ldots, L_9$, $C_1$, $C_2$, $E_p$. 
\end{notn}

\Trick{The class of $C$.} 
One can compute the numerical class of $C$ from Prop. \ref{class2}:
\begin{align*}
C&=\De_{1,2,3,u,v}+\De_{2,4,6,7,u}+\De_{3,6,9,u,v}+\De_{5,7,8,u,v}+\De_{1,4,5,u}+\\
&+\De_{1,6,8,v}+\De_{2,5,9,v}+\De_{1,7,9}+\De_{3,4,8}+\De_{2,8}+\De_{3,5}+\De_{3,7}+
\De_{4,9}+\\
&+2\De_{4,v}+\De_{5,6}+\De_{7,v}+\De_{8,9}+
\De_{8,u}+\De_{9,u}+\De_{u,p}+\De_{v,p}.
\end{align*}

\Trick{Family $S\ra\PP^1_R$ in local coordinates.\label{set-up}} 
The blow-up $S$ of $p$ is an arithmetic surface in $\PP^2_R\times \PP^1_R$ with equation:
$$\big(Y-2(\tau-1)X\big)v=\big(Z-2(2-\tau)X\big)u,$$
(where $u, v$ are the coordinates on $\PP^1_R$). The exceptional divisor $E_p$ is cut by 
$$Y-2(\tau-1)X=Z-2(2-\tau)X=0.$$

We will need to consider both charts $v=1$ and $u=1$.

\Trick{Chart $v=1$.\label{chart v=1}}   
The proper transforms of the twelve sections are :
\begin{align*}
p &:\quad Z=2(2-\tau)X,\\
1 &: \quad Z=0,\\
2 &:\quad 2(\tau-1)X+\big(Z-2(2-\tau)X\big)u=0,\\
3 &:\quad 2(\tau-1)X+\big(Z-2(2-\tau)X\big)u=(1-\tau)Z,\\
4 &:\quad X=0,\\
5 &:\quad X=Z,\\
6 &:\quad 2(\tau-1)X+\big(Z-2(2-\tau)X\big)u=\tau X,\\
7 &:\quad 2(\tau-1)X+\big(Z-2(2-\tau)X\big)u=X,\\
8 &:\quad 2(\tau-1)X+\big(Z-2(2-\tau)X\big)u=\tau X-(\tau-1)Z,\\
9 &:\quad 2(\tau-1)X+\big(Z-2(2-\tau)X\big)u=X-Z,\\
u &:\quad X\big(\tau+u\big)+Z\big(\tau-1\big)u=0,\\
v &:\quad X\big((-2\tau+6)u^2+(-\tau+5)u+1\big)+Z\big((\tau-2)u^2+u+1\big)=0
\end{align*}

\Trick{Chart $u=1$.\label{chart u=1}} 
The proper transforms of the twelve sections are :
\begin{align*}
p &:\quad Y=2(\tau-1)X\\
1 &: \quad Yv+2X\big((1-\tau)v+2-\tau\big)=0,\\
2 &:\quad Y=0,\\
3 &:\quad Yv+2X\big((1-\tau)v+2-\tau\big)=(2-\tau)Y,\\
4 &:\quad X=0,\\
5 &:\quad Yv+2X\big((1-\tau)v+2-\tau\big)=X,\\
6 &:\quad Y=\tau X,\\
7 &:\quad Y=X,\\
8 &:\quad Yv+2X\big((1-\tau)v+2-\tau\big)=(\tau-1)X-(\tau-2)Y,\\
9 &:\quad Yv+2X\big((1-\tau)v+2-\tau\big)=X-Y,\\
u &:\quad -X\big(1+\tau v\big)+Y\big(\tau-1\big)v=0,\\
v &:\quad X\big((2-2\tau)v^2+(7-4\tau)v+7-3\tau\big)+Y\big(v^2+v+(\tau-2)\big)=0
\end{align*}

\Trick{Break of the curve $C$ in characteristic $5$ (outline).}

This is similar to the argument in Section \ref{break of hypergraph curve}. Consider the induced rational map: 
$$\PP^1_R\dra\M_{0,12}.$$

In order to resolve this map, one has to blow-up the arithmetic surface $\PP^1_R$ several times along the characteristic $5$ fiber $\PP^1_{\FF_5}$ of $\PP^1_R\ra\Spec R$ (at $\tau=-1$). We now outline the strategy. We blow-up of the arithmetic surface $\PP^1_R$ along four distinct points in $\PP^1_{\FF_5}$; in chart $v=1$ they are given by: 
$$u=0,\quad u=2,\quad u=-1,\quad u=\infty.$$

Let the corresponding exceptional divisors be $E_1$, $E_2$, $E_3$, $E_4$.  We blow-up one more point on $E_2$, resulting in an exceptional curve $E_5$ (See Fig.~\ref{pic_c}). We let $T$ be the resulting arithmetic surface. We will abuse notations and denote by $E_2$ the proper transform of $E_2$ in $T$.

\begin{figure}[htbp]
\includegraphics[width=4in]{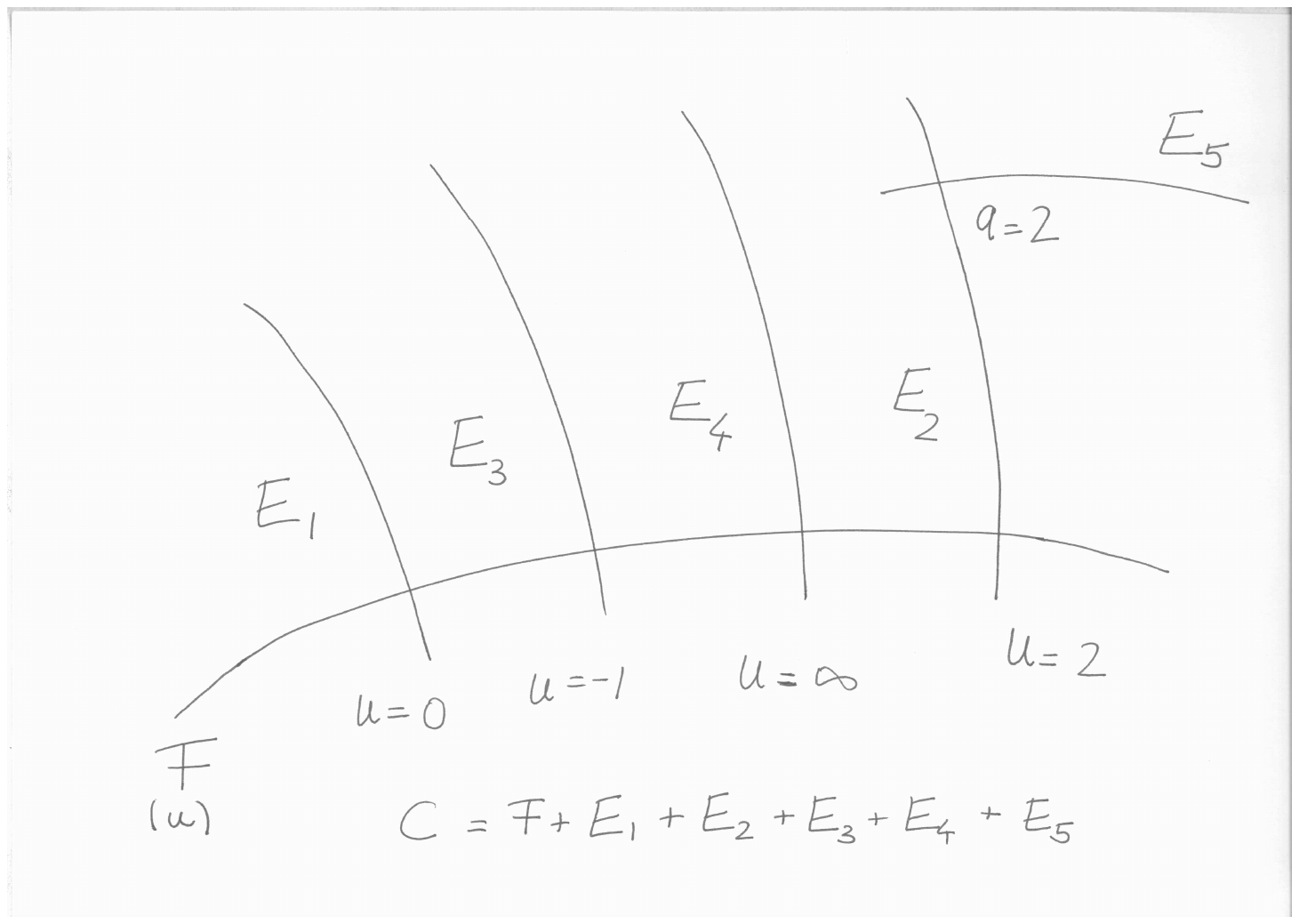}
\caption{Components of the characteristic $5$ fiber.}\label{pic_c}
\end{figure}

\begin{notn}
Let $F$ denote the proper transform of the characteristic $3$ fiber $\PP^1_{\FF_3}$. 
\end{notn}

We construct a family $\pi: S\ra T$ with $12$ sections, such that over an open set of the characteristic zero fiber of $T\ra\Spec R$, this is the universal family. It is 
easy to see that along a dense open $T^0\subseteq T$  the sections are disjoint and thus define a map $T^0\ra\M_{0,12}$. Moreover, we can enlarge $T^0$ such that that its intersection with each of $F$, $E_i$ is non-empty. Simply blow-up the total space $S$ along the sections that become equal over the generic points of these curves; occasionally one will have to blow-up several times. This is an easy calculation, which we omit. The result is a family over $T^0$ 
of semistable rational curves with twelve disjoint sections. By contracting unstable components in fibers, we obtain a family of stable curves over $T^0$. The maps 
$T^0\cap F\ra\MM_{0,12}$, $T^0\cap E_i\ra\MM_{0,12}$ extend uniquely to morphisms $$F\ra\MM_{0,12},\quad E_i\ra\MM_{0,12}.$$

Just as in Section \ref{break of hypergraph curve}, from the universal family $S\ra T^0$ restricted to $T^0\cap F$, $T^0\cap E_i$ we can determine the classes of the curves $F$, $E_i$. One will eventually have to do further blow-ups to resolve the map $\PP^1_R\dra\M_{0,12}$, but since one can check directly the equality of numerical classes:
$$C=F+E_1+ E_2+ E_3+E_4+E_5,$$
this proves that any other extra components in the characteristic $5$ fiber will map constantly to $\MM_{0,12}$.   It is easy to see that each of the curves $E_i$ is a sum of  $F$-curves. In the next section we give a similar argument that shows $F$ is a sum of $F$-curves. 

\Trick{The class of $F$.\label{main comp}}
As  $\tau=-1$, if $u$ is general, the equations in (\ref{chart v=1}) describe the curve $F$ as a curve that lies in the boundary component:
$$\de_{578p}\cong\MM_{0,9}\times\MM_{0,5}.$$

As a result we have an equality of numerical classes $F=F'+F''$,
where $F'$, resp., $F''$ are the two projections of $F$ onto $\MM_{0,9}$ and $\MM_{0,5}$ respectively. 



\Trick{The class of $F'$.} This can be determined directly from the equations in (\ref{main comp}). Alternatively, one can use the fact that in characteristic $5$ the conics $C_1$ and $C_2$ become tangent at $d$ (hence, $p=d$) (see Section \ref{tangent}). The class of $F'$ as a curve in $$\M_{0,9}=\M_{0, \{1,2,3,4,6,9,u,v,x\}},$$ (where $x$ is the attaching point) is given by:
\begin{align*}
F'&=\De_{1,2,3,u,v}+\De_{3,6,9,u,v}+\De_{2,4,6,u}+\De_{u,v,x}+\De_{1,4,u}+\\
&\De_{1,6,v}+\De_{2,9,v}+\De_{1,9}+\De_{3,4}+\De_{4,9}+2\De_{4,v}+\De_{9u}.
\end{align*}

As a curve in $\M_{0,12}$, the class of $F'$ is:
\begin{align*}
F'&=\De_{1,2,3,u,v}+\De_{3,6,9,u,v}+\De_{5,7,8,u,v,p}+\De_{2,4,6,u}-2\De_{5,7,8,p}+\\
&+\De_{1,6,v}+\De_{1,4,u}+\De_{2,9,v}+\De_{1,9}+\De_{3,4}+\De_{4,9}+
2\De_{4,v}+\De_{9u}.
\end{align*}

\Trick{The class of $F''$ (see Fig.~\ref{pic_d}).} We use the equations in (\ref{main comp}). We blow-up the total space $S$ along $\tau+1=Z-X=0$ in order to separate the sections $5, 7, 8, p$. In local coordinates
$Z-X=(\tau+1)W$, with exceptional divisor cut by $\tau=-1$ and new coordinate $W$. The proper transforms of the four sections are given by:
\begin{align*}
p &:\quad W=(2-\tau)X,\\
5 &:\quad  W=0,\\
7 &:\quad  X(\tau-2)+\big(W+X(\tau-2)\big)u=0,\\
8 &:\quad  X(\tau-2)+W(\tau-1)+\big(W+X(\tau-2)\big)u=0.
\end{align*}

\begin{figure}[htbp]
\includegraphics[width=4in]{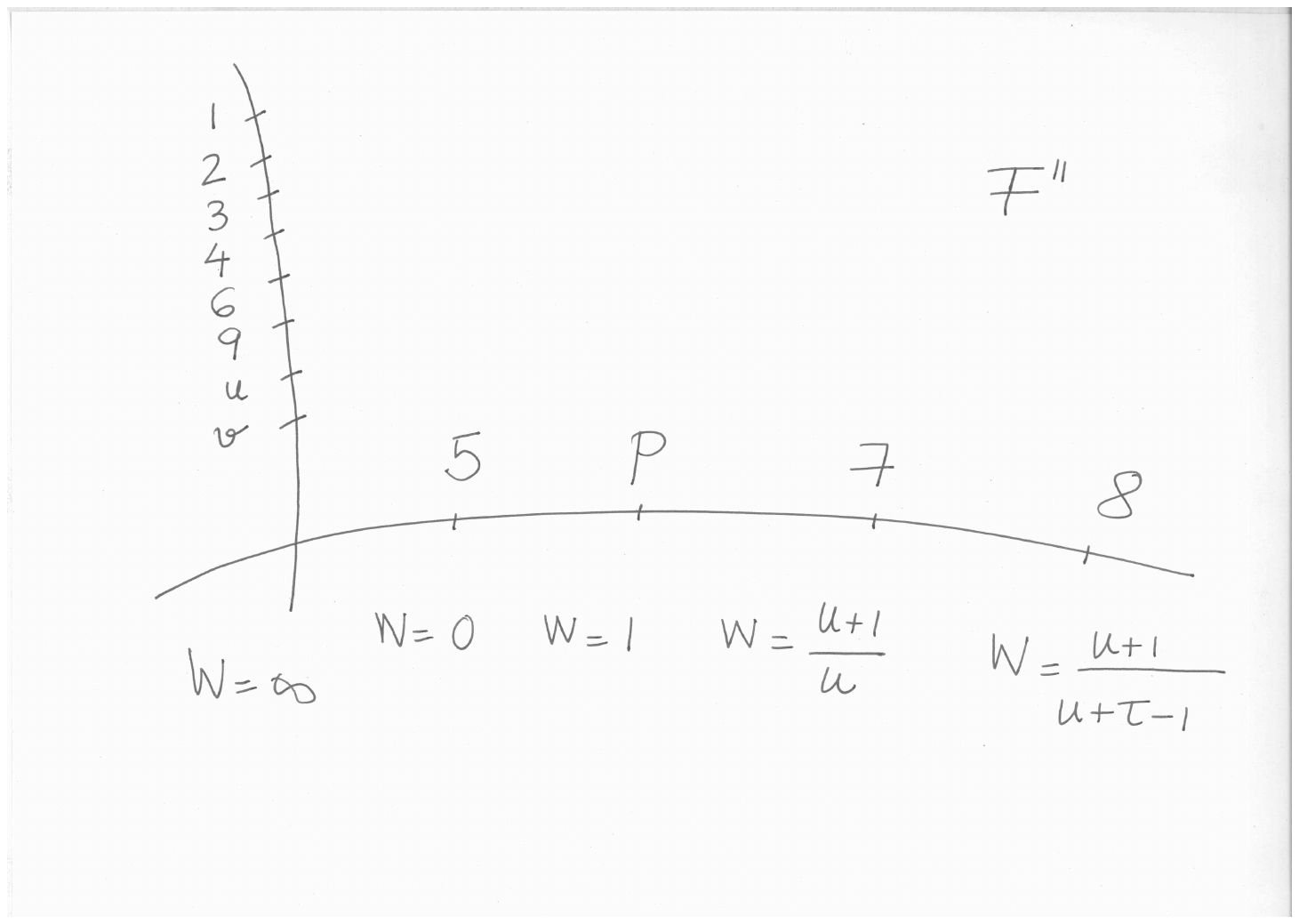}
\caption{The component $F''$.}\label{pic_d}
\end{figure}

The ``attaching section" is given by $X=0$. As a curve in $\M_{0,5}$, $F''$ has class  $\De_{5,7,8}+\De_{5,7,p}+\De_{5,8,p}+\De_{7,8,p}$.
As classes in in $\MM_{0,12}$, we have:
$$F''=-2\De_{5,7,8,p}+\De_{5,7,8}+\De_{5,7,p}+\De_{5,8,p}+\De_{7,8,p}.$$

\Trick{The class of $E_1$ (see Fig.~\ref{pic_e}).} We use the notations from (\ref{chart v=1}). 
We blow-up $\PP^1_R$ at the point $\tau=-1$, $u=0$. In local coordinates: $u=(\tau+1)a$ with exceptional divisor $E_1: \tau=-1$ and new coordinate $a$. The proper transforms of the twelve sections have equations:
\begin{align*}
p &: \quad Z=2(2-\tau)X,\\
1 &: \quad Z=0,\\
2 &:\quad 2(\tau-1)X+\big(Z+2(\tau-2)X\big)(\tau+1)a=0,\\
3 &:\quad 2(\tau-1)X+\big(Z+2(\tau-2)X\big)(\tau+1)a=(1-\tau)Z,\\
4 &:\quad X=0,\\
5 &:\quad X=Z,\\
6 &:\quad 2(\tau-1)X+\big(Z+2(\tau-2)X\big)(\tau+1)a=\tau X,\\
7 &:\quad (\tau-2)X+\big(Z+2(\tau-2)X\big)a=0,\\
8 &:\quad 2(\tau-1)X+\big(Z+2(\tau-2)X\big)(\tau+1)a=\tau X-(\tau-1)Z,\\
9 &:\quad 2(\tau-1)X+\big(Z+2(\tau-2)X\big)(\tau+1)a=X-Z,\\
u &:\quad X\big(\tau+(\tau+1)a\big)+Z\big(\tau^2-1\big)a=0,\\
v &:\quad X\big(10a^2+(\tau+6)a+1\big)+Z\big(5(\tau-1)a^2+(\tau+1)a+1\big)=0.
\end{align*}

Along $E_1: \tau=-1$ the sections become:
\begin{align*}
p=5=8 &: \quad Z=X,\\
1=9 &: \quad Z=0,\\
2=4=6=u &:\quad X=0,\\
3 &:\quad X=2Z,\\
7 &:\quad (2-a)X+aZ=0,\\
v &: \quad Z=-X.
\end{align*}
\begin{figure}[htbp]
\includegraphics[width=4in]{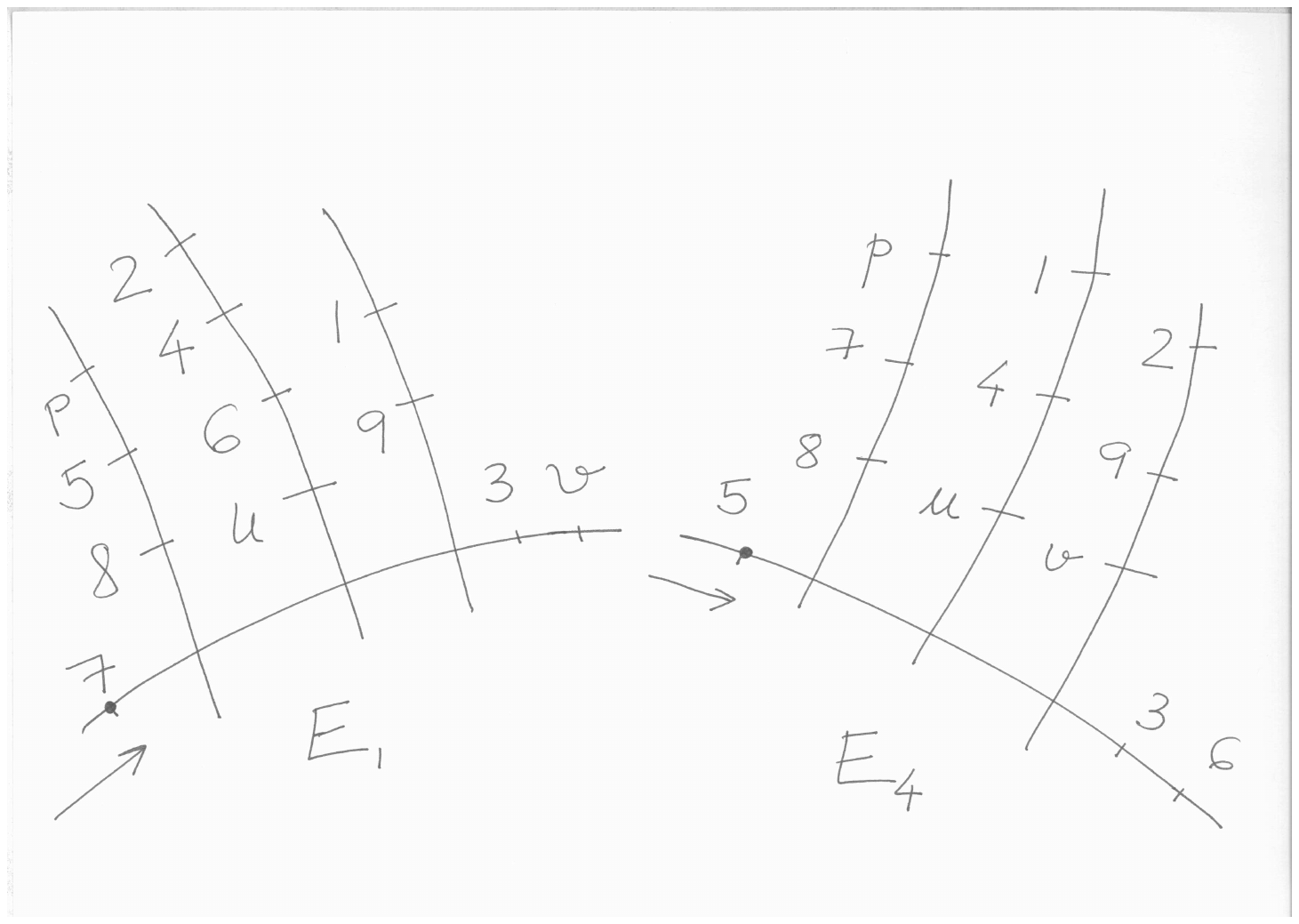}
\caption{The components $E_1$ and $E_4$ of the characteristic $5$ fiber.}\label{pic_e}
\end{figure}

It follows that $E_1$ has numerical class:
$$E_1=\De_{2,4,6,7,u}-\De_{2,4,6,u}+\De_{5,7,8,p}+\De_{1,7,9}-\De_{5,8,p}-\De_{1,9}+
\De_{3,7}+\De_{7,v}.$$

\Trick{The classes $E_2$ and $E_5$ (see Fig.~\ref{pic_f})} We use the notations from (\ref{chart v=1}). We blow-up $\PP^1_R$ at the point $\tau=-1$, $u=2$.  In local coordinates: $u=2+(\tau+1)a$ with exceptional divisor $E_2: \tau=-1$ and new coordinate $a$. The proper transforms of the twelve sections have equations:
\begin{align*}
p &: \quad Z=2(2-\tau)X,\\
1 &: \quad Z=0,\\
2 &:\quad 2(\tau-1)X+\big(Z+2(\tau-2)X\big)\big((\tau+1)a+2\big)=0,\\
3 &:\quad 2(\tau-1)X+\big(Z+2(\tau-2)X\big)\big((\tau+1)a+2\big)=(1-\tau)Z,\\
4 &:\quad X=0,\\
5 &:\quad X=Z,\\
6 &:\quad 2(\tau-1)X+\big(Z+2(\tau-2)X\big)\big((\tau+1)a+2\big)=\tau X,\\
7 &:\quad 2(\tau-1)X+\big(Z+2(\tau-2)X\big)\big((\tau+1)a+2\big)=X,\\
8 &:\quad (3\tau-7)X+Z+\big(Z+2(\tau-2)X\big)a=0,\\
9 &:\quad 2(\tau-1)X+\big(Z+2(\tau-2)X\big)\big((\tau+1)a+2\big)=X-Z,\\
u &:\quad X\big((\tau+1)a+\tau+2\big)+Z(\tau-1)\big((\tau+1)a+2\big)=0,\\
v &:\quad X\big(10a^2+(38-7\tau)a+35-10\tau\big)+Z\big(5(\tau-1)a^2+(9\tau-11)a+4\tau-5\big).
\end{align*}

Along $E_2: \tau=-1$ the sections become:
\begin{align*}
p=5=7 &: \quad Z=X,\\
1=6=v &: \quad Z=0,\\
2 &: \quad X=2Z,\\
3=4 &: \quad X=0,\\
9=u &:  \quad Z=-X.\\
8 &: \quad -aX+(a+1)Z=0.
\end{align*}
\begin{figure}[htbp]
\includegraphics[width=4in]{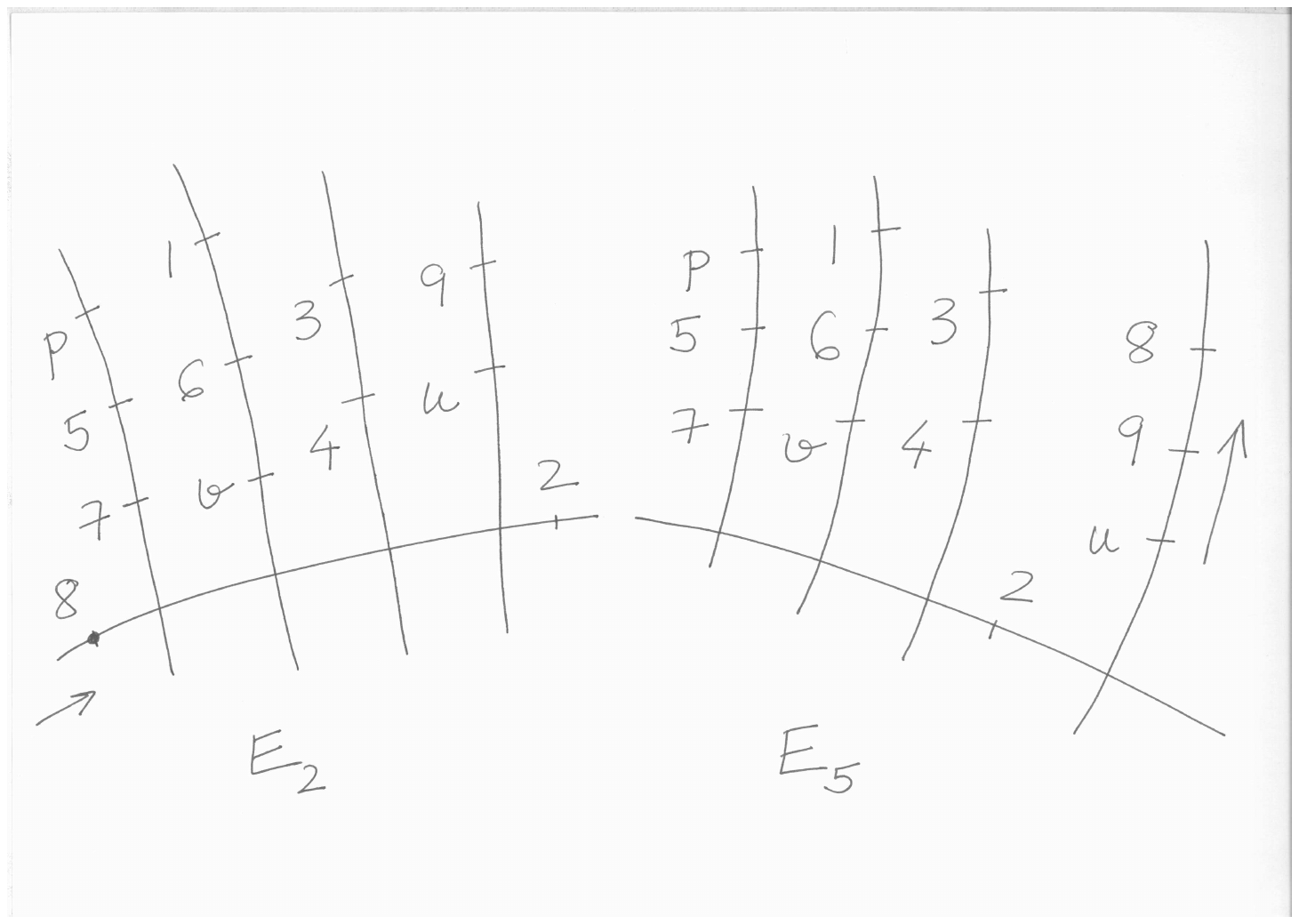}
\caption{The components $E_2$ and $E_5$ of the characteristic $5$ fiber.}\label{pic_f}
\end{figure}

It follows that $E_2$ has class:
$$E_2=\De_{1,6,8,v}+\De_{5,7,8,p}-\De_{1,6,v}-\De_{5,7,p}+\De_{3,4,8}+\De_{8,9,u}+
\De_{2,8}-\De_{3,4}-\De_{9,u}.$$

We now blow-up the point $a=2$ on $E_2$. In local coordinates, $a=2+(\tau+1)b$, with exceptional divisor $E_5: \tau=-1$ and new coordinate $b$. The sections $8$, $9$ and $u$ coincide along $E_5$ and after blowing up this locus, the sections are separated at a general point of $E_5$. The class of $E_5$ is the class of an $F$-curve:
$$E_5=-\De_{8,9,u}+\De_{8,9}+\De_{8,p}+\De_{9,p}$$

\Trick{The class of $E_3$ (see Fig.~\ref{pic_g}).} We use the notations from (\ref{chart v=1}). We blow-up $\PP^1_R$ at the point $\tau=-1$, $u=-1$.  In local coordinates: $u=-1+(\tau+1)a$ with exceptional divisor $E_3: \tau=-1$ and new coordinate $a$. The proper transforms of the twelve sections along $E_3$ are given by:

\begin{align*}
p=5=7=8=u=v &: \quad Z=X,\\
1 &: \quad Z=0,\\
2 &:\quad Z=2X,\\
3 &:\quad Z=-X,\\
4=9 &:\quad X=0,\\
6 &:\quad Z=3X.
\end{align*}
\begin{figure}[htbp]
\includegraphics[width=4in]{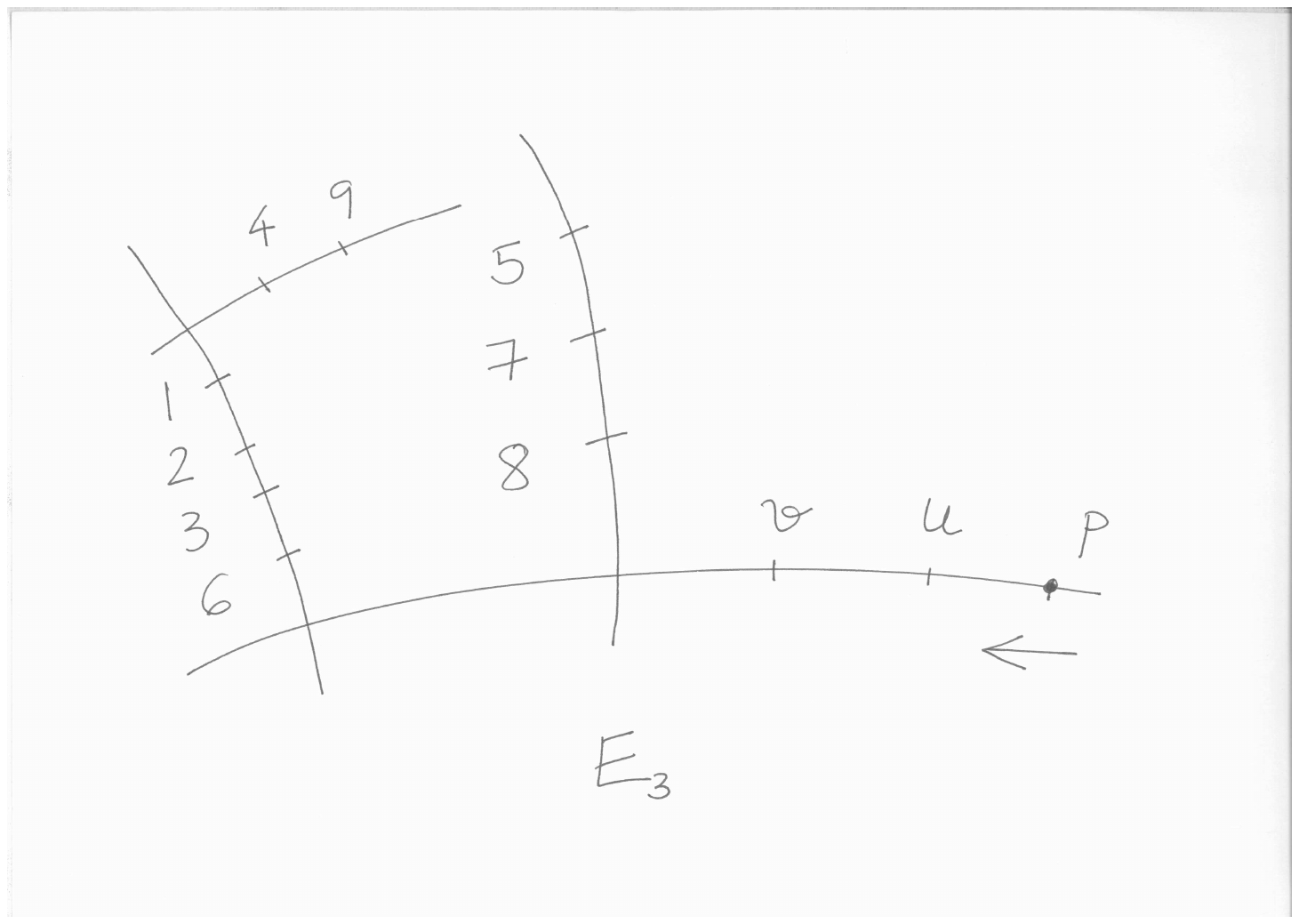}
\caption{The component $E_3$ of the characteristic $5$ fiber.}\label{pic_g}
\end{figure}

We blow-up the total space $S$ along $\tau=-1$, $Z=X$. In local coordinates:
$Z=X+(\tau+1)W$ with exceptional divisor $E: \tau+1=0$ and new coordinate $W$. The proper transforms of the six sections that were meeting along $E_3$ have local equations:
\begin{align*}
p &: \quad  W=(2-\tau)X,\\
5 &: \quad W=0,\\
7 &:\quad Xa(2\tau-3)+W\big(a(\tau+1)-1\big)=0,\\
8 &:\quad Xa(2\tau-3)+W\big(a(\tau+1)+\tau-2\big)=0,\\
u &:\quad Xa\tau+W(\tau-1)\big((\tau+1)a-1\big)=0,\\
v &:\quad X\big(5a^2+(\tau-2)a\big)+W\big(5(\tau-1)a^2+(-3\tau+7)a+\tau-2\big)=0.
\end{align*}

The ``attaching section" is given by $X=0$.  
Along $E_3: \tau+1=0$ we have:
\begin{align*}
p &: \quad  W=3X,\\
5=7=8 &: \quad W=0,\\
u &:\quad W=3aX,\\
v &: \quad W=-aX.
\end{align*}

The sections $5, 7, 8$ are separated when blowing up along $v=\tau+1=0$.
The curve $E_3$ is contained in the boundary components $\de_{123469}$, $\de_{49}$,$\de_{578}$ and it comes from a curve in $\M_{0,6}$, as only the markings $u, v$ move as the parameter $a$ moves along $E_3$ (all other cross-ratios are fixed). It follows that: 
$$E_3=-\De_{5,7,8,u,v,p}+\De_{5,7,8,u,v}+\De_{5,7,8,p}-\De_{5,7,8}+\De_{u,p}+\De_{v,p}$$

\Trick{The class of $E_4$ (see Fig.~\ref{pic_e}).}  We blow-up $\PP^1_R$ at the point $\tau=-1$, $u=\infty$. We use the  chart $u=1$ and the equations of the twelve sections in  (\ref{chart v=1}). We blow-up $\PP^1_R$ at the point $\tau=-1$, $v=0$. Consider the chart given by $v=(\tau+1)b$, with exceptional divisor $E_4: \tau+1=0$ and new coordinate $b$. For all but the $5$'th section, the proper transforms of the sections have the same equations (simply substitute $v=(\tau+1)b$). The proper transform of the $5$'th section has equation: 
$$5:\quad (2-\tau)X+\big(Y-2(\tau-1)X\big)b=0.$$

Along $E_4: \tau+1=0$, we have:
\begin{align*}
p=7=8 &:\quad Y=X\\
1=4=u &: \quad X=0,\\
2=9=v &:\quad Y=0,\\
3 &:\quad Y=2X,\\
5 &:\quad Yb+X(3-b)=0,\\
6 &:\quad Y=-X.
\end{align*}

It follows that $E_4$ has numerical class:
$$E_4=\De_{1,4,5,u}+\De_{2,5,9,v}+\De_{5,7,8,p}-\De_{1,4,u}-\De_{2,9,v}-\De_{7,8,p}+
\De_{3,5}+\De_{5,6}.$$

\begin{rmk}\label{K2}
In the notations of Section \ref{break of 2-conics curve} we have:
$$(K+\De)\cdot C=28,\quad (K+\De)\cdot F'=14,\quad  (K+\De)\cdot F''=2,$$
$$(K+\De)\cdot E_1=3,\quad (K+\De)\cdot E_2=3,\quad (K+\De)\cdot E_3=2,$$ 
$$(K+\De)\cdot E_4=3, \quad (K+\De)\cdot E_5=1.$$

Note that $K\cdot C=6$, and thus the lower bound for the dimension of the Hom scheme $\Hom(\PP^1,\MM_{0,12})$ at $[C]$ is $0$, and as in Rmk. \ref{K1} we note that $C$ is rigid, but not by a large margin. Similarly, the components of the characteristic $5$ fiber are not rigid. 
\end{rmk}


\section{Arithmetic break of a ``Two Conics'' curve - part II}\label{tangent}

We give a different description of the curve $F'$. As of now, the curve $F'$ is coming from a curve in $\MM_{0,9}$, and although we know its class, it is less clear how  it decomposes as a sum of $F$-curves. We note that the curve  $F'$ is  the {\bf irreducible} fiber in characteristic $5$ of a different family, this one over $\Spec(\ZZ)$.  We will prove that this new family breaks in characteristic $3$ into several components, all of which can be written as sums of $F$-curves. 

\

\Trick{Set-up.\label{set-up tangent}}
Consider a configuration similar to the one in Section \ref{break of 2-conics curve}, but one in which we drop the lines $L_5, L_7, L_8$ and impose that the conics $C_1$ and $C_2$ are tangent at $d$ (see Fig.~\ref{asfsdfd}). Namely, consider the following configuration of nine points:
$$a'=(1,0,0),  \quad b'=(0,1,0), \quad c'=(0,0,1), $$
$$d'=(1,1,1),  \quad e'=(3,2,1),\quad  f'=(1,1,0),$$
$$g'=(0,2,1), \quad  h'=(1,0,1), \quad i'=(1,-1,0). $$
(Only in characteristic $5$ this is the same as the previous configuration!) We have:

\begin{figure}[htbp]
\includegraphics[width=5in]{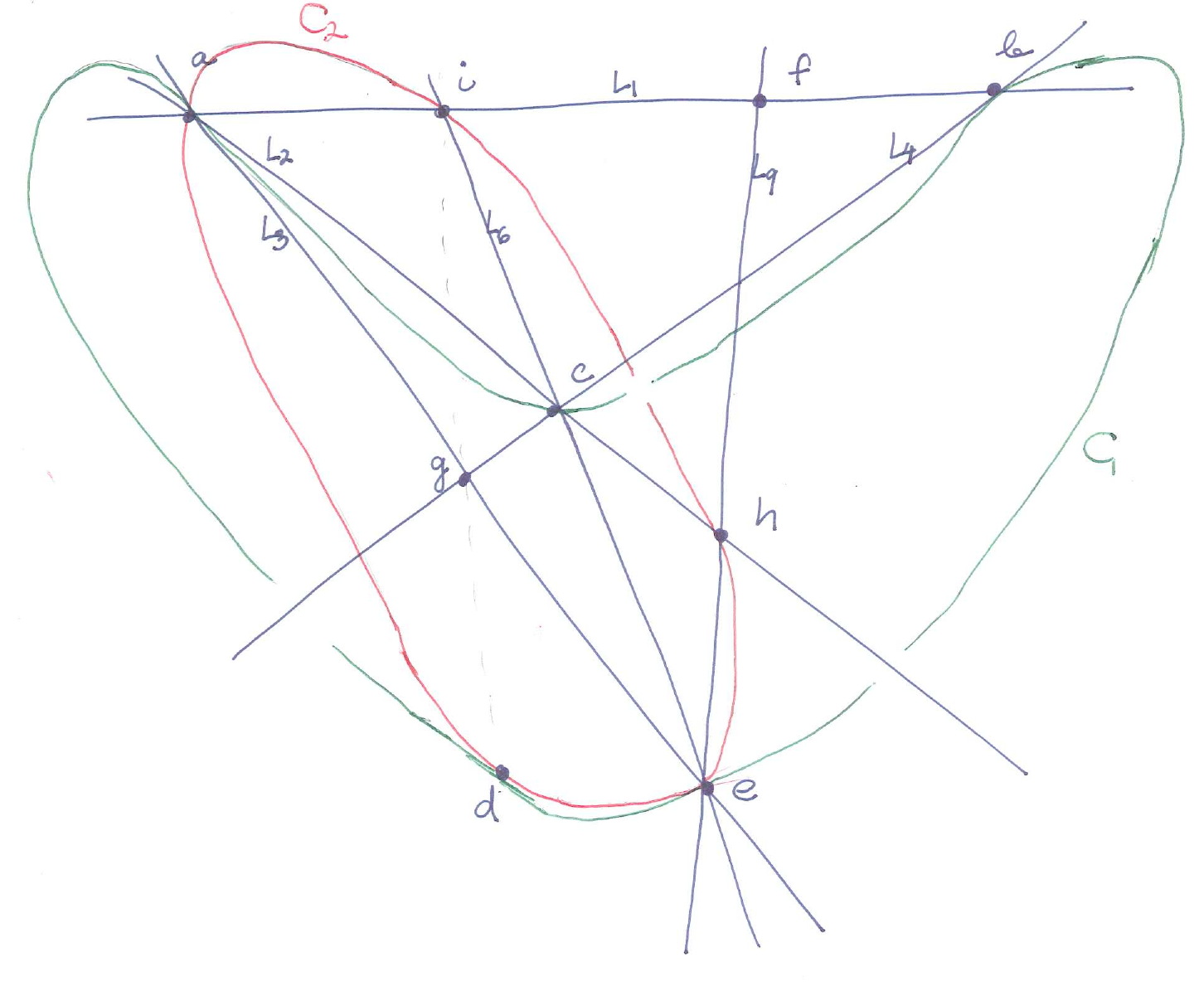}
\caption{New configuration}\label{asfsdfd}
\end{figure}
\begin{align*}
L'_1=a'b'i'f': &\quad Z=0,\\
L'_2=a'c'h': & \quad Y=0,\\ 
L'_3=a'g'e': &\quad Y=2Z,\\
L'_4=b'c'g': &\quad X=0,\\ 
L'_6=c'e'i': & \quad Y=-X,\\ 
L'_9=e'f'h': &\quad \frac{3}{2}Y=Z-X\\
C'_1=a'b'c'd'e':& \quad 2XY-3XZ+YZ=0,\\
C'_2=a'd'e'h'i': &\quad \frac{3}{4}Y^2+Z^2+ \frac{3}{4}XY-XZ-\frac{3}{2}YZ=0.
\end{align*}

Note that this configuration of lines and (tangent at $d$) conics is now rigid. Using the pencil of lines through the point $d$, we obtain as before a curve in $\M_{0,9}$. 
More precisely, let $S'$ be the blow-up $\PP^2_{\ZZ}$ at $d$ and let $E_x$ be the exceptional divisor. There are nine sections of $S'\ra\PP^1_{\ZZ}$ given by the proper transforms of the lines and conics, as well as the exceptional divisor $E_x$.  This induces a rational map: 
$$\PP^1_{\ZZ}\dra\M_{0,9}=\M_{0,\{1,2,3,4,6,9,u,v,x\}}$$
with $F'$ being the image of the morphism $\PP^1_{\QQ}\ra\M_{0,9}$.

\Trick{Breaking in characteristic $3$ (outline).}
We work on $\M_{0,9}=\M_{0,\{1,2,3,4,6,9,u,v,x\}}$. This is similar to the arguments in 
Section \ref{break of hypergraph curve} and Section \ref{break of 2-conics curve}. 
Consider the induced rational map: $$\PP^1_{\ZZ}\dra\M_{0,9}.$$

In order to resolve this map, one has to blow-up the arithmetic surface $\PP^1_{\ZZ}$ several times along the characteristic $3$ fiber $\PP^1_{\FF_3}$ of $\PP^1_{\ZZ}\ra\Spec \ZZ$. 
We first blow-up $\PP^1_{\ZZ}$ at one point in $\PP^1_{\FF_3}$, resulting in an exceptional divisor $E_1$. 
\begin{notn}
Let $G$ denote the proper transform of the characteristic $3$ fiber $\PP^1_{\FF_3}$. 
\end{notn}

Next, we blow-up the intersection point of $G$ and $E_1$, resulting in an exceptional divisor $E_2$. We blow-up another  point in $E_1$ and we let $E_3$ denote the corresponding exceptional divisor (see Fig. \ref{pic_h}).
\begin{figure}[htbp]
\includegraphics[width=4in]{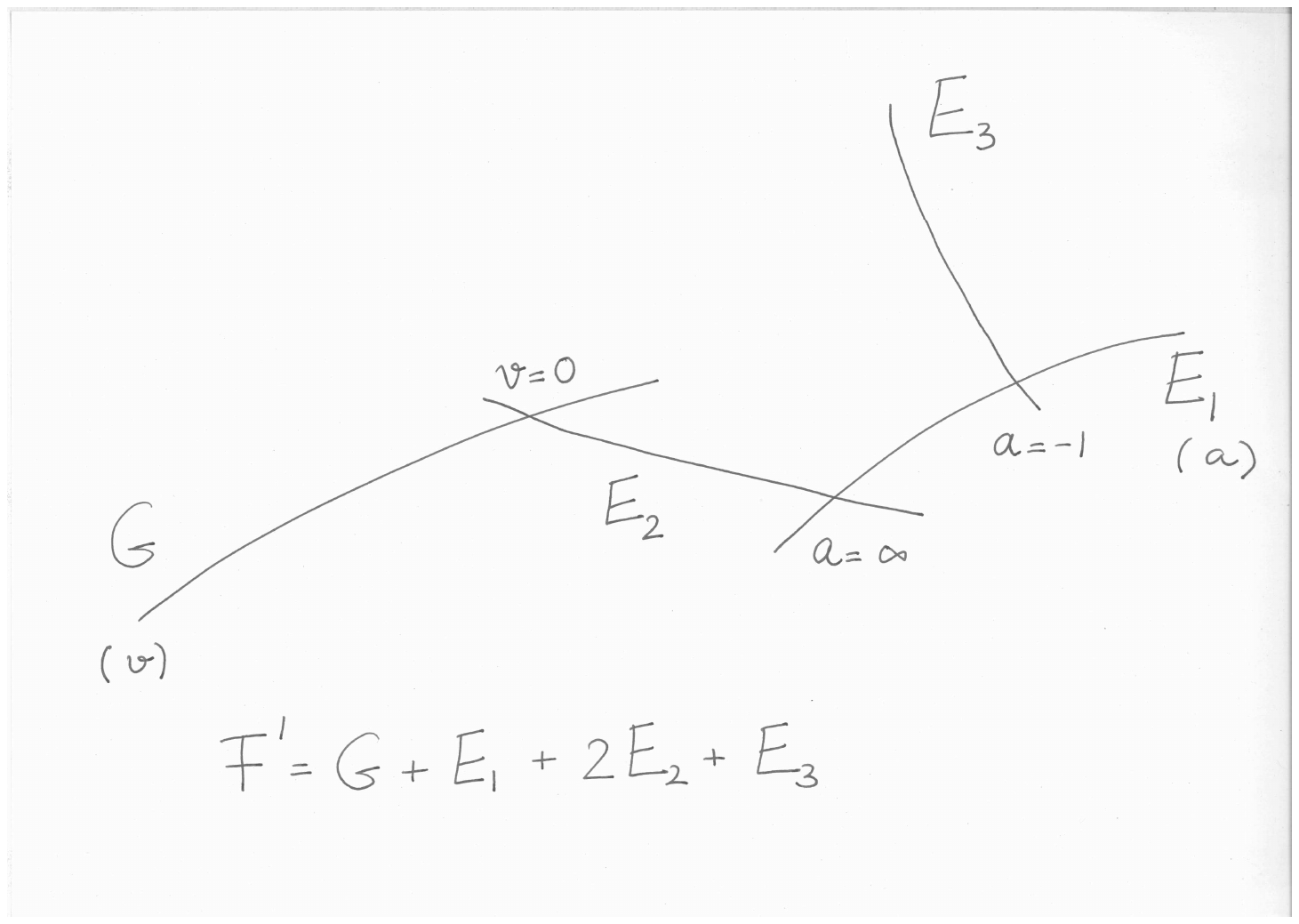}
\caption{The components of the characteristic $3$ fiber.}\label{pic_h}
\end{figure}

We let $T'$ be the resulting arithmetic surface. We abuse notations and denote by $E_1$ the proper transform of $E_1$ in $T'$.
Just as in Section \ref{break of 2-conics curve}, 
from the $\PP^1$-bundle  $S'\ra\PP^1_{\ZZ}$, we construct a family over $T'$ with nine sections, such that over a dense open set ${T'}^0$, this gives the universal family. 
Moreover, ${T'}^0$ intersects non-emptily each of curves $G$, $E_i$. Therefore, one has morphisms: $$G\ra\MM_{0,9},\quad E_i\ra\MM_{0,9},$$
As before, one can determine the classes of $G$, $E_i$ and check directly that:
$$F'=G+E_1+ 2E_2+ E_3.$$

This proves that any other extra components in the characteristic $3$ fiber will map constantly to $\MM_{0,9}$. Note that the exceptional divisor $E_2$ appears in this fiber with multiplicity $2$, since we blow-up a node of the fiber. It is easy to see that each of the curves $G$, $E_i$ is a sum of  $F$-curves. 

\Trick{Local coordinates on $S'$.\label{local coord on S'}} 
Recall that $S'$ is the blow-up of $\PP^2_{\ZZ}$ at $d=(1,1,1)$. This is an arithmetic threefold in $\PP^2_{\ZZ}\times \PP^1_{\ZZ}$ with local equation in $\PP^2_{\ZZ}\times \AA^1_{\ZZ}$ given by:
$$Z=X+\big(Y-X\big)v.$$
(Here $X,Y,Z$ are the coordinates on $\PP^2_{\ZZ}$ and $v$ is the coordinate on $\AA^1_{\ZZ}$.) 
The exceptional divisor $E_x$ is cut by $Y=X$.
By substituting $Z$ in the equations (\ref{set-up tangent}), we obtain equations for the proper transforms of the nine sections:
\begin{align*}
x &:\quad Y=X,\\
1 &: \quad X(1-v)+Yv=0,\\
2 &:\quad Y=0,\\
3 &:\quad Y(1-2v)+2X(v-1)=0,\\
4 &:\quad X=0,\\
6 &:\quad  Y=-X,\\
9 &:\quad  Y(2v-3)-2Xv=0,\\
u &:\quad Yv+3X(1-v)=0,\\
v &:\quad Y(4v^2-6v+3)+4X(v-v^2)=0.
\end{align*}

\Trick{The class of $G$.} By passing to characteristic $3$ in (\ref{local coord on S'}), it follows that $G$ is contained in the boundary components $\de_{1v}$, $\de_{2u}$, $\de_{9x}$ and as a a curve in $\M_{0,6}$ is described by:
\begin{align*}
x=9 &:\quad Y=1,\\
1=v &: \quad Y=\frac{v-1}{v},\\
2=u &:\quad Y=0,\\
3 &:\quad  Y=\frac{v-1}{v+1},\\
4 &:\quad Y=\infty,\\
6 &:\quad  Y=-1
\end{align*}

The class of $G$ in $\M_{0,9}$ can be computed to be: 
$$G=\De_{2,4,6,u}+\De_{1,2,3,u,v}+\De_{1,6,v}+\De_{1,4,v}+\De_{3,4}+\De_{3,6}
-2\De_{1,v}-\De_{2,u}-\De_{9,x}.$$


\Trick{Class of $E_1$ (see Fig.~\ref{pic_i}).\label{chart a}}
In the notations of (\ref{local coord on S'}), we blow-up $\PP^1_{\ZZ}$  along $3=0, v=0$. In local coordinates,  we have $v=3a$, with exceptional divisor $E_1: 3=0$ and new coordinate $a$.  The proper transforms of the nine sections have equations:
\begin{align*}
x &:\quad Y=X,\\
1 &: \quad X(1-3a)+3Ya=0,\\
2 &:\quad Y=0,\\
3 &:\quad Y(1-6a)+2X(3a-1)=0,\\
4 &:\quad X=0,\\
6 &:\quad  Y=-X,\\
9 &:\quad  Y(2a-1)-2Xa=0,\\
u &:\quad Ya+X(1-3a)=0,\\
v &:\quad Y(12a^2-6a+1)+4X(a-3a^2)=0.
\end{align*}
\begin{figure}[htbp]
\includegraphics[width=4in]{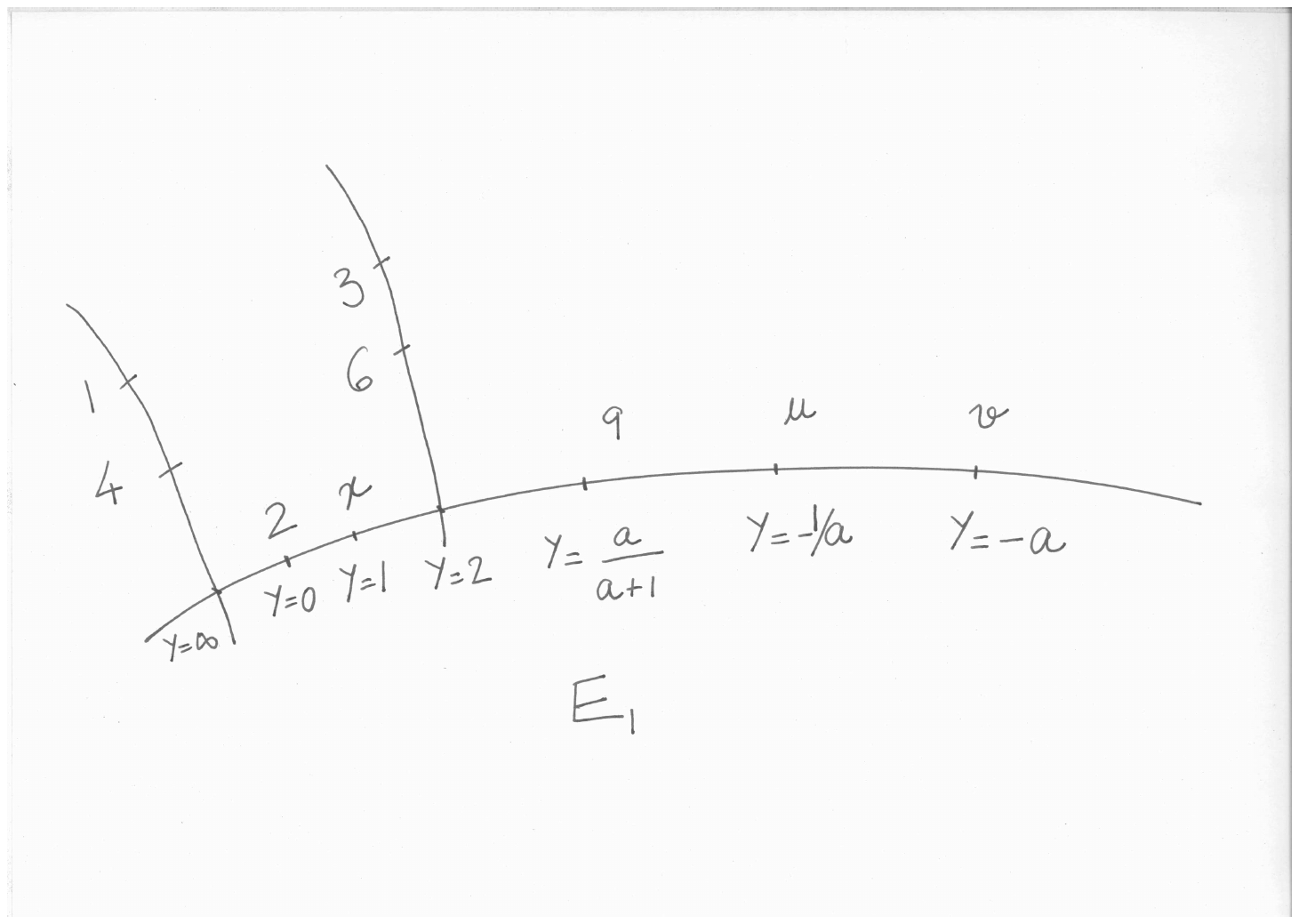}
\caption{The component $E_1$ of the characteristic $3$ fiber.}\label{pic_i}
\end{figure}

By passing to characteristic $3$ in the above equations, we obtain that $E_1$ is contained in the boundary components $\de_{1,4}$ and $\de_{3,6}$ and thus comes from a  curve in $\MM_{0,7}$ (thus a sum of $F$-curves by Cor. \ref{F-conj}). As a curve in $\M_{0,9}$, we have:
\begin{align*}
E_1&=\De_{3,6,9,u,v}+\De_{1,4,9}+\De_{1,4,u}+\De_{1,4,v}+\De_{2,9,v}+\\
&\De_{u,v,x}+\De_{2,u}+\De_{9,u}+\De_{9,x}-3\De_{1,4}-\De_{3,6}.
\end{align*}

\Trick{The class of $E_2$ (see Fig.~\ref{pic_j}).} We will blow-up the intersection point of $G$ and $E_1$. For this it is necessary to look in the other chart of the first blow-up, given by $3=vs$ (with $s=\frac{1}{a}$ the new coordinate on $E_1$). In this chart  we have $E_1: v=0$, $G: s=0$. The proper transforms of the nine sections have equations:
\begin{align*}
x &:\quad Y=X,\\
1 &: \quad X(1-v)+Yv=0,\\
2 &:\quad Y=0,\\
3 &:\quad Y(1-2v)+2X(v-1)=0,\\
4 &:\quad X=0,\\
6 &:\quad  Y=-X,\\
9 &:\quad  Y(2-s)-2X=0,\\
u &:\quad Y+Xs(1-v)=0,\\
v &:\quad Y(4v-6+s)+4X(1-v)=0.
\end{align*}

We blow-up $\PP^1_{\ZZ}$ at $v=s=0$. In local coordinates, we have $v=sw$, with exceptional divisor $E_2: s=0$ and new coordinate $w$ (and thus $3=vs=s^2w$). 
The proper transforms of the nine sections have equations:
\begin{align*}
x &:\quad Y=X,\\
1 &: \quad X(1-sw)+Ysw=0,\\
2 &:\quad Y=0,\\
3 &:\quad Y(1-2sw)+2X(sw-1)=0,\\
4 &:\quad X=0,\\
6 &:\quad  Y=-X,\\
9 &:\quad  Y(2-s)-2X=0,\\
u &:\quad Y+Xs(1-sw)=0,\\
v &:\quad Y(4sw-6+s)+4X(1-sw)=0.
\end{align*}

Along $E_2 (s=0)$ the sections become:
\begin{align*}
x=9 &:\quad Y=X,\\
1=4=v &: \quad X=0,\\
2=u &:\quad Y=0,\\
3=6 &:\quad Y=-X.
\end{align*}
\begin{figure}[htbp]
\includegraphics[width=4in]{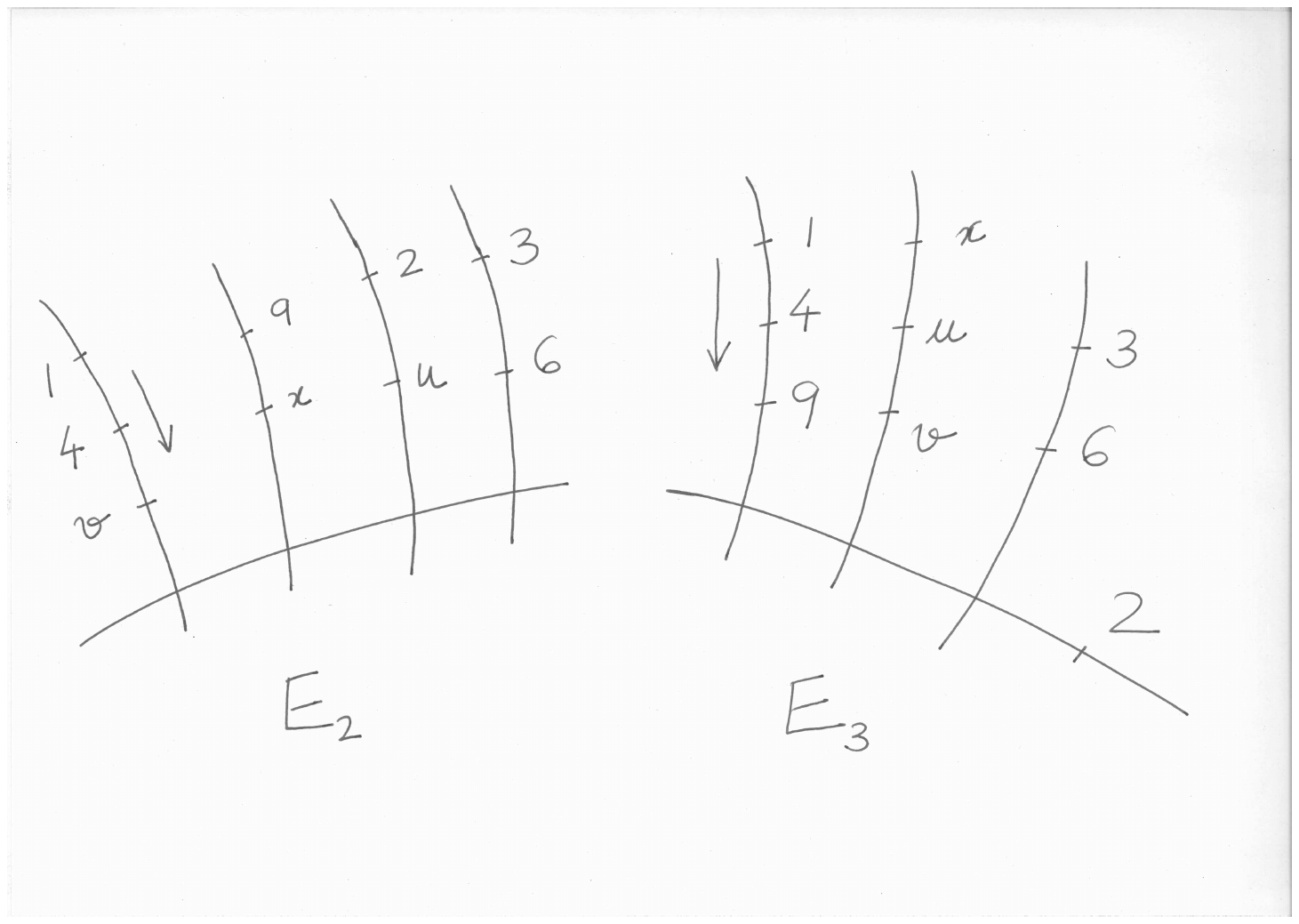}
\caption{The components $E_2$ and $E_3$ of the characteristic $3$ fiber.}\label{pic_j}
\end{figure}

Blowing up the total space along the locus $1=4=v$ ($s=X=0$), we obtain an arithmetic threefold with local equation $X=sX_1$, exceptional divisor $s=0$ and new coordinate $X_1$. The proper transforms of the sections $1, 4, v$ are given by:
\begin{align*}
1 &: \quad X_1(1-sw)+Yw=0,\\
4 &:\quad X_1=0,\\
v &:\quad Y(4w-2sw+1)+4X_1(1-sw)=0.
\end{align*}

Along $E_2$ ($s=0$) these sections become:
\begin{align*}
1 &: \quad X_1+Yw=0,\\
4 &:\quad X_1=0,\\
v &:\quad Y(w+1)+X_1=0.
\end{align*}

The ``attaching section" is cut by $Y=0$. It follows that $E_2$ has class an $F$-curve:
$$E_2=-\De_{1,4,v}+\De_{1,4}+\De_{1,v}+\De_{4,v}.$$

\Trick{The class of $E_3$ (see Fig.~\ref{pic_j}).} In the notations of (\ref{chart a}), we blow up the point $a=-1$ on $E_1$. In local coordinates, we have $a=3b-1$, with exceptional divisor $E_3: 3=0$ and  new coordinate $b$. The proper trasnforms of the nine sections have equations:
\begin{align*}
x &:\quad Y=X,\\
1 &: \quad X(4-9b)+3Y(3b-1)=0,\\
2 &:\quad Y=0,\\
3 &:\quad Y(7-18a)+2X(9b-4)=0,\\
4 &:\quad X=0,\\
6 &:\quad  Y=-X,\\
9 &:\quad  Y(6b-3)-2X(3b-1)=0,\\
u &:\quad Y(3b-1)+X(4-9b)=0,\\
v &:\quad Y\big(12(3b-1)^2-6(3b-1)+1\big)+4X(3b-1)(4-9b)=0.
\end{align*}

Along $E_3$ ($3=0$) the sections become:
\begin{align*}
x=u=v &:\quad Y=X,\\
1=4=9 &: \quad X=0,\\
2 &:\quad Y=0,\\
3=6 &:\quad Y=-X.
\end{align*}

We blow up the total space along the locus $1=4=9$ ($3=X=0$). The new arithmetic threefold is locally cut by $X=3X_1$, with exceptional divisor $3=0$ and  new coordinate $X_1$. The proper transforms of the sections $1, 4, 9$ are given by:
\begin{align*}
1 &: \quad X_1(4-9b)+Y(3b-1)=0,\\
4 &:\quad X_1=0,\\
9 &:\quad  Y(2b-1)-2X_1(3b-1)=0.
\end{align*}
The ``attaching section" is $Y=0$. Along $E_3$ ($3=0$) the sections become:
\begin{align*}
1 &: \quad Y=X_1,\\
4 &:\quad X_1=0,\\
9 &:\quad  Y(b+1)+X_1=0.
\end{align*}

It follows that $E_3$ has the same class as an $F$-curve:
$$E_3=-\De_{1,4,9}+\De_{1,4}+\De_{1,9}+\De_{4,9}.$$







\end{document}